\documentclass[final]{siamart0216}
\usepackage{amsfonts}
\usepackage{todonotes}
\usepackage{amsfonts,amsmath,amssymb}
\usepackage{mathrsfs,mathtools,stmaryrd,wasysym}
\usepackage{enumerate}
\usepackage{xspace,mydef} 
\usepackage{esint}
\usepackage{graphicx,psfrag}
\usepackage{hyperref}
\usepackage{pgf,tikz,pgfplots}
\usepackage{pstricks-add}
\usepackage{enumitem}
\usepackage{yfonts}

\usetikzlibrary{arrows}
\usepackage{rotating}
\DeclareMathAlphabet{\mathpzc}{OT1}{pzc}{m}{it}

\usepackage[notcite,notref]{showkeys} % To see crossreferences.
\newsiamremark{remark}{Remark}

\newcommand{\norm}[1]{{\left\vert\kern-0.25ex\left\vert\kern-0.25ex\left\vert #1 
\right\vert\kern-0.25ex\right\vert\kern-0.25ex\right\vert}}

\newcommand{\TheTitle}{
Darcy's problem coupled with the heat equation under singular forcing: analysis and discretization}
\newcommand{\ShortTitle}{Darcy's problem coupled
with the heat equation}
\newcommand{\TheAuthors}{A. Allendes, G. Campa\~na, F. Fuica and E.~Ot\'arola}

\headers{\ShortTitle}{\TheAuthors}

\title{{\TheTitle}\thanks{AA is partially supported by ANID through FONDECYT project 1170579. GC was supported by ANID through Beca doctorado nacional 21200920. FF is supported by UTFSM through Beca de Mantenci\'on. EO is partially supported by ANID through FONDECYT project 11180193.}}

\author{Alejandro Allendes\thanks{Departamento de Matem\'atica, Universidad T\'ecnica Federico Santa Mar\'ia, Valpara\'iso, Chile.
(\email{alejandro.allendes@usm.cl}, \url{http://aallendes.mat.utfsm.cl/}).}
\and 
Gilberto Campa\~na\thanks{Departamento de Ciencias, Universidad T\'ecnica Federico Santa Mar\'ia, Valpara\'iso, Chile.
(\email{gilberto.campana@usm.cl}.}
\and
Francisco Fuica\thanks{Departamento de Matem\'atica, Universidad T\'ecnica Federico Santa Mar\'ia, Valpara\'iso, Chile.
(\email{francisco.fuica@sansano.usm.cl}.}  
\and
Enrique Ot\'arola\thanks{Departamento de Matem\'atica, Universidad T\'ecnica Federico Santa Mar\'ia, Valpara\'iso, Chile.
(\email{enrique.otarola@usm.cl}, \url{http://eotarola.mat.utfsm.cl/}).}   
}

\ifpdf
\hypersetup{
  pdftitle={\TheTitle},
  pdfauthor={\TheAuthors}
}
\fi

\date{Draft version of \today.}

\begin{document}

\maketitle

\begin{abstract}
We study the existence of solutions for Darcy's problem coupled with the heat equation under singular forcing; the right-hand side of the heat equation corresponds to a Dirac measure. The studied model allows thermal diffusion and viscosity depending on the temperature. We propose a finite element solution technique and analyze its convergence properties. In the case that the thermal diffusion is constant, we propose an a posteriori error estimator and investigate reliability and efficiency properties. We illustrate the theory with numerical examples.
\end{abstract}


% REQUIRED

%%%%%%%%%%%%%%%%%%%%%%%%%%%%%%%%%%%%%%%%%%%%%%%%%%%%%%%%%%%%%%%%
%%%%%%%%%%%%%%%%%%%%%%%%%%%%%%%%%%%%%%%%%%%%%%%%%%%%%%%%%%%%%%%%
%%%%%%%%%%%%%%%%%%%%%%%%%%%%%%%%%%%%%%%%%%%%%%%%%%%%%%%%%%%%%%%%
%%%%%%%%%%%%%%%%%%%%%%%%%%%%%%%%%%%%%%%%%%%%%%%%%%%%%%%%%%%%%%%%
%%%%%%%%%%%%%%%%%%%%%%%%%%%%%%%%%%%%%%%%%%%%%%%%%%%%%%%%%%%%%%%%
%%%%%%%%%%%%%%%%%%%%%%%%%%%%%%%%%%%%%%%%%%%%%%%%%%%%%%%%%%%%%%%%
%%%%%%%%%%%%%%%%%%%%%%%%%%%%%%%%%%%%%%%%%%%%%%%%%%%%%%%%%%%%%%%%
%%%%%%%%%%%%%%%%%%%%%%%%%%%%%%%%%%%%%%%%%%%%%%%%%%%%%%%%%%%%%%%%

\section{Introduction}

\label{sec:intro}

In this work we are interested in the analysis and discretization of the temperature distribution of a fluid in a porous medium modelled by a convection--diffusion equation coupled with Darcy's law. To make matters precise, we let $\Omega\subset\mathbb{R}^2$ be an open and bounded domain with Lipschitz boundary $\partial\Omega$. We are interested in the analysis and discretization of the following system of partial differential equations (PDEs) in its strong form:
\begin{equation}\label{eq:model}
\left\{
\begin{array}{rcll}
\nu(T) \mathbf{u} + \nabla \mathsf{p} & = & \mathbf{f} & \text{ in }\Omega,\\
\text{div}~\mathbf{u} & = & 0 & \text{ in }\Omega, \\
-\text{div}(\kappa(T)\nabla T)+\text{div}(\mathbf{u}~T) & = & \textnormal{g} & \text{ in }\Omega,\\
\mathbf{u}\cdot \mathbf{n} & = & 0 & \text{ on }\partial\Omega,\\
T & = & 0 & \text{ on }\partial\Omega.
\end{array}
\right.
\end{equation}
The unknowns are the velocity field $\mathbf{u}$, the pressure $\mathsf{p}$, and the temperature $T$ of the fluid, respectively. The data are the viscosity coefficient $\nu$, the thermal diffusivity coefficient $\kappa$, the external density force $\mathbf{f}$, and the external heat source $\textnormal{g}$. The viscosity and thermal diffusivity coefficients may depend nonlineary on the temperature $T$. In \eqref{eq:model}, $\mathbf{n}$ denotes the unit outward normal vector on $\partial\Omega$. In this work we are particularly interested in the case that $\textnormal{g}=\delta_z$, where $\delta_z$ corresponds to the Dirac delta distribution supported at the interior point $z\in\Omega$.

The analysis and discretization of the heat equation coupled with Darcy's law by a nonlinear viscosity depending on the temperature have been studied in a number of works. To the best of our knowledge, the first article that considers such a problem is \cite{MR3802674}. In this work the authors derive existence of solutions, without restriction on the data, by Galerkin's method and Brouwer's fixed point theorem \cite[Theorem 2.3]{MR3802674}; uniqueness is established when the data are suitably restricted \cite[Theorem 2.6]{MR3802674}. In addition, the authors of \cite{MR3802674} propose and analyze two numerical schemes based on finite element methods and derive optimal a priori error estimates. Recently, the results of \cite{MR3802674} have been complemented and extended in \cite{MR4036533}, where the authors introduce a new non-stabilized method and prove, for a sufficiently small mesh-size, existence and uniqueness of a solution; a priori error estimates are also derived. Later, in \cite{MR4041519}, the authors devise and analyze a posteriori error estimators for the two numerical schemes considered in \cite{MR3802674}. In the recent work \cite{GMRB}, the authors analyze a new fully--mixed finite element method based on the introduction of the pseudoheat flux as a further unknown. The authors prove the unique solvability of the underlying continuous formulation, present a discrete formulation, and derive a priori error estimates. We conclude this paragraph by mentioning the work \cite{MR3523581}, where a different coupling of Darcy's system with the heat equation is analyzed: the viscosity $\nu$ is constant but the exterior force $\mathbf{f}$ depends on the temperature. In this work, the authors provide existence and uniqueness results and analyze a spectral discretization.

When, in system \eqref{eq:model}, with smooth forcing, the Darcy's system is replaced by the stationary Navier--Stokes equations, we arrive at the classical and generalized steady state Boussinesq problems \cite{MR1370148,MR1047471}. These problems, which are particular instances of an incompressible nonisothermal fluid flow model, have been extensively studied over the last decades; it is thus no surprise that their analysis and approximation, at least in energy--type spaces, are very well developed. For a variety of finite element solution techniques used to discretize the classical and generalized steady state Boussinesq problems, we refer the interested reader to the following nonextensive list of references: \cite{MR1051838,MR1364404,MR1681112,MR2111747,MR2802085,MR3267163,MR3232446,MR3454217,MR3918673,MR3845205,MR4053090,MR4265062,MR4080228}; see also the references therein. To conclude this paragraph, we mention the work \cite{MR4265062}, where the authors study, on the basis of weighted estimates and weighted Sobolev spaces, existence and approximation results for a Boussinesq model of thermally driven convection under singular forcing; a posteriori error estimates are also analyzed.

To best of our knowledge, this is the first work that analyzes problem \eqref{eq:model} with singular data. Our main source of difficulty and interest here is that the external heat source is rough or singular. As a result \emph{standard energy arguments do not apply}; the fluid velocity and the temperature lie in different spaces. In addition, the temperature $T$ exhibits reduced regularity properties: $T \in W_0^{1,p}(\Omega) \setminus H_0^1(\Omega)$, with $p<2$. This and the fact that the velocity component of a solution to the Darcy's problem has very low regularity, namely, $\mathbf{u}\in \mathbf{H}_0(\textnormal{div},\Omega)$, complicate both the analysis of the continuous problem and the study of discretization techniques. Regarding discretization, we devise suitable adaptive finite element methods (AFEMs) to solve \eqref{eq:model}. These techniques are motivated by the fact that $T$ exhibits reduced regularity properties. In what follows we list what, we believe, are the main contributions of our work:

\begin{itemize}
\item \emph{Existence of solutions:} We introduce a concept of weak solution within the space $\mathbf{H}_0(\textnormal{div},\Omega)\times L_0^2(\Omega)\times W_0^{1,p}(\Omega)$, with $p<2$, and show, on the basis of a fixed point argument, the existence of solutions; see Theorem \ref{thm:existence}.

\item \emph{Discretization:} We discretize the coupled system \eqref{eq:model} by using the Raviart--Thomas finite element space of order zero, piecewise constant finite elements, and continuous piecewise linear finite elements for the velocity, the pressure, and the temperature, respectively. Under suitable assumptions on data we prove, in Theorem \ref{thm:existence_discrete}, the existence of discrete solutions and, in Theorem \ref{thm:convergence}, the existence of a subsequence that weakly converges to a solution of the continuous problem.

\item \emph{A posteriori error estimates:} We devise a residual--based a posteriori error estimator for the proposed finite element discretization of system \eqref{eq:model} that can be decomposed as the sum of three individual contributions: one contribution that accounts for the discretization of the heat equation and two contributions related to the discretization of the Darcy's system. We prove, in Theorem \ref{thm:global_reliability}, that the devised error estimator is globally reliable. We explore local efficiency estimates in Section \ref{sec:efficiency_estimates}.
\end{itemize}

The rest of the manuscript is organized as follows. We set notation and collect background information in Section \ref{sec:notation}. In Section \ref{sec:coupled_problem}, we introduce a notion of weak solution for problem \eqref{eq:model} and analyze the existence of solutions. A numerical discretization technique for problem \eqref{eq:model} is proposed in Section \ref{sec:fem}, where we also analyze convergence properties of discretizations. In Section \ref{sec:a_posteriori_anal}, we design and analyze an a posteriori error estimator for the proposed finite element scheme. We derive global reliability properties and explore local efficiency estimates. Finally, a series of numerical experiments are presented in Section \ref{sec:numericalexperiments}, which illustrate the theory and reveal a competitive performance of AFEMs based on the devised a posteriori error estimator.

%%%%%%%%%%%%%%%%%%%%%%%%%%%%%%%%%%%%%%%%%%%%%%%%%%%%%%%%%%%%%%%%
%%%%%%%%%%%%%%%%%%%%%%%%%%%%%%%%%%%%%%%%%%%%%%%%%%%%%%%%%%%%%%%%
%%%%%%%%%%%%%%%%%%%%%%%%%%%%%%%%%%%%%%%%%%%%%%%%%%%%%%%%%%%%%%%%
%%%%%%%%%%%%%%%%%%%%%%%%%%%%%%%%%%%%%%%%%%%%%%%%%%%%%%%%%%%%%%%%
%%%%%%%%%%%%%%%%%%%%%%%%%%%%%%%%%%%%%%%%%%%%%%%%%%%%%%%%%%%%%%%%
%%%%%%%%%%%%%%%%%%%%%%%%%%%%%%%%%%%%%%%%%%%%%%%%%%%%%%%%%%%%%%%%
%%%%%%%%%%%%%%%%%%%%%%%%%%%%%%%%%%%%%%%%%%%%%%%%%%%%%%%%%%%%%%%%
%%%%%%%%%%%%%%%%%%%%%%%%%%%%%%%%%%%%%%%%%%%%%%%%%%%%%%%%%%%%%%%%

\section{Notation and preliminaries} 
\label{sec:notation}
Let us set notation and describe the setting we shall operate with.

%%%%%%%%%%%%%%%%%%%%%%%%%%%%%%%%%%%%%%%%%%%%%%%%%%%%%%%%%%%%%%%%
%%%%%%%%%%%%%%%%%%%%%%%%%%%%%%%%%%%%%%%%%%%%%%%%%%%%%%%%%%%%%%%%
%%%%%%%%%%%%%%%%%%%%%%%%%%%%%%%%%%%%%%%%%%%%%%%%%%%%%%%%%%%%%%%%
%%%%%%%%%%%%%%%%%%%%%%%%%%%%%%%%%%%%%%%%%%%%%%%%%%%%%%%%%%%%%%%%

\subsection{Notation}
Let $d \in \{1,2\}$ and $\mathcal{O} \subset \mathbb{R}^d$ be an open and bounded domain. We shall use standard notation for Lebesgue and Sobolev spaces. The space of functions in $L^2(\mathcal{O})$ that have zero average is denoted by $L^2_0(\mathcal{O})$. By $W^{m,r}(\mathcal{O})$, we denote the Sobolev space of functions in $L^r(\mathcal{O})$ with partial derivatives of order up to $m$ in $L^r(\mathcal{O})$; $m$ denotes a positive integer and $1 \leq r \leq \infty$. We denote by $W_0^{m,r}(\mathcal{O})$ the closure with respect to the norm in $W^{m,r}(\mathcal{O})$ of the space of $C^{\infty}$ functions compactly supported in $\mathcal{O}$. 
% When $r = 2$, we set $H^{m}(\mathcal{O}) := W^{m,2}(\mathcal{O})$ and $H_0^{m}(\mathcal{O}) := W_0^{m,2}(\mathcal{O})$. 
We use uppercase bold letters to denote the vector-valued counterparts of the aforementioned spaces whereas lowercase bold letters are used to denote vector-valued functions.

Let us introduce some spaces utilized in the analysis of Darcy's problem:
\begin{equation*}
\mathbf{H}(\textnormal{div},\mathcal{O}):=\{ \mathbf{v}\in \mathbf{L}^2(\mathcal{O})  :  \text{div }\mathbf{v}\in L^2(\mathcal{O})\}
\end{equation*}
and $\mathbf{H}_0(\text{div},\mathcal{O}) := \left \{ \mathbf{v} \in \mathbf{H}(\textnormal{div},\mathcal{O}): \mathbf{v}  \cdot \mathbf{n}|_{\partial \mathcal{O}} = 0 \right \}$.  We equip both spaces, $\mathbf{H}(\textnormal{div},\mathcal{O})$ and $\mathbf{H}_0(\text{div},\mathcal{O})$, with the following norm:
\[
\|\mathbf{v}\|_{\mathbf{H}(\text{div},\mathcal{O})} := \left( \|\mathbf{v}\|_{\mathbf{L}^2(\mathcal{O})}^2+\|\text{div}~\mathbf{v}\|_{L^2(\mathcal{O})}^2 \right)^{\frac{1}{2}}.
\]
We also introduce $\mathbf{V}(\mathcal{O}) := \{  \mathbf{v} \in \mathbf{H}_0(\textnormal{div},\mathcal{O}): \text{div}~\mathbf{v} = 0  \}$. 

To perform an a posteriori error analysis, we will make use of the so-called curl operator. When $d=2$, we define, for $v\in H^1(\mathcal{O})$ and $ \mathbf{v}=(v_1,v_2)\in \mathbf{H}^1(\mathcal{O})$,
\begin{equation*}
\mathbf{curl} ~ v :=  \left(\frac{\partial v}{\partial x_2},-\frac{\partial v}{\partial x_1}\right), \qquad 
\mathbf{curl}~  \mathbf{v} :=  \frac{\partial v_2}{\partial x_1} -\frac{\partial v_1}{\partial x_2}.
\end{equation*}
With this operator at hand, we define $\mathbf{H}(\mathbf{curl},\Omega):= \left\{ \mathbf{v} \in L^2(\Omega)^2: \mathbf{curl}~  \mathbf{v}  \in L^2(\Omega) \right\}$.

If $\mathcal{W}$ and $\mathcal{Z}$ are Banach function spaces, we write $\mathcal{W} \hookrightarrow\mathcal{Z}$ to denote that $\mathcal{W}$ is continuously embedded in $\mathcal{Z}$. We denote by $\mathcal{W}'$ and $\|\cdot\|_{\mathcal{W}}$ the dual and the norm of $\mathcal{W}$, respectively. Given $p \in(1, \infty)$, we denote by $p'$ its H\"older conjugate, i.e., the real number such that $1/p + 1/p' = 1$. The relation $a\lesssim b$ indicates that $a\leq Cb$, with a constant $C$ that neither depends on $a$, $b$, nor the discretization parameters. The value of $C$ might change at each occurrence.  

We finally mention that, throughout this work $\Omega\subset \mathbb{R}^2$ is an open and bounded polygonal domain with Lipschitz boundary $\partial\Omega$.

%%%%%%%%%%%%%%%%%%%%%%%%%%%%%%%%%%%%%%%%%%%%%%%%%%%%%%%%%%%%%%%%
%%%%%%%%%%%%%%%%%%%%%%%%%%%%%%%%%%%%%%%%%%%%%%%%%%%%%%%%%%%%%%%%
%%%%%%%%%%%%%%%%%%%%%%%%%%%%%%%%%%%%%%%%%%%%%%%%%%%%%%%%%%%%%%%%
%%%%%%%%%%%%%%%%%%%%%%%%%%%%%%%%%%%%%%%%%%%%%%%%%%%%%%%%%%%%%%%%

\subsection{Darcy's equations}
We begin this section by recalling the fact that, on Lipschitz domains, the divergence operator is surjective from $\mathbf{H}_0(\text{div},\Omega)$ to $L_0^2(\Omega)$: there exists $\beta>0$ such that \cite[inequality (2.14)]{MR3802674}, \cite[inequality (2.13)]{MR4041519}
\begin{equation}\label{eq:infsup}
\inf_{\mathsf{q}\in L_0^2(\Omega)} \sup_{\mathbf{v}\in \mathbf{H}_0(\text{div},\Omega)}\dfrac{\int_\Omega \mathsf{q}\text{ div }\mathbf{v}\mathsf{d}\mathbf{x}}{\|\mathbf{v}\|_{\mathbf{H}(\text{div},\Omega)}\|\mathsf{q}\|_{L^2(\Omega)}}\geq \beta>0.
\end{equation}

We introduce the following weak formulation of standard Darcy's equations: Find $(\mathbf{u},\mathsf{p})\in \mathbf{H}_0(\textnormal{div},\Omega)\times L_0^2(\Omega)$ such that
\begin{equation}\label{eq:darcy_problem}
\begin{array}{rcll}
\displaystyle\int_\Omega(\eta \mathbf{u}\cdot \mathbf{v} - \mathsf{p}\textnormal{ div }\mathbf{v})\mathsf{d}\mathbf{x}& = & 	\displaystyle\int_\Omega\mathbf{f}\cdot \mathbf{v}\mathsf{d}\mathbf{x} \quad &\forall \mathbf{v}\in \mathbf{H}_0(\textnormal{div},\Omega),\\
\displaystyle\int_\Omega \mathsf{q}\textnormal{ div } \mathbf{u}\mathsf{d}\mathbf{x} & = & 0 \quad &\forall \mathsf{q}\in L_0^2(\Omega).
\end{array}
\end{equation}
Here, $\mathbf{f}\in \mathbf{L}^2(\Omega)$ and $\eta$ denotes a function in $L^\infty(\Omega)$ that satisfies
\begin{equation}\label{eq:viscosity}
\eta_{-}\leq \eta(\mathbf{x})\leq \eta_{+} \quad \text{for a.e.}~\mathbf{x}\in \Omega,
\qquad
\eta_{-},\eta_{+} > 0.
\end{equation}

The next result follows from the inf-sup theory for saddle point problems \cite[Theorem 2.34]{Guermond-Ern}.

\begin{theorem}[well-posedness of Darcy's equations]\label{thm:wp_darcy}
Let $\mathbf{f}\in \mathbf{L}^2(\Omega)$ and $\eta\in L^\infty(\Omega)$ be such that \eqref{eq:viscosity} holds. Then, problem \eqref{eq:darcy_problem} admits a unique solution $(\mathbf{u},\mathsf{p})\in \mathbf{H}_0(\textnormal{div},\Omega)\times L_0^2(\Omega)$. In addition, the following estimate holds:
\begin{equation*}
\|\mathbf{u}\|_{\mathbf{H}(\textnormal{div},\Omega)}+\|\mathsf{p}\|_{L^2(\Omega)}\lesssim \|\mathbf{f}\|_{\mathbf{L}^2(\Omega)},
\end{equation*}
where the hidden constant is independent of the data $\mathbf{f}$, $\eta$ and the solution $(\mathbf{u},\mathsf{p})$.
\end{theorem}

%%%%%%%%%%%%%%%%%%%%%%%%%%%%%%%%%%%%%%%%%%%%%%%%%%%%%%%%%%%%%%%%
%%%%%%%%%%%%%%%%%%%%%%%%%%%%%%%%%%%%%%%%%%%%%%%%%%%%%%%%%%%%%%%%
%%%%%%%%%%%%%%%%%%%%%%%%%%%%%%%%%%%%%%%%%%%%%%%%%%%%%%%%%%%%%%%%
%%%%%%%%%%%%%%%%%%%%%%%%%%%%%%%%%%%%%%%%%%%%%%%%%%%%%%%%%%%%%%%%
%%%%%%%%%%%%%%%%%%%%%%%%%%%%%%%%%%%%%%%%%%%%%%%%%%%%%%%%%%%%%%%%
%%%%%%%%%%%%%%%%%%%%%%%%%%%%%%%%%%%%%%%%%%%%%%%%%%%%%%%%%%%%%%%%
%%%%%%%%%%%%%%%%%%%%%%%%%%%%%%%%%%%%%%%%%%%%%%%%%%%%%%%%%%%%%%%%
%%%%%%%%%%%%%%%%%%%%%%%%%%%%%%%%%%%%%%%%%%%%%%%%%%%%%%%%%%%%%%%%

\section{The coupled problem}
\label{sec:coupled_problem} 
The main goal of this section is to show the existence of weak solutions for problem \eqref{eq:model}. As a first step, we introduce the set of assumptions under which we will operate and set a weak formulation.

%%%%%%%%%%%%%%%%%%%%%%%%%%%%%%%%%%%%%%%%%%%%%%%%%%%%%%%%%%%%%%%%
%%%%%%%%%%%%%%%%%%%%%%%%%%%%%%%%%%%%%%%%%%%%%%%%%%%%%%%%%%%%%%%%
%%%%%%%%%%%%%%%%%%%%%%%%%%%%%%%%%%%%%%%%%%%%%%%%%%%%%%%%%%%%%%%%
%%%%%%%%%%%%%%%%%%%%%%%%%%%%%%%%%%%%%%%%%%%%%%%%%%%%%%%%%%%%%%%%

\subsection{Main assumptions and weak formulation}\label{sec:main_assump}

We will operate under the following assumptions on the viscosity and diffusivity coefficients.

\begin{itemize}
\item \emph{Viscosity:} The viscosity $\nu$ is a function that is strictly positive and bounded, i.e., there exist positive constants $\nu_{-}$ and $\nu_{+}$ such that
\begin{equation}\label{eq:nu}
\nu_{-}\leq \nu(s)\leq \nu_{+} \quad \forall s\in\mathbb{R}.
\end{equation}
In addition, we assume that $\nu \in C^{0,1}(\mathbb{R})$ with Lipschitz constant $C_{\mathcal{L}}$, i.e.,
\begin{equation*}
|\nu(s_{1})-\nu(s_{2})| \leq C_{\mathcal{L}}|s_{1}-s_{2}| \quad \forall s_{1},s_{2}\in\mathbb{R}.
\end{equation*}
\item \emph{Diffusivity:} The thermal coefficient $\kappa$ is a strictly positive and bounded function, i.e., there exist positive constants $\kappa_{-}$ and $\kappa_{+}$ such that
\begin{equation}\label{eq:kappa}
\kappa_{-}\leq \kappa(s)\leq \kappa_{+} \quad \forall s\in\mathbb{R}.
\end{equation}
We also assume that $\kappa \in C^{0,1}(\mathbb{R})$.
\end{itemize}

%%%%%%%%%%%%%%%%%%%%%%%%%%%%%%%%%%%%%%%%%%%%%%%%%%%%%%%%%%%%%%%%
%%%%%%%%%%%%%%%%%%%%%%%%%%%%%%%%%%%%%%%%%%%%%%%%%%%%%%%%%%%%%%%%
%%%%%%%%%%%%%%%%%%%%%%%%%%%%%%%%%%%%%%%%%%%%%%%%%%%%%%%%%%%%%%%%
%%%%%%%%%%%%%%%%%%%%%%%%%%%%%%%%%%%%%%%%%%%%%%%%%%%%%%%%%%%%%%%%

\subsection{Weak solutions}
\label{sec:weak_solutions}

We adopt the following notion of weak solution.

\begin{definition}[weak solution]
Let $\mathbf{f}\in \mathbf{L}^2(\Omega)$, $z\in\Omega$, and $p<2$. We say that $(\mathbf{u},\mathsf{p},T)\in \mathbf{H}_0(\textnormal{div},\Omega)\times L_0^2(\Omega)\times W_0^{1,p}(\Omega)$ is a weak solution to \eqref{eq:model} if
\begin{equation}\label{eq:modelweak}
\left\{
\begin{array}{rcll}
\displaystyle\int_\Omega (\nu(T) \mathbf{u}\cdot \mathbf{v}-  \mathsf{p} \textnormal{ div }\mathbf{v})\mathsf{d}\mathbf{x} & = & 	\displaystyle\int_\Omega\mathbf{f}\cdot \mathbf{v}\mathsf{d}\mathbf{x} \quad &\forall \mathbf{v}\in \mathbf{H}_0(\textnormal{div},\Omega),\\
\displaystyle\int_\Omega \mathsf{q}~\textnormal{div} ~\mathbf{u}\mathsf{d}\mathbf{x} & = & 0 \quad &\forall \mathsf{q}\in L_0^2(\Omega),\\
\displaystyle\int_\Omega(\kappa(T)\nabla T\cdot \nabla S - T \mathbf{u}\cdot\nabla S)\mathsf{d}\mathbf{x} & = & \langle \delta_{z}, S\rangle \quad &\forall S\in W_0^{1,p'}(\Omega).
\end{array}
\right.
\end{equation}
Here, $\langle\cdot,\cdot\rangle$ denotes the duality pairing between $W_0^{1,p'}(\Omega)$ and  $W^{-1,p}(\Omega):=(W_0^{1,p'}(\Omega))'$.
\end{definition}

The following comments are now in order. The asymptotic behavior of solutions $\chi$ to second order elliptic problems with homogeneous Dirichlet boundary conditions and $\delta_z$ as a forcing term is dictated by $|\nabla \chi(\mathbf{x})| \approx |\mathbf{x} - z|^{-1}$ \cite[Theorem 3.3]{MR0223740}. On the basis of a simple computation, this asymptotic behavior motivates us to seek for a temperature distribution within the space $W_0^{1,p}(\Omega)$ for $p<2$. On the other hand, we notice that, owing to our assumptions on data and definition of weak solution, all terms in problem \eqref{eq:modelweak} are well-defined. In particular, in view of H\"older's inequality, we have the following bound for the convective term:
\begin{align}
\label{eq:termconv}
\left|\int_{\Omega}T\mathbf{u}\cdot\nabla S d\textbf{x}\right|
&\leq \|\mathbf{u}\|_{\mathbf{L}^2(\Omega)}\|T\|_{L^{\frac{2p}{2-p}}(\Omega)}\|\nabla S\|_{\mathbf{L}^{p'}(\Omega)} \\
&\leq C_{e}
\|\mathbf{u}\|_{\mathbf{L}^2(\Omega)}\| \nabla T\|_{\mathbf{L}^{p}(\Omega)}\| \nabla S\|_{\mathbf{L}^{p'}(\Omega)},\nonumber
\end{align}
where we have utilized the standard Sobolev embedding $W_0^{1,p}(\Omega)\hookrightarrow L^{\frac{2p}{2-p}}(\Omega)$ \cite[Theorem 4.12, Case \textbf{C}]{MR2424078}; $C_{e}$ denotes the best constant in such an embedding.

\subsection{A problem for the single variable $T$}

To analyze problem \eqref{eq:modelweak}, we follow the ideas in \cite[Section 2.2]{MR3802674} and observe that \eqref{eq:modelweak} can be rewritten as a problem for the single variable $T$. In fact, for a given temperature $T$, the first two equations in problem \eqref{eq:modelweak} correspond to a Darcy's problem that, in view of Theorem \ref{thm:wp_darcy}, admits a unique solution $(\mathbf{u},\mathsf{p})\in \mathbf{H}_0(\textnormal{div},\Omega)\times L_0^2(\Omega)$. We notice that the variables $\mathbf{u}$ and $\mathsf{p}$ can be seen as functions depending on $T$. This motivates the notation $(\mathbf{u},\mathsf{p})=(\mathbf{u}(T),\mathsf{p}(T))$. Problem \eqref{eq:modelweak} is thus equivalent to the following reduced formulation \cite[Section 2.2]{MR3802674}: Find $T\in W_0^{1,p}(\Omega)$, with $p<2$, such that
\begin{equation}\label{eq:heat_decoupled}
\int_\Omega(\kappa(T)\nabla T\cdot \nabla S - T \mathbf{u}(T)\cdot\nabla S)\mathsf{d}\mathbf{x}  =  \langle \delta_{z}, S\rangle \qquad \forall S\in W_0^{1,p'}(\Omega),
\end{equation}
where $\mathbf{u}(T) \in \mathbf{H}_0(\textnormal{div},\Omega)$ denotes the velocity component of the solution $(\mathbf{u}(T),\mathsf{p}(T))$ to the following problem: Find $(\mathbf{u}(T),\mathsf{p}(T)) \in \mathbf{H}_0(\textnormal{div},\Omega) \times L_0^2(\Omega)$ such that
\begin{equation}\label{eq:Darcy_decoupled}
\begin{array}{rcll}
\displaystyle\int_\Omega (\nu(T) \mathbf{u}(T)\cdot \mathbf{v}-  \mathsf{p}(T) \textnormal{ div }\mathbf{v})\mathsf{d}\mathbf{x} & = & 	\displaystyle\int_\Omega\mathbf{f}\cdot \mathbf{v}\mathsf{d}\mathbf{x} \quad &\forall \mathbf{v}\in \mathbf{H}_0(\textnormal{div},\Omega),\\
\displaystyle\int_\Omega \mathsf{q}~\textnormal{div} ~\mathbf{u}(T)\mathsf{d}\mathbf{x} & = & 0 \quad &\forall \mathsf{q}\in L_0^2(\Omega).
\end{array}
\end{equation}

%%%%%%%%%%%%%%%%%%%%%%%%%%%%%%%%%%%%%%%%%%%%%%%%%%%%%%%%%%%%%%%%
%%%%%%%%%%%%%%%%%%%%%%%%%%%%%%%%%%%%%%%%%%%%%%%%%%%%%%%%%%%%%%%%
%%%%%%%%%%%%%%%%%%%%%%%%%%%%%%%%%%%%%%%%%%%%%%%%%%%%%%%%%%%%%%%%
%%%%%%%%%%%%%%%%%%%%%%%%%%%%%%%%%%%%%%%%%%%%%%%%%%%%%%%%%%%%%%%%

\subsection{A stationary heat equation with convection} 
\label{subsection:heatequation} 
In this section, we study the existence and uniqueness of solutions for a stationary heat equation with convection and singular forcing. To accomplish this task, we begin our studies by introducing the function $\xi\in L^\infty(\Omega)$, which is such that
\begin{equation}\label{eq:diffusivity_xi}
\xi_{-}\leq \xi(\mathbf{x})\leq \xi_{+} \quad \text{for a.e. } \mathbf{x}\in \Omega,
\qquad \xi_{-},\xi_{+} > 0.
\end{equation}
In addition, we assume that $\xi$ is uniformly continuous. With this function at hand, we introduce the following weak version of the aforementioned stationary heat equation with convection: 
% Find $T\in W_0^{1,p}(\Omega)$ such that
\begin{equation}\label{eq:heat}
T\in W_0^{1,p}(\Omega):
\quad
\displaystyle\int_\Omega\left(\xi\nabla T\cdot \nabla S - T\mathbf{v}\cdot\nabla S \right)\mathsf{d}\mathbf{x} =  \langle \delta_{z}, S\rangle \quad \forall S\in W_0^{1,p'}(\Omega).
\end{equation}
Here, $p$ is such that $4/3 - \epsilon < p < 2$, where $\epsilon >0$, $\mathbf{v}\in \mathbf{L}^2(\Omega)$, $z\in\Omega$, and $1/p+1/p'=1$.

We present the following well-posedness result.

\begin{proposition}[case $\mathbf{v}=\mathbf{0}$]\label{prop:u=0}
Let $\xi\in L^\infty(\Omega)$ be such that \eqref{eq:diffusivity_xi} holds. Assume, in addition, that $\xi$ is uniformly continuous. Then, problem \eqref{eq:heat} with $\mathbf{v}=\mathbf{0}$ is well-posed. This, in particular, implies that
\begin{equation}\label{eq:inf_sup_Poisson}
\| \nabla R \|_{\mathbf{L}^{p}(\Omega)}\leq C_{\xi}\sup_{S\in W_0^{1,p'}(\Omega)}\dfrac{\int_\Omega\xi\nabla R\cdot \nabla S \mathsf{d}\mathbf{x}}{\|\nabla S\|_{\mathbf{L}^{p'}(\Omega)}}\qquad \forall R\in W_0^{1,p}(\Omega),
\end{equation}
with a constant $C_{\xi}$ depending on $\xi$, $p$, and $\Omega$.
\end{proposition}
\begin{proof}
The well-posedness of problem \eqref{eq:heat} on Lipschitz domains with $\mathbf{v}=\mathbf{0}$ for $p$ such that $4/3 - \epsilon < p < 2$ and $\epsilon >0$ follows, for instance, from \cite[Theorem B]{MR2141694}. Notice that $\xi$ is bounded and uniformly continuous.
With such a result at hand, the inf-sup condition \eqref{eq:inf_sup_Poisson} thus follows from \cite[Theorem 2.6]{Guermond-Ern}.
\end{proof}

We now analyze the case with nonzero convection.

\begin{proposition}[case $\mathbf{v}\neq\mathbf{0}$]\label{prop:wp_temp}
Let $\xi\in L^\infty(\Omega)$ be such that \eqref{eq:diffusivity_xi} holds. Assume, in addition, that $\xi$ is uniformly continuous. If
\begin{equation}\label{condicion}
C_\xi C_e\|\mathbf{v}\|_{\mathbf{L}^2(\Omega)} \leq \alpha< 1,	
\end{equation}
with $C_e$ being the best constant in the Sobolev embedding $W_0^{1,p}(\Omega)\hookrightarrow L^{\frac{2p}{2-p}}(\Omega)$, then problem \eqref{eq:heat} is well-posed. This, in particular, implies that the solution $T\in W_0^{1,p}(\Omega)$ of problem \eqref{eq:heat} satisfies
\begin{equation}\label{eq:estimatesheat}
\| \nabla T\|_{\mathbf{L}^{p}(\Omega)}\leq C_{\alpha} \|\delta_z\|_{W^{-1,p}(\Omega)}, \quad C_{\alpha}=\tfrac{C_\xi}{1-\alpha}, 
\qquad
4/3-\epsilon <p<2.
\end{equation}
\end{proposition}
\begin{proof}
We begin the proof by introducing the map $\mathcal{A}:W_0^{1,p} (\Omega)\to W^{-1,p}(\Omega)$ by
\begin{equation*}
\langle \mathcal{A}T,S\rangle:=\int_\Omega \xi \nabla T\cdot \nabla S \mathsf{d}\mathbf{x},
\qquad
\forall T\in W_0^{1,p}(\Omega),
~\forall S \in W_0^{1,p'}(\Omega).
\end{equation*}
It is clear that $\mathcal{A}$ is linear and bounded. In addition, in view of the inf-sup condition \eqref{eq:inf_sup_Poisson}, we conclude that $\mathcal{A}$ is invertible and $\|\mathcal{A}^{-1}\|_{\mathcal{L}(W^{-1,p}(\Omega),W_0^{1,p}(\Omega))}\leq C_\xi$.

Let $\mathbf{v}\in \mathbf{L}^2(\Omega)$. Let us also introduce the map $\mathcal{B}_{\mathbf{v}}: W_0^{1,p}(\Omega)\to W^{-1,p}(\Omega)$ by
\begin{equation*}
\langle\mathcal{B}_{\mathbf{v}}T,S\rangle:=-\int_\Omega T\mathbf{v}\cdot \nabla S\mathsf{d}\textbf{x},\qquad \forall T\in  W_0^{1,p}(\Omega),~\forall S\in W_0^{1,p'}(\Omega).
\end{equation*}
The map $\mathcal{B}_{\mathbf{v}}$ is linear and, in view of \eqref{eq:termconv}, bounded. In fact, we have
\begin{equation*}
\displaystyle\|\mathcal{B}_{\mathbf{v}}\|_{\mathcal{L}(W_0^{1,p}(\Omega),W^{-1,p}(\Omega))}
= \sup_{T\in W_0^{1,p}(\Omega)} \dfrac{\| \mathcal{B}_{\mathbf{v}}T \|_{W^{-1,p}(\Omega)}}{\|\nabla T\|_{\mathbf{L}^{p}(\Omega)} }
\leq C_e\|\mathbf{v}\|_{\mathbf{L}^2(\Omega)}.
\end{equation*}

We thus employ the previously defined linear and bounded maps $\mathcal{A}$ and $\mathcal{B}_\mathbf{v}$ to rewrite problem \eqref{eq:heat} as the following operator equation in $W_0^{1,p}(\Omega)$: $(I+\mathcal{A}^{-1}\mathcal{B}_{\mathbf{v}})T=\mathcal{A}^{-1}\delta_{z}$. We now observe that the boundedness of the maps $\mathcal{A}^{-1}$ and $\mathcal{B}_\mathbf{v}$ combined with assumption \eqref{condicion} yield that the $\mathcal{L}(W_0^{1,p}(\Omega))$-norm of the map $\mathcal{A}^{-1}\mathcal{B}_{\mathbf{v}}$  is bounded by
$
C_{\xi}C_e \|\mathbf{v}\|_{\mathbf{L}^2(\Omega)}\leq \alpha <1.
$
This bound, in view of \cite[Theorem 1.B]{MR816732}, allows us to conclude that problem \eqref{eq:heat} admits a unique solution. The desired bound for $T$ can be obtained directly from the aforementioned operator equation. In fact, we have
\begin{equation*}
\| \nabla T\|_{\mathbf{L}^{p}(\Omega)}\leq \dfrac{\|\mathcal{A}^{-1}\|_{\mathcal{L}(W^{-1,p}(\Omega),W_0^{1,p}(\Omega))}}{1-\|\mathcal{A}^{-1}\mathcal{B}_{\mathbf{v}}\|_{\mathcal{L}(W_0^{1,p}(\Omega))}}\|\delta_z\|_{W^{-1,p}(\Omega)}\leq \dfrac{C_\xi}{1-\alpha}\|\delta_z\|_{W^{-1,p}(\Omega)}.
\end{equation*}
This concludes the proof.
\end{proof}

%%%%%%%%%%%%%%%%%%%%%%%%%%%%%%%%%%%%%%%%%%%%%%%%%%%%%%%%%%%%%%%%
%%%%%%%%%%%%%%%%%%%%%%%%%%%%%%%%%%%%%%%%%%%%%%%%%%%%%%%%%%%%%%%%
%%%%%%%%%%%%%%%%%%%%%%%%%%%%%%%%%%%%%%%%%%%%%%%%%%%%%%%%%%%%%%%%
%%%%%%%%%%%%%%%%%%%%%%%%%%%%%%%%%%%%%%%%%%%%%%%%%%%%%%%%%%%%%%%%

\subsection{The coupled problem}
We now proceed to analyze the existence of solutions for problem \eqref{eq:heat_decoupled}--\eqref{eq:Darcy_decoupled} on the basis of a fixed point argument. Let $p < 2$ and let $\mathcal{F}: W_0^{1,p}(\Omega)\to W_0^{1,p}(\Omega)$ be the map defined by $\mathcal{F}(\theta)=\zeta$, where $\zeta$ denotes the solution to the following problem: Find $\zeta \in W_0^{1,p}(\Omega)$ such that
\begin{equation}\label{eq:operator_F}
\int_\Omega(\kappa(\theta)\nabla \zeta\cdot \nabla S - \zeta \mathbf{u}(\theta)\cdot\nabla S)\mathsf{d}\mathbf{x}  =  \langle \delta_{z}, S\rangle \quad \forall S\in W_0^{1,p'}(\Omega).
\end{equation}
We recall that $1/p+1/p'=1$. Notice that the definition of $\zeta$ implies solving, for a prescribed temperature $\theta \in W_0^{1,p}(\Omega)$, a Darcy's problem with viscosity $\nu(\theta)$. Since $\nu$ satisfies the estimates in \eqref{eq:nu}, Theorem \ref{thm:wp_darcy} guarantees the existence of a unique solution $(\mathbf{u}(\theta),\mathsf{p}(\theta))$ for such a problem. Once this solution is obtained, the stationary heat equation \eqref{eq:operator_F}, with the nonzero convection $\mathbf{u}(\theta)$, is thus solved to obtain $\zeta\in W_0^{1,p}(\Omega)$.

As a first instrumental result, we prove that the mapping $\mathcal{F}$ is well-defined. To accomplish this task, we introduce the ball
\begin{equation*}
\mathfrak{B}_{T}:= \left\{ \theta\in W_0^{1,p}(\Omega)~:~ \|\nabla\theta\|_{\mathbf{L}^{p}(\Omega)}\leq C_{\frac{1}{2}}\|\delta_{z}\|_{W^{-1,p}(\Omega)} \right \},
\end{equation*}
where $C_{\frac{1}{2}}$ is as in \eqref{eq:estimatesheat}. Since it will be useful in the analysis that follows, we define
\begin{equation}\label{def:constant_C}
\mathfrak{C}:=(2C_{e}C_{\kappa})^{-1},
\end{equation}
where $C_e$ is defined as in \eqref{eq:termconv} and $C_{\kappa}$ corresponds to the constant involved in the inf-sup condition \eqref{eq:inf_sup_Poisson} when $\xi$ is replaced by $\kappa$.

\begin{lemma}[$\mathcal{F}$ is well-defined] 
\label{lemma:welldefined}
Let $p$ such that $4/3 - \epsilon < p < 2$ and $\epsilon >0$. Let $\mathbf{f}$ be such that  $\|\mathbf{f}\|_{\mathbf{L}^2(\Omega)}\leq \mathfrak{C} \nu_{-}$, where $\nu_{-}$ is as in \eqref{eq:nu}. Then, the map $\mathcal{F}$ is well-defined on $\mathfrak{B}_{T}$ and, in addition, $\mathcal{F}(\mathfrak{B}_{T})\subset \mathfrak{B}_{T}$.
\end{lemma}
\begin{proof}
Let $\theta\in \mathfrak{B}_{T}$. Invoke Theorem \ref{thm:wp_darcy} to conclude the existence of a unique pair $(\mathbf{u}(\theta), \mathsf{p}(\theta)) \in \mathbf{H}_0(\text{div},\Omega) \times L_0^2(\Omega)$ solving \eqref{eq:Darcy_decoupled} with $T=\theta$. In addition, by testing $\mathbf{v} = \mathbf{u}(\theta)$ in the first equation of \eqref{eq:Darcy_decoupled}, we immediately obtain the estimate
\begin{equation*}
\|\mathbf{u}(\theta)\|_{\mathbf{L}^2(\Omega)}
\leq
\nu_{-}^{-1}
\|\mathbf{f}\|_{\mathbf{L}^2(\Omega)}
\leq
\mathfrak{C} = (2C_{e}C_{\kappa})^{-1},
\end{equation*}
upon utilizing definition \eqref{def:constant_C}. Consequently, $\mathbf{u}(\theta)$ satisfies the bound \eqref{condicion} with $\alpha= \frac12$. We are thus in position to apply the results of Proposition \ref{prop:wp_temp} to conclude the existence of a unique solution $\zeta\in W_0^{1,p}(\Omega)$ to \eqref{eq:operator_F} satisfying $\| \nabla \zeta \|_{\mathbf{L}^{p}(\Omega)}\leq C_{\frac{1}{2}} \|\delta_z\|_{W^{-1,p}(\Omega)}$. We have thus proved that $\zeta\in\mathfrak{B}_{T}$.
\end{proof}

We now proceed to obtain an existence result for problem \eqref{eq:heat_decoupled}--\eqref{eq:Darcy_decoupled} via a fixed point argument.

\begin{theorem}[existence]\label{thm:existence}
Under the assumptions of Lemma \ref{lemma:welldefined}, there exists a solution $(\mathbf{u},\mathsf{p},T) \in \mathbf{H}_0(\textnormal{div},\Omega)\times L_0^2(\Omega)\times W_0^{1,p}(\Omega)$ for problem \eqref{eq:modelweak}. In addition, we have that $T\in \mathfrak{B}_{T}$.
\end{theorem}
\begin{proof}
We proceed on the basis of the Leray--Schauder fixed point theorem for the map $\mathcal{F}: \mathfrak{B}_{T} \rightarrow \mathfrak{B}_{T} $ \cite[Theorem 8.8]{MR787404}, \cite[Theorem 2.A]{MR816732}. In order to apply such a theorem, we first observe that, directly from its definition, the set $\mathfrak{B}_{T}$ is nonempty, closed, bounded, and convex. Additionally, Lemma \ref{lemma:welldefined} guarantees that $\mathcal{F}(\mathfrak{B}_T)\subset \mathfrak{B}_T$. It thus suffices to prove that  $\mathcal{F}$ is compact.

Let $\{\theta_{n}\}_{n\geq 0}\subset \mathfrak{B}_{T}$ be a sequence such that $\theta_{n}\rightharpoonup \theta\in W_0^{1,p}(\Omega)$ for $p<2$. Since $\mathfrak{B}_{T}$ is closed and convex, it immediately follows that $\mathfrak{B}_{T}$ is weakly closed. This implies that $\theta \in \mathfrak{B}_T$. Define $\zeta_{n}:=\mathcal{F}(\theta_n)$ and $\zeta:=\mathcal{F}(\theta)$. In what follows, we prove that $\zeta_{n}\to \zeta$ in $W_0^{1,p}(\Omega)$ as $k\uparrow\infty$. To accomplish this task, we invoke the problems that $\zeta$ and $\zeta_n$ satisfy and observe that the difference $e_{\zeta,n}:=\zeta-\zeta_n$ verifies
\begin{multline*}
\int_\Omega \left(\kappa(\theta_n)\nabla e_{\zeta,n}-e_{\zeta,n}\mathbf{u}(\theta)\right)\cdot \nabla S \mathsf{d}\mathbf{x}\\
=\int_\Omega \left[ \zeta_n (\mathbf{u}(\theta)-\mathbf{u}(\theta_{n}))+(\kappa(\theta_n)-\kappa(\theta))\nabla \zeta \right]\cdot \nabla S\mathsf{d}\mathbf{x}
:=\langle g_n, S\rangle,
\end{multline*}
for all $S\in W_0^{1,p'}(\Omega)$, i.e., $e_{\zeta,n}$ solves a heat equation with nonzero convection.  Notice that, since $\|\mathbf{u}(\theta)\|_{\mathbf{L}^2(\Omega)}\leq \mathfrak{C}$, we can invoke Proposition \ref{prop:wp_temp} to immediately arrive at
\begin{equation}\label{eq:bound_e_zeta_n}
\|\nabla e_{\zeta,n}\|_{\mathbf{L}^{p}(\Omega)}\leq C_{\frac{1}{2}} \|g_n\|_{W^{-1,p}(\Omega)}.
\end{equation}

Let us now study convergence properties of $\{g_{n}\}_{n\geq 0}$ as $n\uparrow \infty$. To accomplish this task, we analyze each term compromised in the definition of $g_{n}$ separately. We first invoke H\"older's inequality and the embedding $W_0^{1,p}(\Omega)\hookrightarrow L^{2p/(2-p)}(\Omega)$ to arrive at
\begin{align*}
\int_{\Omega}\zeta_n(\mathbf{u}(\theta)-\mathbf{u}(\theta_{n}))\cdot \nabla S\mathsf{d}\mathbf{x}
&\leq
\|\zeta_n\|_{L^{\frac{2p}{2-p}}(\Omega)}\|\mathbf{u}(\theta)-\mathbf{u}(\theta_{n})\|_{\mathbf{L}^2(\Omega)}\|\nabla S\|_{\mathbf{L}^{p'}(\Omega)}\hspace{-5cm} \\  \nonumber
&\leq C_{e}\|\nabla \zeta_n\|_{\mathbf{L}^{p}(\Omega)}\|\mathbf{u}(\theta)-\mathbf{u}(\theta_{n})\|_{\mathbf{L}^2(\Omega)}\|\nabla S\|_{\mathbf{L}^{p'}(\Omega)};
\end{align*}
$C_{e}$ being the best constant in the aforementioned embedding. Let us now estimate the term $\|\mathbf{u}(\theta)-\mathbf{u}(\theta_{n})\|_{\mathbf{L}^2(\Omega)}$ in the previous inequality. Invoke \eqref{eq:nu}, add and subtract the term $\sqrt{\nu(\theta)}\mathbf{u}(\theta)$, and utilize a triangle inequality to obtain
\begin{multline*}
\|\mathbf{u}(\theta)-\mathbf{u}(\theta_{n})\|_{\mathbf{L}^2(\Omega)}
\lesssim 
\| \sqrt{\nu(\theta_n)} ( \mathbf{u}(\theta)-\mathbf{u}(\theta_{n}) )\|_{\mathbf{L}^2(\Omega)}
\\
\lesssim
\|\nu(\theta)^{\frac{1}{2}}\mathbf{u}(\theta)-\nu(\theta_n)^{\frac{1}{2}}\mathbf{u}(\theta_{n})\|_{\mathbf{L}^2(\Omega)}
+
\|(\nu(\theta_n)^{\frac{1}{2}}-\nu(\theta)^{\frac{1}{2}})\mathbf{u}(\theta)\|_{\mathbf{L}^2(\Omega)}
=\mathrm{I}+\mathrm{II}.
\end{multline*}
Since $\theta_n\rightharpoonup \theta$ in $W_0^{1,p}(\Omega)$, we invoke the compact embedding $W_0^{1,p}(\Omega)\hookrightarrow L^2(\Omega)$ \cite[Theorem 6.3, Part \textbf{I}]{MR2424078} to obtain the strong convergence $\theta_n\to \theta$ in $L^2(\Omega)$ as $n \uparrow \infty$. An application of \cite[Lemma 2.1]{MR3802674} thus reveals that $\mathrm{I}\to 0$ as $n\uparrow \infty$. On the other hand, since $\nu$ is continuous and uniformly bounded, the strong convergence $\theta_n\to \theta$ in $L^2(\Omega)$ guarantees that $\nu(\theta_n) \to \nu(\theta)$ in $L^2(\Omega)$ \cite[Theorem 7]{MR120342}. Invoke the boundedness of $\nu$, the fact that $\mathbf{u}(\theta) \in \mathbf{L}^2(\Omega)$, and the Lebesgue dominated convergence to conclude that $\mathrm{II}\to 0$ as $n\uparrow \infty$. To control the remaining term in $g_{n}$ we proceed with similar arguments upon noticing that $\kappa$ is continuous and uniformly bounded, which imply that $(\kappa(\theta_n)-\kappa(\theta))\nabla \zeta \to \mathbf{0}$ in $\mathbf{L}^p(\Omega)$. Therefore, in view of \eqref{eq:bound_e_zeta_n}, $\zeta_n\to \zeta$ in $W_0^{1,p}(\Omega)$ as $n\uparrow \infty$. We have thus proved that the weak convergence $\theta_n\rightharpoonup \theta$ in $W_0^{1,p}(\Omega)$ implies the strong one $\zeta_n\to \zeta$ in $W_0^{1,p}(\Omega)$ as $n\uparrow \infty$. This shows that $\mathcal{F}$ is compact and concludes the proof.
\end{proof}

%%%%%%%%%%%%%%%%%%%%%%%%%%%%%%%%%%%%%%%%%%%%%%%%%%%%%%%%%%%%%%%%
%%%%%%%%%%%%%%%%%%%%%%%%%%%%%%%%%%%%%%%%%%%%%%%%%%%%%%%%%%%%%%%%
%%%%%%%%%%%%%%%%%%%%%%%%%%%%%%%%%%%%%%%%%%%%%%%%%%%%%%%%%%%%%%%%
%%%%%%%%%%%%%%%%%%%%%%%%%%%%%%%%%%%%%%%%%%%%%%%%%%%%%%%%%%%%%%%%
%%%%%%%%%%%%%%%%%%%%%%%%%%%%%%%%%%%%%%%%%%%%%%%%%%%%%%%%%%%%%%%%
%%%%%%%%%%%%%%%%%%%%%%%%%%%%%%%%%%%%%%%%%%%%%%%%%%%%%%%%%%%%%%%%
%%%%%%%%%%%%%%%%%%%%%%%%%%%%%%%%%%%%%%%%%%%%%%%%%%%%%%%%%%%%%%%%
%%%%%%%%%%%%%%%%%%%%%%%%%%%%%%%%%%%%%%%%%%%%%%%%%%%%%%%%%%%%%%%%

\section{Finite element approximation}
\label{sec:fem}

In this section, we describe and analyze a finite element solution technique to approximate solutions to problem \eqref{eq:modelweak}. We begin our analysis by introducing some terminology and a few basic ingredients \cite{MR2373954,CiarletBook,Guermond-Ern}. We denote by $\mathscr{T}_h = \{ K\}$ a conforming partition, or mesh, of $\bar{\Omega}$ into closed simplices $K$ with size $h_K = \text{diam}(K)$. Define $h:=\max_{ K \in \mathscr{T}_h} h_K$. We denote by $\mathbb{T} = \{\mathscr{T}_h \}_{h>0}$ a collection of conforming and shape regular meshes $\mathscr{T}_h$.  We define $\mathscr{S}$ as the set of internal one-dimensional interelement boundaries $\gamma$ of $\T_h$. For $K \in \T_h$, let $\mathscr{S}^{}_K$ denote the subset of $\mathscr{S}$ that contains the sides in $\mathscr{S}$ which are sides of $K$. We denote by $\mathcal{N}_{\gamma}$, for $\gamma \in \mathscr{S}$, the subset of $\T_h$ that contains the two elements that have $\gamma$ as a side. In addition, we define \emph{stars} or \emph{patches} associated with an element $K \in \T_h$ as:
\begin{equation}
\label{eq:patch}
\mathcal{N}_K= \cup \{K' \in \T_h: \mathscr{S}_K \cap \mathscr{S}_{K'} \neq \emptyset \},
\quad 
\mathcal{N}_K^*= \cup \{K' \in \T_h: K \cap {K'} \neq \emptyset \}.
\end{equation}
In an abuse of notation, below we denote by $\mathcal{N}_K$ and $\mathcal{N}_K^*$ either the sets themselves or the union of its elements.

Given a mesh $\mathscr{T}_{h} \in \mathbb{T}$, we define the finite element space of continuous piecewise polynomials of degree one:
\begin{equation*}
V_{h}:=\{S_{h}\in C(\bar{\Omega}): S_{h}|_K\in \mathbb{P}_{1}(K) \ \forall K\in \T_{h}\}\cap W_0^{1,p\prime}(\Omega),
\end{equation*}
where $p'>2$. Notice that, for each $h>0$, $V_{h}\subset W_0^{1,p\prime}(\Omega)\subset W_0^{1,p}(\Omega)$. 

We denote by $I_h$ the \emph{Lagrange interpolation operator} and immediately notice that, since $W_0^{1,p\prime}(\Omega)\hookrightarrow C(\bar{\Omega})$, $I_h$ is well-defined as a map from $W_0^{1,p\prime}(\Omega)$ into $V_h$ \cite[Example 1.106]{Guermond-Ern}. The following error estimate can be found in \cite[Theorem 1.103]{Guermond-Ern}: for each $K\in {\T}_h$,
\begin{equation}\label{eq:interp_Lagrange}
\|S-I_h S\|_{L^{p\prime}(K)}
\lesssim
h_K\|\nabla S\|_{\mathbf{L}^{p\prime}(K)} \quad \forall S\in W_0^{1,p\prime}(K).
\end{equation}
With this estimate at hand, a trace identity yields, for $\gamma \in \mathscr{S}$, the estimate
\begin{equation}\label{eq:interp_Lagrange_side}
 \|S-I_h S\|_{L^{p\prime}(\gamma)}
\lesssim
h_{\gamma}^\frac{1}{p}\|\nabla S\|_{\mathbf{L}^{p\prime}(\mathcal{N}_{\gamma})} \quad \forall S\in W_0^{1,p\prime}(\mathcal{N}_{\gamma}).
\end{equation}

To approximate the pair velocity--pressure that solves problem \eqref{eq:Darcy_decoupled}, we consider the Raviart--Thomas finite element space of order zero ($RT_{0}$):
\begin{align*}
\mathbf{X}_{h} &:= \mathbf{W}_{h} \cap \mathbf{H}_{0}(\text{div},\Omega),\\
Q_{h}&:=\{\mathsf{q}_{h}\in L^2(\Omega): \mathsf{q}_{h}|_{K}^{} \in \mathbb{P}_{0}(K) \ \forall K\in \T_{h} \}\cap L_0^2(\Omega),
\end{align*}
where $\mathbf{W}_{h}:=\left\{ \mathbf{v}_{h}\in \mathbf{H}(\text{div},\Omega): \mathbf{v}_{h}|_{K}^{}=a_{K}\mathbf{x}+\mathbf{b}_{K}, \ a_{K}\in\mathbb{R}, \ \mathbf{b}_{K}\in\mathbb{R}^{2},\ \forall K\in\T_{h} \right \}$. The spaces $\mathbf{X}_{h}$ and $Q_{h}$ satisfy the following discrete inf-sup condition \cite[Theorem 13.2]{MR1115239}: there exists $\tilde{\beta}>0$, independent of the discretization parameter $h$, such that
\begin{equation}\label{eq:discrete_inf_sup}
\inf_{\mathsf{q}_h\in Q_{h}} \sup_{\mathbf{v}_h\in\mathbf{X}_{h}}\dfrac{\int_\Omega \mathsf{q}_h \text{ div }\mathbf{v}_h\mathsf{d}\mathbf{x}}{\|\mathbf{v}_h\|_{\mathbf{H}(\text{div},\Omega)}\|\mathsf{q}_h\|_{L^2(\Omega)}}\geq \tilde{\beta}>0.
\end{equation}

Let us introduce the interpolation operator $\mathbf{P}_h: \mathbf{H}^1(\Omega)\cap \mathbf{H}_{0}(\text{div},\Omega)\to \mathbf{X}_{h}$, which satisfies, for each $K\in {\T}_h$, the following error estimates \cite[Theorem 6.3]{MR1115239}:
\begin{align}
\label{eq:interp_RT_i}
\|\mathbf{v}-\mathbf{P}_h\mathbf{v}\|_{\mathbf{L}^2(K)}
& \lesssim
h\|\nabla \mathbf{v}\|_{\mathbf{L}^2(K)} \quad \forall \mathbf{v}\in \mathbf{H}^1(\Omega),
\\
\label{eq:interp_RT_ii}
\|\text{div}(\mathbf{v}-\mathbf{P}_h\mathbf{v})\|_{\mathbf{L}^2(K)}
& \lesssim
h\|\nabla \text{div }\mathbf{v}\|_{\mathbf{L}^2(K)} \quad \forall \mathbf{v}\in \mathbf{H}^1(\Omega):~\text{div }\mathbf{v}\in H^1(\Omega).
\end{align}
We also have the local error estimate \cite[inequality (4.27)]{MR4041519}:
\begin{equation}
 \label{eq:interp_RT_iii}
 \|\mathbf{v}-\mathbf{P}_h\mathbf{v}\|_{\mathbf{L}^2(\gamma)}
\lesssim
h_{\gamma}^{\frac{1}{2}}\|\nabla \mathbf{v}\|_{\mathbf{L}^2(\mathcal{N}_{\gamma})} \quad \forall \mathbf{v}\in \mathbf{H}^1(\Omega).
\end{equation}

In addition, we observe that, for every $\mathbf{v} \in \mathbf{L}^2(\Omega)$ and $\mathbf{w} \in \mathbf{H}(\text{div},\Omega)$, the following density results holds:
\begin{equation}
\label{eq:density_result}
\lim_{h \rightarrow 0} \left( \inf_{\mathbf{v}_h \in \mathbf{X}_h} \| \mathbf{v} - \mathbf{v}_h\|_{\mathbf{L}^2(\Omega)} \right)  = 0,
\quad
\lim_{h \rightarrow 0} \left( \inf_{\mathbf{w}_h \in \mathbf{X}_h} \| \text{div }(\mathbf{w} -\mathbf{w}_h)\|_{\mathbf{L}^2(\Omega)} \right)  = 0.\hspace{-0.3cm}
\end{equation}

Having described our finite element setting, we introduce the following discrete approximation of problem \eqref{eq:modelweak}: Find $(\mathbf{u}_h, \mathsf{p}_h, T_h)\in \mathbf{X}_{h}\times Q_h\times V_h$ such that
\begin{equation}\label{eq:model_discrete}
\left\{
\begin{array}{rcll}
\displaystyle\int_\Omega (\nu(T_h) \mathbf{u}_h\cdot \mathbf{v}_h-  \mathsf{p}_h \textnormal{ div }\mathbf{v}_h)\mathsf{d}\mathbf{x}& = & 	\displaystyle\int_\Omega\mathbf{f}\cdot \mathbf{v}_h\mathsf{d}\mathbf{x} \quad &\forall \mathbf{v}_h \in \mathbf{X}_h,\\
\displaystyle\int_\Omega \mathsf{q}_h\textnormal{ div }\mathbf{u}_h\mathsf{d}\mathbf{x} & = & 0 \quad &\forall \mathsf{q}_h\in Q_h,\\
\displaystyle\int_\Omega(\kappa(T_h)\nabla T_h\cdot \nabla S_h - T_h \mathbf{u}_h\cdot\nabla S_h)\mathsf{d}\mathbf{x} & = & \langle \delta_{z}, S_h\rangle \quad &\forall S_h\in V_h.
\end{array}
\right.
\end{equation}

The main goal of this section is to show that, under similar assumptions to those in Theorem \ref{thm:existence}, problem \eqref{eq:model_discrete} always has a solution for every $h>0$. We also show that, as $h \rightarrow 0$, the sequence of solutions $(\mathbf{u}_h, \mathsf{p}_h, T_h)$ weakly converge, up to subsequences, to a solution of the coupled system \eqref{eq:modelweak}.

%%%%%%%%%%%%%%%%%%%%%%%%%%%%%%%%%%%%%%%%%%%%%%%%%%%%%%%%%%
%%%%%%%%%%%%%%%%%%%%%%%%%%%%%%%%%%%%%%%%%%%%%%%%%%%%%%%%%%
%%%%%%%%%%%%%%%%%%%%%%%%%%%%%%%%%%%%%%%%%%%%%%%%%%%%%%%%%%
%%%%%%%%%%%%%%%%%%%%%%%%%%%%%%%%%%%%%%%%%%%%%%%%%%%%%%%%%%

\subsection{A discrete heat equation}
In this section, we prove a discrete counterpart of Proposition \ref{prop:wp_temp}. To accomplish this task, we first provide a discrete inf-sup condition which directly stems from \cite[Proposition 8.6.2]{MR2373954}.

\begin{proposition}[discrete stability]\label{prop:disc_stab}
Let $\xi\in L^\infty(\Omega)$ be such that \eqref{eq:diffusivity_xi} holds. Then, there exist $h_{\star}>0$ and $\epsilon >0$ such that for all $0< h \leq h_{\star}$ and $R_h \in V_h$, we have
\begin{equation}\label{eq:discrete_inf_sup_Poisson}
\|\nabla R_{h} \|_{\mathbf{L}^{p}(\Omega)}\leq \tilde{C}_{\xi}\sup_{S_{h}\in  V_h}\dfrac{\int_\Omega \xi \nabla R_{h} \cdot \nabla S_{h} \mathsf{d}\mathbf{x}}{\|\nabla S_{h}\|_{\mathbf{L}^{p'}(\Omega)}},
\end{equation}
whenever $2 - \epsilon \leq p < 2$. Here, $\tilde{C}_{\xi}$ is a positive constant that is independent of $h$.
\end{proposition}

Let $\xi\in L^{\infty}(\Omega)$ be such that \eqref{eq:diffusivity_xi} holds and let $\mathbf{v}\in \mathbf{L}^2(\Omega)$. We introduce the following discrete version of problem \eqref{eq:heat}: Find $T_{h}\in V_{h}$ such that
\begin{equation}\label{eq:discrete_heat}
\displaystyle\int_\Omega \left(\xi\nabla T_{h}\cdot \nabla S_{h} - T_{h}\mathbf{v}\cdot\nabla S_{h}\right)\mathsf{d}\mathbf{x} =  \langle \delta_{z}, S_{h}\rangle \qquad \forall S_{h}\in V_h.
\end{equation}

With the result of Proposition \eqref{prop:disc_stab} at hand, in the next result we show that, under a suitable smallness assumption on the convective term, problem \eqref{eq:discrete_heat} always has a discrete solution. In addition, we show that discrete solutions are uniformly bounded with respect to the discretization parameter $h$.

\begin{proposition}[well--posedness]\label{prop:disc_wp_temp}
Let $\xi\in L^\infty(\Omega)$ be such that \eqref{eq:diffusivity_xi} holds. There exist positive constants $h_{\star}$ and $\epsilon$ such that, if
\begin{equation}\label{eq:disc_condicion}
\tilde{C}_\xi C_e\|\mathbf{v}\|_{\mathbf{L}^2(\Omega)} \leq \alpha< 1,
\end{equation}
then the discrete problem \eqref{eq:discrete_heat} is well-posed for all $0< h \leq h_{\star}$ whenever $2 - \epsilon \leq p < 2$. In particular, we have that the solution $T_{h}\in V_{h}$ of problem \eqref{eq:discrete_heat} satisfies the estimate
\begin{equation}\label{eq:disc_estimatesheat}
\| \nabla T_{h}\|_{\mathbf{L}^{p}(\Omega)}\leq \tilde{C}_{\alpha} \|\delta_z\|_{W^{-1,p}(\Omega)}, \qquad \tilde{C}_{\alpha}=\tfrac{\tilde{C}_\xi}{1-\alpha},
\qquad
p \in [2-\epsilon,2).
\end{equation}
\end{proposition}
\begin{proof}
With the discrete inf-sup condition \eqref{eq:discrete_inf_sup_Poisson} at hand, the proof follows the same arguments as those developed for the proof of Proposition \ref{prop:wp_temp}. For brevity, we skip the details.
\end{proof}

%%%%%%%%%%%%%%%%%%%%%%%%%%%%%%%%%%%%%%%%%%%%%%%%%%%%%%%%%%
%%%%%%%%%%%%%%%%%%%%%%%%%%%%%%%%%%%%%%%%%%%%%%%%%%%%%%%%%%
%%%%%%%%%%%%%%%%%%%%%%%%%%%%%%%%%%%%%%%%%%%%%%%%%%%%%%%%%%
%%%%%%%%%%%%%%%%%%%%%%%%%%%%%%%%%%%%%%%%%%%%%%%%%%%%%%%%%%

\subsection{Existence of discrete solutions} Having derived a well-posedness result for the discrete heat equation \eqref{eq:discrete_heat}, we are now in position to prove that our discrete system \eqref{eq:model_discrete} always has a solution. In addition, we show that solutions are uniformly bounded with respect to the discretization parameter $h$.

We proceed via a fixed point argument and define, for each $h > 0$, the map
$\mathcal{F}_{h}: V_h \rightarrow V_h$ by $\theta_{h} \mapsto \mathcal{F}(\theta_{h})=\zeta_{h}$. Here, $\zeta_{h}$ denotes the solution to the following discrete problem: Find $\zeta_{h} \in V_h$ such that
\begin{equation}
\displaystyle\int_\Omega(\kappa(\theta_h)\nabla \zeta_h\cdot \nabla S_h - \zeta_h \mathbf{u}_h(\theta_h)\cdot\nabla S_h)\mathsf{d}\mathbf{x}  =  \langle \delta_{z}, S_h\rangle \quad \forall S_h\in V_h.
\label{eq:discrete_zeta_h}
\end{equation}
As in the continuous case, we note that the definition of $\zeta_h$ implies solving a discrete Darcy's problem with a viscosity $\nu(\theta_h)$ depending on the prescribed discrete temperature $\theta_h$. The aforementioned discrete Darcy's problem reads as follows: Find $(\mathbf{u}_{h}(\theta_h),\mathsf{p}_{h}(\theta_h))\in \mathbf{X}_{h} \times Q_{h}$ such that
\begin{equation}\label{eq:discrete_darcy_system}
\begin{array}{rcll}
\displaystyle\int_\Omega (\nu(\theta_h) \mathbf{u}_h(\theta_h)\cdot \mathbf{v}_h-  \mathsf{p}_h(\theta_h) \textnormal{ div }\mathbf{v}_h)\mathsf{d}\mathbf{x}& = & 	\displaystyle\int_\Omega\mathbf{f}\cdot \mathbf{v}_h\mathsf{d}\mathbf{x} \quad &\forall \mathbf{v}_h\in \mathbf{X}_h,\\
\displaystyle\int_\Omega \mathsf{q}_h\textnormal{ div }\mathbf{u}_h(\theta_h)\mathsf{d}\mathbf{x} & = & 0 \quad &\forall \mathsf{q}_h\in Q_h.
\end{array}\hspace{-0.3cm}
\end{equation}

In what follows, we prove that the mapping $\mathcal{F}_{h}$ is well--defined when it is restricted to a ball of an appropriate size. To accomplish this task, we define the ball
\begin{equation*}
\mathfrak{B}_{T}^{h}:=\{\theta_{h}\in V_{h}:~ \|\nabla\theta_{h}\|_{\mathbf{L}^{p}(\Omega)}\leq \tilde{C}_{\frac{1}{2}}\|\delta_{z}\|_{W^{-1,p}(\Omega)}\},
\end{equation*}
where $\tilde{C}_{\frac{1}{2}}$ is defined as in \eqref{eq:disc_estimatesheat} with $\alpha=\frac{1}{2}$. As a final ingredient, we define
\begin{equation}\label{def:constant_C_tilde}
\tilde{\mathfrak{C}}:=(2C_{e}\tilde{C}_{\kappa})^{-1},
\end{equation}
where $C_e$ is defined as in \eqref{eq:termconv} and $\tilde{C}_{\kappa}$ corresponds to the constant involved in the discrete inf-sup condition \eqref{eq:discrete_inf_sup_Poisson} with $\xi$ being replaced by $\kappa$.

\begin{lemma}[$\mathcal{F}_{h}$ is well-defined] 
\label{lemma:welldefined_Fh}
Let $\mathbf{f}$ be such that  $\|\mathbf{f}\|_{\mathbf{L}^2(\Omega)}\leq \tilde{\mathfrak{C}}\nu_{-}$, where $\nu_{-}$ is as in \eqref{eq:nu}. Then, there exist $h_{\star}>0$ and $\epsilon>0$ such that the map $\mathcal{F}_{h}$ is well-defined on $\mathfrak{B}_{T}^{h}$, for all $0<h \leq h_{\star}$, whenever $2 - \epsilon \leq p < 2$. In addition, we have $\mathcal{F}_h(\mathfrak{B}_{T}^{h})\subset \mathfrak{B}_{T}^{h}$.
\end{lemma}
\begin{proof}
Let $\theta_{h}\in \mathfrak{B}_{T}^{h}$. In view of the discrete inf-sup condition \eqref{eq:discrete_inf_sup}, there exists a unique discrete pair $(\mathbf{u}_{h}(\theta_h),\mathsf{p}_h(\theta_h))\in \mathbf{X}_{h}\times Q_h$ solving problem \eqref{eq:discrete_darcy_system} \cite[Proposition 2.42]{Guermond-Ern}. Moreover, by testing $\mathbf{v}_h = \mathbf{u}_{h}(\theta_h)$ in the first equation of \eqref{eq:discrete_darcy_system}, we arrive at the following bound for the discrete velocity field $\mathbf{u}_{h}(\theta_h)$:
\begin{equation*}
\|\mathbf{u}_{h}(\theta_h)\|_{\mathbf{L}^2(\Omega)}
\leq
\nu_{-}^{-1}\|\mathbf{f}\|_{\mathbf{L}^2(\Omega)}
\leq
\tilde{\mathfrak{C}} = (2C_{e}\tilde{C}_{\kappa})^{-1}.
\end{equation*}
Consequently, $\mathbf{u}_{h}(\theta_h)$ is such that $\tilde{C}_{\kappa} C_{e}\|\mathbf{u}_{h}(\theta_h)\|_{\mathbf{L}^2(\Omega)} \leq 1/2$, i.e., $\mathbf{u}_{h}(\theta_h)$ satisfies \eqref{eq:disc_condicion} with $\alpha = 1/2$. We can thus utilize the results of Proposition \ref{prop:disc_wp_temp} to guarantee the existence of a unique $\zeta_h \in V_{h}$ solving \eqref{eq:discrete_zeta_h}. In addition, Proposition \ref{prop:disc_wp_temp} also yields $\zeta_{h} \in\mathfrak{B}^{h}_{T}$. This concludes the proof.
\end{proof}

We now provide the existence of discrete solutions via a fixed point argument. 

\begin{theorem}[existence]\label{thm:existence_discrete}
Under the assumptions of Lemma \ref{lemma:welldefined_Fh}, there exist $h_{\star}>0$ and $\epsilon>0$ such that \eqref{eq:model_discrete} admits a discrete solution $(\mathbf{u}_h,\mathsf{p}_h,T_h)\in \mathbf{X}_{h} \times Q_{h}\times V_{h}$ for all $0<h \leq h_{\star}$, whenever $2 - \epsilon \leq p < 2$. In addition, we have that $T_{h} \in \mathfrak{B}_{T}^{h}$.
\end{theorem}
\begin{proof}
Since we are in finite dimensions, we apply Brouwer's fixed point theorem \cite[Theorem 3.2]{MR787404}. To be able to invoke such a theorem, we only need to verify the continuity of $\mathcal{F}_h$. This is achieved by repeating the arguments utilized within the proof of Theorem \ref{thm:existence} in combination with the fact that, since we are in finite dimensions, we can pass from weak to strong convergence. 
\end{proof}

%%%%%%%%%%%%%%%%%%%%%%%%%%%%%%%%%%%%%%%%%%%%%%%%%%%%%%%%%%
%%%%%%%%%%%%%%%%%%%%%%%%%%%%%%%%%%%%%%%%%%%%%%%%%%%%%%%%%%
%%%%%%%%%%%%%%%%%%%%%%%%%%%%%%%%%%%%%%%%%%%%%%%%%%%%%%%%%%
%%%%%%%%%%%%%%%%%%%%%%%%%%%%%%%%%%%%%%%%%%%%%%%%%%%%%%%%%%

\subsection{Convergence}

We present the following convergence result.

\begin{theorem}[convergence]\label{thm:convergence}
Let $\mathbf{f}$ be such that $
\| \mathbf{f}\|_{\mathbf{L}^2(\Omega)}\leq \nu_{-}\min\{\mathfrak{C},\tilde{\mathfrak{C}}\},
$
where $\nu_{-}$ is given as in \eqref{eq:nu} and $\mathfrak{C}$ and $\tilde{\mathfrak{C}}$ are defined in \eqref{def:constant_C} and \eqref{def:constant_C_tilde}, respectively. Then, there exist $\epsilon >0$, $h_{\star} >0$, and a nonrelabeled subsequence $\{T_h\}_{0<h \leq h_{\star}}$ such that $T_h \rightharpoonup T$ in $W_0^{1,p}(\Omega)$, as $h\downarrow 0$, whenever $2 - \epsilon \leq p < 2$. The limit point $T$ solves problem \eqref{eq:heat_decoupled} and $(\mathbf{u}(T),\mathsf{p}(T))\in \mathbf{H}_0(\textnormal{div},\Omega)\times L_0^2(\Omega)$ solves \eqref{eq:Darcy_decoupled}.
\end{theorem}
\begin{proof}
In view of the assumption on $\mathbf{f}$, Theorem \ref{thm:existence_discrete} allows us to conclude that, for every $h>0$, the discrete coupled system \eqref{eq:model_discrete} admits at least a solution $(\mathbf{u}_h,\mathsf{p}_h,T_h)$. On the other hand, Theorem \ref{thm:existence_discrete} also guarantees that $\{T_h\}_{0<h \leq h_{\star}}$ is uniformly bounded in $W_0^{1,p}(\Omega)$ while \cite[inequalities (3.13)]{MR3802674} yield, for every $h>0$, the bounds
\begin{equation*}
\|\mathbf{u}_h\|_{\mathbf{L}^2(\Omega)}\leq \nu_{-}^{-1}\|\mathbf{f}\|_{\mathbf{L}^2(\Omega)}, 
\qquad \|\mathsf{p}_h\|_{L^2(\Omega)}\leq \tilde{\beta}^{-1}(1+\nu_{-}^{-1}\nu_{+})\|\mathbf{f}\|_{\mathbf{L}^2(\Omega)}.
\end{equation*}
Consequently, we have that (up to a subsequence) $(\mathbf{u}_h,\mathsf{p}_h,T_h) \rightharpoonup (\mathbf{u},\mathsf{p},T)$ in $\mathbf{L}^2(\Omega)\times L^2(\Omega)\times W_0^{1,p}(\Omega)$, as $h\downarrow 0$, whenever $2 - \epsilon \leq p < 2$.

In what follows we prove that
\begin{enumerate}[label=\textbf{S.\arabic*}]
\renewcommand{\labelenumi}{\textbf{\theenumi}}
\renewcommand{\theenumi}{S.\arabic{enumi}}
\item[(i)] \label{item_i} 
$(\mathbf{u}_h,\mathsf{p}_h)\rightharpoonup (\mathbf{u},\mathsf{p})$ in $\mathbf{H}_0(\textnormal{div},\Omega)\times L_0^2(\Omega)$, as $h\downarrow 0$, and that
\item[(ii)] \label{item_ii} 
$(\mathbf{u},\mathsf{p},T)=(\mathbf{u}(T),\mathsf{p}(T),T)\in \mathbf{H}_0(\textnormal{div},\Omega)\times L_0^2(\Omega) \times W_0^{1,p}(\Omega)$ solves system \eqref{eq:modelweak} or, equivalently, problems \eqref{eq:heat_decoupled} and \eqref{eq:Darcy_decoupled}.
\end{enumerate}

We first prove (i). Let $\mathsf{q}\in C_0^{\infty}(\Omega)$. We invoke the weak convergence $\mathbf{u}_h\rightharpoonup \mathbf{u}$ in $\mathbf{L}^2(\Omega)$ to immediately arrive at
\[
\int_\Omega q\text{ div } \mathbf{u}_h \mathsf{d}\mathbf{x}
=
- \int_\Omega \nabla q \cdot \mathbf{u}_h \mathsf{d}\mathbf{x} 
\rightarrow 
- \int_\Omega \nabla q \cdot \mathbf{u} \mathsf{d}\mathbf{x}
= 
\int_\Omega q \text{ div } \mathbf{u} \mathsf{d}\mathbf{x},
\quad
h \downarrow 0.
\]
Consequently, $\mathbf{u}_h \rightharpoonup \mathbf{u}$ in $\mathbf{H}(\textnormal{div},\Omega)$ as $h \downarrow 0$. The continuity of the normal trace operator \cite[Theorem 2.5]{MR851383} implies that $\mathbf{u} \cdot \mathbf{n}|_{\partial \Omega}=0$ and thus that $\mathbf{u} \in \mathbf{H}_0(\textnormal{div},\Omega)$. The fact that $\mathsf{p} \in L_0^2(\Omega)$ is trivial.

The rest of the proof is dedicated to prove (ii). Let us start by proving that $(\mathbf{u}(T),\mathsf{p}(T))=(\mathbf{u},\mathsf{p})\in\mathbf{H}_0(\textnormal{div},\Omega)\times L_0^2(\Omega)$ solves Darcy's system \eqref{eq:Darcy_decoupled}. 

Let  $\mathbf{v}\in\mathbf{C}_0^\infty(\Omega)$ and let $\mathbf{v}_h\in \mathbf{X}_h$. A simple computation reveals that
\begin{multline*}
\mathfrak{I}:=
\left|\int_\Omega \left(\nu(T_h)\mathbf{u}_h\cdot\mathbf{v}_h \!- \!\mathsf{p}_h \text{ div } \mathbf{v}_h \!-\! \nu(T)\mathbf{u}\cdot\mathbf{v} \!+\! \mathsf{p} \text{ div } \mathbf{v}\right)\! \mathsf{d}\mathbf{x}\right|  \\
\leq \left|\int_\Omega (\nu(T_h)\mathbf{u}_h - \nu(T)\mathbf{u})\cdot\mathbf{v} \,\mathsf{d}\mathbf{x} \right|  
+
\left|\int_\Omega \nu(T_h)\mathbf{u}_h\cdot(\mathbf{v}_h - \mathbf{v})\mathsf{d}\mathbf{x}  \right|
+ \left|\int_\Omega  (\mathsf{p} - \mathsf{p}_h) \text{ div } \mathbf{v} \,\mathsf{d}\mathbf{x}  \right|\\
 + \left|\int_\Omega  \mathsf{p}_h \text{ div } (\mathbf{v} - \mathbf{v}_h) \mathsf{d}\mathbf{x}  \right| = \mathrm{I} + \mathrm{II} + \mathrm{III} + \mathrm{IV}.   
\end{multline*}
The density results stated in \eqref{eq:density_result} immediately reveal that $\mathrm{II}, \mathrm{IV} \rightarrow$ 0 as $h \downarrow 0$. Since $\text{ div }\mathbf{v} \in \mathbf{L}^2(\Omega)$ and $\mathsf{p}_h \rightharpoonup \mathsf{p}$ in $L_0^2(\Omega)$, it is also immediate that $\mathrm{III} \rightarrow 0$ as $h\downarrow 0$. To control the term $\mathrm{I}$, we first notice that
\[
\int_\Omega (\nu(T_h)\mathbf{u}_h - \nu(T)\mathbf{u})\cdot\mathbf{v} \,\mathsf{d}\mathbf{x}
= 
\int_\Omega (\nu(T_h) - \nu(T)) \mathbf{u}_h \cdot \mathbf{v} \mathsf{d}\mathbf{x} 
+
\int_\Omega \nu(T) (\mathbf{u}_h - \mathbf{u})\cdot\mathbf{v} \,\mathsf{d}\mathbf{x}.
\]
Since the embedding $W_0^{1,p}(\Omega)\hookrightarrow L^{q}(\Omega)$ is compact for $q<2p/(2-p)$ \cite[Theorem 6.3, Part \textbf{I}]{MR2424078} and $\nu$ is continuous and uniformly bounded, we have that  $\nu(T_h)\rightarrow \nu(T)$ in $L^{q}(\Omega)$, as $h \downarrow 0$, for $q<2p/(2-p)$ \cite[Theorem 7]{MR120342}. This and the weak convergence $\mathbf{u}_h \rightharpoonup \mathbf{u}$ in $\mathbf{L}^2(\Omega)$ reveal that $\mathrm{I} \rightarrow 0$ as $h \downarrow 0$. Consequently, $\mathfrak{I} \rightarrow 0$ as $h \downarrow 0$, which reveals that the limit point $(\mathbf{u},\mathsf{p})\in\mathbf{H}_0(\textnormal{div},\Omega)\times L_0^2(\Omega)$ solves the first equation in \eqref{eq:Darcy_decoupled}. To prove that the velocity field $\mathbf{u}$ satisfies the second equation in \eqref{eq:Darcy_decoupled}, we let $\mathsf{q}\in C_0^\infty(\Omega)$ with zero mean and $\mathsf{q}_h\in Q_h$ be its $L^2$--projection onto $Q_h$. Hence, in view of the strong convergence $\mathsf{q}_h\to \mathsf{q}$ in $L_0^2(\Omega)$, we obtain
\begin{equation*}
\left|\int_\Omega  (\mathsf{q}\text{ div } \mathbf{u} - \mathsf{q}_h\text{ div } \mathbf{u}_h)\mathsf{d}\mathbf{x} \right|
\leq 
\left|\int_\Omega  (\mathsf{q}-\mathsf{q}_h)\text{ div } \mathbf{u}\, \mathsf{d}\mathbf{x}\right| + \left|\int_\Omega  \mathsf{q}_h\text{ div } (\mathbf{u} - \mathbf{u}_h)\mathsf{d}\mathbf{x} \right| \rightarrow 0,
\end{equation*}
as $h\downarrow 0$. Consequently, the pair $(\mathbf{u},\mathsf{p})$ solves \eqref{eq:Darcy_decoupled}.

It remains to prove that $T\in W_0^{1,p}(\Omega)$ solves \eqref{eq:heat_decoupled} with $\mathbf{u}=\mathbf{u}(T)$. To accomplish this task, we let $S\in C_0^{\infty}(\Omega)$ and $S_h\in V_h$. Set $S_h=I_h S$, utilize H\"older's inequality, the assumptions on $\kappa$, the Lebesgue dominated convergence, and standard properties of the interpolation operator $I_h$ to obtain
\begin{multline*}
\left|\int_\Omega (\kappa(T)\nabla T\cdot \nabla S - \kappa(T_h)\nabla T_h\cdot \nabla S_h)\mathsf{d}\mathbf{x}\right| 
\leq
\left|\int_\Omega (\kappa(T) - \kappa(T_h))\nabla T\cdot \nabla S\mathsf{d}\mathbf{x}\right| \\
+\left|\int_\Omega \kappa(T_h)\nabla(T - T_h)\cdot \nabla S\mathsf{d}\mathbf{x}\right|+\left|\int_\Omega \kappa(T_h)\nabla T_h\cdot \nabla (S-S_h)\mathsf{d}\mathbf{x}\right| \rightarrow 0, \text{ as } h\downarrow 0.
\end{multline*}
Finally, to prove that $\int_\Omega T_h \mathbf{u}_h \cdot \nabla S_h\mathsf{d}\mathbf{x} \to  \int_\Omega T \mathbf{u} \cdot \nabla S\mathsf{d}\mathbf{x}$, as $h\downarrow 0$, we invoke similar arguments to those developed in the proof of Theorem \ref{thm:existence} and the convergence result $\sqrt{\nu(T_h)}\mathbf{u}_h(T_h) \to \sqrt{\nu(T)}\mathbf{u}(T)$ in $\mathbf{L}^2(\Omega)$, as $h\downarrow 0$, which follows from (2.23) in \cite[Lemma 2.1]{MR3802674}. This proves that $T\in W_0^{1,p}(\Omega)$ solves \eqref{eq:heat_decoupled} with $\mathbf{u}=\mathbf{u}(T)$ and concludes the proof.
\end{proof}

%%%%%%%%%%%%%%%%%%%%%%%%%%%%%%%%%%%%%%%%%%%%%%%%%%%%%%%%%%
%%%%%%%%%%%%%%%%%%%%%%%%%%%%%%%%%%%%%%%%%%%%%%%%%%%%%%%%%%
%%%%%%%%%%%%%%%%%%%%%%%%%%%%%%%%%%%%%%%%%%%%%%%%%%%%%%%%%%
%%%%%%%%%%%%%%%%%%%%%%%%%%%%%%%%%%%%%%%%%%%%%%%%%%%%%%%%%%
%%%%%%%%%%%%%%%%%%%%%%%%%%%%%%%%%%%%%%%%%%%%%%%%%%%%%%%%%%
%%%%%%%%%%%%%%%%%%%%%%%%%%%%%%%%%%%%%%%%%%%%%%%%%%%%%%%%%%
%%%%%%%%%%%%%%%%%%%%%%%%%%%%%%%%%%%%%%%%%%%%%%%%%%%%%%%%%%
%%%%%%%%%%%%%%%%%%%%%%%%%%%%%%%%%%%%%%%%%%%%%%%%%%%%%%%%%%

\section{A posteriori error analysis}\label{sec:a_posteriori_anal}

In this section, we devise and analyze an a posteriori error estimator for the coupled system \eqref{eq:modelweak}. We obtain a global reliability estimate and investigate local efficiency results. To perform an analysis, in addition to the assumptions stated in \eqref{eq:nu}, we shall require that: 

$\bullet$ The thermal diffusivity $\kappa$ is a positive constant, and

$\bullet$ The forcing term $\mathbf{f}\in \mathbf{H}^1(\Omega)$ and $\|\mathbf{f}\|_{\mathbf{L}^2(\Omega)}\leq \nu_{-}\min\{\mathfrak{C},\tilde{\mathfrak{C}}\}$; cf. Theorem \ref{thm:convergence}. 

In what follows we comment on the assumption $\mathbf{f}\in \mathbf{H}^1(\Omega)$ (see also the discussion in \cite[Section 4]{MR4041519}): To obtain the identity \eqref{eq:error_equation_darcy_1} we utilize the Green's formula of \cite[Theorem 2.11, Chapter I]{MR851383} on each element $K \in \mathscr{T}$. \cite[Theorem 2.11, Chapter I]{MR851383} also guarantees that the tangential trace $\mathfrak{T}_{\boldsymbol{\tau}}: \boldsymbol{v} \rightarrow \boldsymbol{v} \cdot \boldsymbol{\tau}|_{\partial\mathcal{O}}$ is a linear and continuous operator from $\mathbf{H}( \textbf{curl},\mathcal{O})$ into $H^{-1/2}(\mathcal{O})$ for any Lipschitz domain $\mathcal{O}$. The additional regularity $\mathbf{f} \in \mathbf{H}( \textbf{curl},\Omega)$ would thus seem sufficient. However, in order to have a local and integral representation of the residual on the interior sides $\gamma$ we assume that $\mathbf{f}\in \mathbf{H}^1(\Omega)$ so that the interelement residuals $\llbracket (\mathbf{f}-\nu(T_h)\mathbf{u}_h)\cdot\boldsymbol\tau\rrbracket$, defined in \eqref{eq:jump_darcy}, are well-defined in $L^2(\gamma)$.
% \EO{We immediately comment on the assumption $\mathbf{f}\in \mathbf{H}^1(\Omega)$ \cite[Section 4]{MR4041519}. 
% To simplify the exposition we restrict ourselves to the case $d=3$ and introduce the following Green's formula, which holds for any Lipschitz domain $\mathcal{O}$ and functions $\varphi \in \mathbf{H}^1(\mathcal{O})$ and $\mathbf{v} \in H(\textbf{curl},\mathcal{O})$,
% \[
%  \langle \mathbf{v} \times \mathbf{n}, \varphi \rangle_{\partial \mathcal{O}} =  \int_{\mathcal{O}} \mathbf{v} \cdot \textbf{curl } \varphi \mathrm{d}\mathbf{x} - \int_{\mathcal{O}} \varphi \cdot \textbf{curl } \mathbf{v} \mathrm{d}\mathbf{x},
% \]
% where $\partial \mathcal{O}$ denotes the boundary of $\mathcal{O}$ and the duality reduces to the surface integral when $\mathbf{v}$ is smoother; e.g., $\mathbf{v} \in \mathbf{H}^1(\mathcal{O})$. To perform an a posteriori error analysis, we apply the previously stated Green’s formula on elements of a mesh $\T_h$, which leads to jumps of tangential components on each face.
% ...
% In principle, the additional regularity $\mathbf{f} \in H(\textbf{curl},\mathcal{O})$ would seem sufficient, but considering that the previous Green’s formula have to be applied below in each element, thus leading to jumps of tangential components on each face, it is much simpler to assume that $\mathbf{f}\in \mathbf{H}^1(\Omega)$.
% }

The existence of a solution $(\mathbf{u},\mathsf{p},T) \in \mathbf{H}_0(\textnormal{div},\Omega)\times L_0^2(\Omega)\times W_0^{1,p}(\Omega)$ to system \eqref{eq:modelweak} is guaranteed by Theorem \ref{thm:existence} for $p<2$. Theorem \ref{thm:existence_discrete} guaantees the existence of $\epsilon >0$ and $h_{\star}$ such that the discrete problem \eqref{eq:model_discrete}
admits a solution $(\mathbf{u}_h,\mathsf{p}_h,T_h)\in\mathbf{X}_{h}\times Q_{h}\times V_{h}$ for every $0<h \leq h_{\star}$ and $2-\epsilon \leq p <2$. Within our a posteriori error analysis setting, since we will not be dealing with uniform refinement, the parameter $h$ does not bear the meaning of a mesh size. It can thus be thought as $h = 1/k$, where $k\in\mathbb{N}$ is the index set in a sequence of refinements of an initial mesh or partition $\mathscr{T}_{0}$.

%%%%%%%%%%%%%%%%%%%%%%%%%%%%%%%%%%%%%%%%%%%%%%%%%%%%%%%%%%
%%%%%%%%%%%%%%%%%%%%%%%%%%%%%%%%%%%%%%%%%%%%%%%%%%%%%%%%%%
%%%%%%%%%%%%%%%%%%%%%%%%%%%%%%%%%%%%%%%%%%%%%%%%%%%%%%%%%%
%%%%%%%%%%%%%%%%%%%%%%%%%%%%%%%%%%%%%%%%%%%%%%%%%%%%%%%%%%

\subsection{A posteriori error estimators}

In this section, we devise an a posteriori error estimator for the finite element approximation \eqref{eq:model_discrete} of system \eqref{eq:modelweak}. The proposed error estimator will be decomposed as the sum of three individual contributions: one contribution that accounts for the discretization of the heat equation with convection and two contributions related to the discretization of Darcy's system.

Let us begin our analysis by introducing some notation. Let  $\mathbf{w}_{h}$ be a discrete tensor valued function and let $\gamma \in \mathscr{S}$ be an internal side. We define the \emph{jump} or \emph{interelement residual} of $\mathbf{w}_{h}$ on $\gamma$ by
\begin{equation*}
\llbracket \mathbf{w}_{h} \cdot \mathbf{n} \rrbracket:= \mathbf{n}^{+} \cdot \mathbf{w}_{h}|^{}_{K^{+}} + \mathbf{n}^{-} \cdot \mathbf{w}_{h}|^{}_{K^{-}},
\end{equation*}
where $\mathbf{n}^{+}$ and $\mathbf{n}^{-}$ denote the unit normals to $\gamma$ pointing towards $K^{+}$ and $K^{-}$, respectively; $K^{+}$, $K^{-} \in \T_h$ are such that $K^{+} \neq K^{-}$ and $\partial K^{+} \cap \partial K^{-} = \gamma$. Similarly, 
\begin{equation*}
\llbracket \mathbf{w}_{h} \cdot \boldsymbol{\tau} \rrbracket:= \boldsymbol{\tau}^{+} \cdot \mathbf{w}_{h}|^{}_{K^{+}} + \boldsymbol{\tau}^{-} \cdot \mathbf{w}_{h}|^{}_{K^{-}},
\end{equation*}
where $\boldsymbol\tau^{+}$ and $\boldsymbol\tau^{-}$ denote the unit tangents to $\gamma$; cf. Figure \ref{fig:normal_and_tangent}. Notice that $\boldsymbol\tau^{+}$ and $\mathbf{n}^{+}$ are orthogonal; similarly $\boldsymbol\tau^{-}$ and $\mathbf{n}^{-}$.

\begin{figure}[!ht]
\centering
\includegraphics[width=7.5cm,height=2.0cm,scale=0.33]{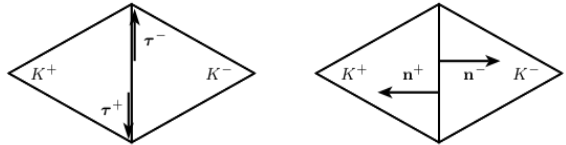}
\caption{Representation of the tangent and normal vectors, respectively, on $K^{+}$ and $K^{-}$.}
\label{fig:normal_and_tangent}
\end{figure}

\subsubsection{Heat equation with convection: local indicators and a posteriori error estimator}
Let $K \in \mathscr{T}_h$ be a simplex and $\gamma \in \mathscr{S}$ be an internal side. We define the element residual $\mathpzc{R}_K$ and the interelement residual $\mathpzc{J}_{\gamma}$ as
\begin{equation}
\mathpzc{R}_K:= (-\nabla T_h \cdot \mathbf{u}_h - T_h \text{div }\mathbf{u}_h)|_{K},
\qquad
\mathpzc{J}_{\gamma}:= \llbracket(\kappa\nabla T_h -T_h \mathbf{u}_h)\cdot \mathbf{n} \rrbracket.
\label{eq:residual_heat}
\end{equation}
With the residuals $\mathpzc{R}_K$ and $\mathpzc{J}_{\gamma}$ at hand, we define a local indicator $\mathpzc{E}_{p,K}$ associated to the underlying finite element discretization of the heat equation on the basis of three scenarios. First, if $z \in K$ and $z$ is not a vertex of $K$, then
\begin{equation}
\label{eq:local_indicator_temp_I}
\mathpzc{E}_{p,K}:=
\left( 
h_K^{2-p} 
+
h_K^p \| \mathpzc{R}_K\|_{L^p(K)}^p 
+
h_K \| \mathpzc{J}_{\gamma} \|_{L^p(\partial K \setminus \partial \Omega)}^p
\right)^{\frac{1}{p}}.
\end{equation}
Second, if $z \in K$ and $z$ is a vertex of $K$, then
\begin{equation}\label{eq:local_indicator_temp_II}
\mathpzc{E}_{p,K}:=
\left( 
h_K^p \| \mathpzc{R}_K  \|_{L^p(K)}^p 
+
h_K \|
\mathpzc{J}_{\gamma}
\|_{L^p(\partial K \setminus \partial \Omega)}^p
\right)^{\frac{1}{p}}.
\end{equation}
Third, if $z \notin K$, then the indicator $\mathpzc{E}_{p,K}$ is defined as in \eqref{eq:local_indicator_temp_II}. 

The following comments are now in order. We first recall that we consider our elements $K$ to be closed sets. On the other hand, the Lagrange interpolation operator $I_h$ is well-defined over the space $W_0^{1,p\prime}(\Omega)$ with $p\prime>2$. Since $I_h$ is constructed by matching the point values at the Lagrange nodes, we have the basic property
\[
(S -I_{h}S)(z) = 0
\qquad
\forall S \in W_0^{1,p\prime}(\Omega), \, p\prime>2, \forall K \in \T_h.
\]
Here $z$ denotes a vertex of $K$. This simple observation points in the direction of explaining the discrepancy between definitions \eqref{eq:local_indicator_temp_I} and \eqref{eq:local_indicator_temp_II}.

With the previous indicators at hand, we define the corresponding error estimator
\begin{equation}\label{eq:error_estimator_heat}
\mathpzc{E}_{p,\T}:=\left( \sum_{K\in\mathscr{T}} \mathpzc{E}_{p,K}^p\right)^{\frac{1}{p}},\qquad p<2.
\end{equation}

\subsubsection{Darcy's problem: local indicators and a posteriori error estimator}
The devising and analysis of residual-type a posteriori error estimates for Raviart--Thomas finite element approximations of Darcy's problem are not as simple as for the Laplace's equation. In fact, two difficulties prevent the success of the straightforward application of frequently used arguments. First, traces of functions in $\mathbf{H}(\text{div},\Omega)$ are only contained in $H^{-1/2}(\partial \Omega)$ \cite[Theorem 2.5]{MR851383} and, for an internal side $\gamma$, the corresponding jump term $\llbracket p_h \mathbb{I} \mathbf{n} \rrbracket$ does not belong to $H^{1/2}(\gamma)$ \cite[Section III.3.3]{MR1115205}; $\mathbb{I}$ denotes the identity matrix. The second difficulty is given the fact that the space $\mathbf{H}(\text{div},\Omega)$ is anisotropic \cite{MR1427472}; for solenoidal functions the $\mathbf{H}(\text{div},\Omega)$-norm and the $L^2(\Omega)$-norm coincide. The authors of \cite{MR1427472} circumvent these difficulties and devise an a posteriori error estimator that is reliable and efficient in suitable mesh-dependent norms; see \cite[inequality (3.12)]{MR1427472} and \cite[Theorem 3.3]{MR1427472}. However, the devised error estimator for the natural norm in $\mathbf{H}(\text{div},\Omega) \times L^2(\Omega)$ is only reliable; the derived efficiency estimate \cite[inequality (4.20)]{MR1427472}, that involves a negative power of $h$, is not optimal. Later, the author of \cite{MR1408371} provides reliable and efficient a posteriori error estimates in the natural norm in $\mathbf{H}(\text{div},\Omega) \times L^2(\Omega)$; the difficulties arising from the anisotropy of the norm are circumvented by utilizing a Helmholtz decomposition of square-integrable tensors.

Inspired by the developments in \cite[Section 4]{MR4041519}, and in view of the assumption that $\mathbf{f}\in \mathbf{H}^{1}(\Omega)$, we apply the $\textbf{curl}$ operator to the  the first equation of Darcy's system \eqref{eq:Darcy_decoupled}, in its strong form, to obtain
\begin{equation*}
\textbf{curl}(\nu(T)\mathbf{u}) = \textbf{curl }\mathbf{f} \quad \text{in }\Omega.
\end{equation*}

Let $K \in \T_h$ and $\gamma \in \mathscr{S}$. We define the element and interelement residuals
\begin{align}
&\mathfrak{R}_{K} := \mathbf{curl}(\mathbf{f}-\nu(T_h)\mathbf{u}_h), 
&& \mathfrak{J}_{\gamma}:= \llbracket (\mathbf{f}-\nu(T_h)\mathbf{u}_h)\cdot\boldsymbol\tau\rrbracket,
\label{eq:residuals_Darcy_1}
\\
& \mathscr{R}_{K}:=\mathbf{f}-\nu(T_h)\mathbf{u}_h-\nabla \mathsf{p}_h, 
&& \mathscr{J}_{\gamma}:= \llbracket \mathsf{p}_h\mathbf{n}\rrbracket.
\label{eq:residuals_Darcy_2}
\end{align}
With these ingredients at hand, we define local error indicators
\begin{align}\label{eq:local_indicator_darcy_I}
\mathfrak{E}_{K}:&=\left(h_K^2\|\mathfrak{R}_{K}\|_{L^2(K)}^2 + h_K\|\mathfrak{J}_{\gamma}\|_{L^2(\partial K\setminus \partial\Omega)}^2\right)^{\frac{1}{2}},
\\
\mathscr{E}_{K}:&=\left(h_K^2\| \mathscr{R}_{K} \|_{\mathbf{L}^2(K)}^2+h_K\|\mathscr{J}_{\gamma}\|_{\mathbf{L}^2(\partial K\setminus \partial\Omega)}^2\right)^{\frac{1}{2}},
\label{eq:local_indicator_darcy_II}
\end{align}
and a posteriori error estimators associated to Darcy's system as
\begin{equation}\label{eq:error_estimators_Darcy}
\mathfrak{E}_{\T}:=\left( \sum_{K\in\mathscr{T}} \mathfrak{E}_{K}^2\right)^{\frac{1}{2}}, \qquad \mathscr{E}_{\T}:=\left( \sum_{K\in\mathscr{T}} \mathscr{E}_{K}^2\right)^{\frac{1}{2}}.
\end{equation}

%%%%%%%%%%%%%%%%%%%%%%%%%%%%%%%%%%%%%%%%%%%%%%%%%%%%%%%%%%
%%%%%%%%%%%%%%%%%%%%%%%%%%%%%%%%%%%%%%%%%%%%%%%%%%%%%%%%%%
%%%%%%%%%%%%%%%%%%%%%%%%%%%%%%%%%%%%%%%%%%%%%%%%%%%%%%%%%%
%%%%%%%%%%%%%%%%%%%%%%%%%%%%%%%%%%%%%%%%%%%%%%%%%%%%%%%%%%

\subsection{Reliability estimates}\label{sec:reliability_estimates}

In this section, we obtain a global reliability estimate for the total a posteriori error estimator
\begin{equation}\label{def:total_error_estimator}
\mathsf{E}_{\T}:=\mathpzc{E}_{p,\T}+\mathfrak{E}_{\T}+\mathscr{E}_{\T}.
\end{equation}

The following result is instrumental to perform our analysis.

\begin{lemma}[auxiliary result]\label{lemma:aux_curl_result}
Let $\mathbf{v}\in\mathbf{V}(\Omega)$, where $\mathbf{V}(\Omega)$ is defined as in Section \ref{sec:notation}. Then, there exists a unique function $\vartheta\in H_0^1(\Omega)$ such that
\begin{equation*}
\mathbf{v} = \mathbf{curl~}\vartheta,
\qquad
 \|\vartheta\|_{H^1(\Omega)}\lesssim \|\mathbf{v}\|_{\mathbf{L}^2(\Omega)}.
\end{equation*}
The hidden constant is independent of $\vartheta$ and $\mathbf{v}$.
\end{lemma}
\begin{proof}
The proof of the existence of a function $\vartheta \in H^1(\Omega)$ satisfying the desired properties can be found in \cite[Chapter I, Theorem 3.1]{MR851383}. The correction of $\vartheta \in H^1(\Omega)$ satisfying the needed boundary condition is available in \cite[Appendix A]{MR4041519}.
\end{proof}

As a final preparatory step, we define the discretization errors associated to the temperature, the velocity, and the pressure, respectively, as follows:
\[
e_{T} := T-T_h,
\qquad
\mathbf{e}_{\mathbf{u}} := \mathbf{u}-\mathbf{u}_h,
\qquad
e_{\mathsf{p}} := \mathsf{p}-\mathsf{p}_h.
\]

We are now ready to enunciate and prove the main result of this section.

\begin{theorem}[global reliability]\label{thm:global_reliability}
Let $(\mathbf{u},\mathsf{p},T) \in \mathbf{H}_0(\textnormal{div},\Omega)\times L_0^2(\Omega)\times W_0^{1,p}(\Omega)$ be a solution to \eqref{eq:modelweak} with a forcing term $\mathbf{f}$, which is such that $\|\mathbf{f}\|_{\mathbf{L}^2(\Omega)}\leq \nu_{-}\min\{\mathfrak{C},\tilde{\mathfrak{C}}\}$. Let $(\mathbf{u}_h,\mathsf{p}_h,T_h) \in \mathbf{X}_{h}\times Q_{h}\times V_{h}$ be a solution to the discrete system \eqref{eq:model_discrete} for $0<h \leq h_{\star}$. Assume that
\begin{equation}\label{eq:discrete_assump}
\mathfrak{C}^{-1}\|\mathbf{e}_{\mathbf{u}} \|_{\mathbf{L}^{2}(\Omega)} \leq \rho,
\qquad
0 < \rho < 1,
\end{equation}
where $\mathfrak{C}=(2C_{e}C_{\kappa})^{-1}$, $C_e$ is the best constant in the Sobolev embedding $W_0^{1,p}(\Omega)\hookrightarrow L^{2p/(2-p)}(\Omega)$, and $C_{\kappa}$ is defined as in \eqref{eq:estimatesheat} with $\alpha = \frac{1}{2}$. Assume, in addition, that
\begin{equation}\label{eq:global_rel_assump}
C_{\mathcal{L}}\|\delta_z\|_{W^{-1,p}(\Omega)} \|\mathbf{u}\|_{\mathbf{L}^{p\prime}(\Omega)} \leq\varrho (1-\rho)\nu^{}_{-}\mathfrak{C}^2,\qquad 0<\varrho<1,
\end{equation}
where $\frac{1}{p}+\frac{1}{p'}=1$ and $C_{\mathcal{L}}$ denotes the Lipschitz constant of $\nu$. Then
\begin{equation}\label{eq:upper_bound}
\|\nabla e_{T}\|_{\mathbf{L}^p(\Omega)} + \|\mathbf{e}_{\mathbf{u}}\|_{\mathbf{L}^2(\Omega)}+\|e_{\mathsf{p}}\|_{L^2(\Omega)} 
\lesssim
\mathsf{E}_{\T}.
\end{equation}
Here, the hidden constant is independent of continuous and discrete variables, the size of the elements in the mesh $\mathscr{T}_h$, and $\#\mathscr{T}_h$.
\end{theorem}
\begin{proof} 
We begin by controlling the temperature error $\|\nabla e_{T}\|_{\mathbf{L}^p(\Omega)}$. To accomplish this task, we invoke equation \eqref{eq:heat_decoupled}, an elementwise integration by parts formula, and Galerkin orthogonality to obtain
\begin{multline}
\label{eq:error_equation_temp}
\int_\Omega \left( \kappa\nabla e_{T} + e_{T} \mathbf{e}_\mathbf{u}  -e_T \mathbf{u}  - T\mathbf{e}_{\mathbf{u}}\right)\cdot\nabla S\mathsf{d}\mathbf{x} = \langle \delta_{z},S-I_{h}S\rangle \\
+\sum_{K\in\mathscr{T}}\int_{K}\mathpzc{R}_K(S-I_{h}S)\mathsf{d}{\mathbf{x}} 
+ 
\sum_{\gamma\in\mathscr{S}}\int_\gamma \mathpzc{J}_{\gamma}(S-I_{h}S)\mathsf{d}s
\end{multline}
for every $S\in W_0^{1,p\prime}(\Omega)$. The element and interelement residuals, $\mathpzc{R}_K$ and $\mathpzc{J}_{\gamma}$, respectively, are defined in \eqref{eq:residual_heat}. Hence, the inf-sup condition \eqref{eq:inf_sup_Poisson}, H\"older's inequality, and the local interpolation bounds \eqref{eq:interp_Lagrange} and \eqref{eq:interp_Lagrange_side} yield
\begin{multline*}
\|\nabla e_{T}\|_{\mathbf{L}^{p}(\Omega)}
\leq C_{\kappa}\sup_{S\in W_0^{1,p'}(\Omega)}\frac{1}{\|\nabla S\|_{\mathbf{L}^{p'}(\Omega)}}\int_\Omega\kappa\nabla e_{T}\cdot \nabla S\mathsf{d}\mathbf{x} 
\\ 
\leq C_{\kappa} \left( \|e_{T}\|_{L^{\frac{2p}{2-p}}(\Omega)}( \|\mathbf{e}_{\mathbf{u}}\|_{\mathbf{L}^2(\Omega)} + \|\mathbf{u}\|_{\mathbf{L}^2(\Omega)} ) +  \|T\|_{L^{\frac{2p}{2-p}}(\Omega)} \|\mathbf{e}_{\mathbf{u}}\|_{\mathbf{L}^2(\Omega)} + C\mathpzc{E}_{p,\T} \right),
\end{multline*} 
where $C>0$. In view of this estimate, the embedding $W_0^{1,p}(\Omega)\hookrightarrow L^{\frac{2p}{2-p}}(\Omega)$, the bound $\|\mathbf{u}\|_{\mathbf{L}^2(\Omega)}\leq \mathfrak{C}=(2C_{e}C_{\kappa})^{-1}$, and the assumption \eqref{eq:discrete_assump}, we deduce that
\begin{equation}\label{eq:temp_estimate}
(1-\rho)\|\nabla e_{T}\|_{\mathbf{L}^{p}(\Omega)} 
\leq 
4C_{\kappa}^{2}C_{e} \|\delta_z\|_{W^{-1,p}(\Omega)}\|\mathbf{e}_{\mathbf{u}}\|_{\mathbf{L}^2(\Omega)} + 2CC_{\kappa}\mathpzc{E}_{p,\T},
\end{equation}
where $C_e$ denotes the best constant in the embedding $W_0^{1,p}(\Omega)\hookrightarrow L^{\frac{2p}{2-p}}(\Omega)$ and $C_{\kappa}$ is defined as in \eqref{eq:estimatesheat} with $\alpha = \frac{1}{2}$.

We now control the velocity error. 
% $\|\mathbf{e}_{\mathbf{u}}\|_{\mathbf{L}^2(\Omega)}$. 
Let $\mathbf{v} \in \mathbf{V}(\Omega)$, where $\mathbf{V}(\Omega)$ is defined in Section \ref{sec:notation}. 
An application of Lemma \ref{lemma:aux_curl_result} yields the existence of a unique function $\vartheta\in H_0^1(\Omega)$ such that $\mathbf{v} = \mathbf{curl~}\vartheta$ and $ \|\vartheta\|_{H^1(\Omega)}\lesssim \|\mathbf{v}\|_{\mathbf{L}^2(\Omega)}$. 
% Set $q=0$. 
Utilize the pair $(\mathbf{v},0)$ as a test pair in problem \eqref{eq:Darcy_decoupled} to deduce, in view of an elementwise integration by parts formula based on \cite[Theorem 2.11, Chapter I]{MR851383} and Galerkin orthogonality, the identity 
\begin{multline}\label{eq:error_equation_darcy_1}
\int_\Omega \Big(\nu(T_h)\mathbf{e}_{\mathbf{u}}+(\nu(T)-\nu(T_h))\mathbf{u}\Big)\cdot\mathbf{v}\mathsf{d}\mathbf{x} \\
= \sum_{K\in\mathscr{T}}\int_K \mathfrak{R}_{K}(\vartheta-\mathcal{I}_{h}\vartheta)\mathsf{d}\mathbf{x} 
+ \sum_{\gamma\in\mathscr{S}}\int_\gamma \mathfrak{J}_{\gamma}(\vartheta-\mathcal{I}_{h}\vartheta)\mathsf{d}s.
\end{multline}
Here, $\mathcal{I}_{h}:H_0^1(\Omega) \to V_h$ denotes the Cl\'ement interpolation operator \cite{MR2373954,MR0520174}. We recall that the element and interelement residuals, $\mathfrak{R}_{K}$ and $\mathfrak{J}_{\gamma}$, respectively, are defined in \eqref{eq:residuals_Darcy_1}. Since $\nu$ is uniformly bounded, namely, $\nu$ satisfies \eqref{eq:nu}, \eqref{eq:error_equation_darcy_1} in combination with H\"older's inequality, standard interpolation error estimates for $\mathcal{I}_{h}$, and the estimate $\|\vartheta\|_{H^1(\Omega)}\lesssim \|\mathbf{v}\|_{\mathbf{L}^2(\Omega)}$, which follows from Lemma \ref{lemma:aux_curl_result}, allow us to obtain that
\begin{equation}\label{eq:estimate_vel}
\nu_{-}\|\mathbf{e}_{\mathbf{u}}\|_{\mathbf{L}^2(\Omega)}\leq C_{\mathcal{L}}C_{e}\|\nabla e_{T}\|_{\mathbf{L}^p(\Omega)}\|\mathbf{u}\|_{\mathbf{L}^{p\prime}(\Omega)} + C\mathfrak{E}_{\T}, 
\qquad
C>0.
\end{equation}
Replacing this estimate into \eqref{eq:temp_estimate} immediately yields
$\|\nabla e_{T}\|_{\mathbf{L}^{p}(\Omega)} \lesssim \mathpzc{E}_{p,\T}+\mathfrak{E}_{\T}$, upon utilizing assumption \eqref{eq:global_rel_assump}. This bound combined with estimate \eqref{eq:estimate_vel} and assumption \eqref{eq:global_rel_assump} yield the a posteriori error estimate $\|\mathbf{e}_{\mathbf{u}}\|_{\mathbf{L}^2(\Omega)} \lesssim \mathpzc{E}_{p,\T}+\mathfrak{E}_{\T}$.

We finally control the pressure error.
% $\|e_{\mathsf{p}}\|_{L^2(\Omega)}$. 
Since $e_{\mathsf{p}}\in L_0^2(\Omega)$, we invoke the inf-sup condition between $\mathbf{H}_0^1(\Omega)$ and $L_0^2(\Omega)$ \cite[Corollary 2.4, Chapter I]{MR851383}, \cite[Corollary B.71]{Guermond-Ern} to conclude the existence of $\mathbf{v} \in \mathbf{H}_0^1(\Omega)\subset \mathbf{H}_{0}(\text{div},\Omega)$ such that 
\begin{equation}\label{eq:e_p_tilde_v}
\|e_{\mathsf{p}}\|_{L^2(\Omega)}^2 = \int_\Omega e_{\mathsf{p}}\text{ div } \mathbf{v} \mathsf{d}\mathbf{x}, \qquad \|\nabla \mathbf{v} \|_{\mathbf{L}^2(\Omega)} \lesssim  \|e_{p}\|_{L^2(\Omega)}.
\end{equation}
% Set $\mathsf{q}=0$. 
% Utilize the pair 
Set $(\mathbf{v},0)$ as a test pair in Darcy's problem \eqref{eq:Darcy_decoupled} and utilize an elementwise integration by parts formula to obtain 
\begin{multline}\label{eq:error_equation_darcy_2}
\int_\Omega \nu(T_h)\mathbf{e}_{\mathbf{u}}\cdot \mathbf{v} \mathsf{d}\mathbf{x} 
+ 
\int_\Omega (\nu(T)-\nu(T_h))\mathbf{u}\cdot \mathbf{v} \mathsf{d}\mathbf{x} 
- 
\int_\Omega e_{\mathsf{p}}\text{ div } \mathbf{v} \mathsf{d}\mathbf{x} \\
= 
\sum_{K\in\mathscr{T}}\int_K \mathscr{R}_{K}\cdot(\mathbf{v}-\mathbf{P}_{h}\mathbf{v})\mathsf{d}\mathbf{x} 
- 
\sum_{\gamma\in\mathscr{S}}\int_\gamma \mathscr{J}_{\gamma}\cdot(\mathbf{v}-\mathbf{P}_{h} \mathbf{v})\mathsf{d}s,
\end{multline}
where $\mathbf{P}_h: \mathbf{H}^1(\Omega)\cap \mathbf{H}_{0}(\text{div},\Omega)\to \mathbf{X}_{h}$ denotes the interpolation operator introduced in Section \ref{sec:fem}. On the basis of the identities \eqref{eq:e_p_tilde_v} and \eqref{eq:error_equation_darcy_2}, H\"older's inequality, the Lipschitz property of $\nu$, the embedding $W_0^{1,p}(\Omega)\hookrightarrow L^{2p/(2-p)}(\Omega)$, the interpolation error estimates \eqref{eq:interp_RT_i} and \eqref{eq:interp_RT_iii}, and assumption \eqref{eq:global_rel_assump}, we obtain the estimate
\begin{equation*}
\|e_{\mathsf{p}}\|_{L^2(\Omega)}
\lesssim 
\|\mathbf{e}_{\mathbf{u}}\|_{\mathbf{L}^2(\Omega)}+ \|\nabla e_{T}\|_{\mathbf{L}^p(\Omega)}  + \mathscr{E}_{\T}.
\end{equation*}

This estimate and the bound $\|\mathbf{e}_{\mathbf{u}}\|_{\mathbf{L}^2(\Omega)}+ \|\nabla e_{T}\|_{\mathbf{L}^p(\Omega)}\lesssim \mathpzc{E}_{p,\T}+\mathfrak{E}_{\T}$, which follows from \eqref{eq:temp_estimate} and \eqref{eq:estimate_vel}, allow us to conclude the desired a posteriori error estimate \eqref{eq:upper_bound}.
\end{proof}

%%%%%%%%%%%%%%%%%%%%%%%%%%%%%%%%%%%%%%%%%%%%%%%%%%%%%%%%%%
%%%%%%%%%%%%%%%%%%%%%%%%%%%%%%%%%%%%%%%%%%%%%%%%%%%%%%%%%%
%%%%%%%%%%%%%%%%%%%%%%%%%%%%%%%%%%%%%%%%%%%%%%%%%%%%%%%%%%
%%%%%%%%%%%%%%%%%%%%%%%%%%%%%%%%%%%%%%%%%%%%%%%%%%%%%%%%%%

\subsection{Efficiency estimates}\label{sec:efficiency_estimates}
In this section, we study efficiency properties of the a posteriori error estimator $\mathsf{E}_{\T}$, defined in \eqref{def:total_error_estimator}, by examining each of its contributions separately. To accomplish this task, we will invoke standard residual estimation techniques which are based on the consideration of suitable bubble functions \cite{MR3059294,MR1885308}. Before proceeding with such an analysis, we introduce the following notation: for an edge or triangle $G$, let $\mathcal{V}(G)$ be the set of vertices of $G$. With this notation at hand, we define, for  $K\in\mathscr{T}_h$ and $\gamma\in\mathscr{S}$, the standard \textit{element} and \textit{edge} bubble functions
\begin{equation}\label{def:standard_bubbles}
\varphi^{}_{K}=
27\prod_{\textsc{v} \in \mathcal{V}(K)} \lambda^{}_{\textsc{v}},
\qquad
\varphi^{}_{\gamma}=
4 \prod_{\textsc{v} \in \mathcal{V}(\gamma)}\lambda^{}_{\textsc{v}}|^{}_{K'},
\quad 
K' \subset \mathcal{N}_{\gamma},
\end{equation}
respectively, where $\lambda_{\textsc{v}}$ are the barycentric coordinates of $K$. We recall that $\mathcal{N}_{\gamma}$ corresponds to the patch composed of the two elements of $\mathscr{T}_h$ sharing $\gamma$.

Inspired by references \cite{MR3237857,MR2262756}, we also introduce some suitable bubble functions which are particularly useful for analyzing the indicators associated to the discretization of the heat equation with convection \eqref{eq:heat_decoupled}. Given $K\in\mathscr{T}_h$, we define the element bubble function $\phi_K$ as
\begin{equation}\label{def:bubble_element}
\phi_K(\mathbf{x}):=
h_K^{-2}|\mathbf{x}-z|^{2}\varphi_K(\mathbf{x})  \text{  if  } z\in K,
\qquad
\phi_K(\mathbf{x}):=
\varphi_K(\mathbf{x}) \text{  if  } z\not\in K.
\end{equation}
Given $\gamma\in \mathscr{S}$, we define the edge bubble function $\phi_\gamma$ as
\begin{equation}\label{def:bubble_side}
\phi_\gamma(\mathbf{x}):=
h_\gamma^{-2}|\mathbf{x}-z|^{2} \varphi_\gamma(\mathbf{x})  \text{ if } z\in \mathring{\mathcal{N}}_\gamma,
\qquad
\phi_\gamma(\mathbf{x}):=
\varphi_\gamma(\mathbf{x}) \text{ if } z\not\in \mathring{\mathcal{N}}_\gamma,
\end{equation}
where $\mathring{\mathcal{N}}_\gamma$ denotes the interior of $\mathcal{N}_\gamma$. We recall that the Dirac measure $\delta_{z}$ is supported at $z \in \Omega$: it can thus be supported on the interior, an edge, or a vertex of an element $K$ of the triangulation $\mathscr{T}_h$.

Given $\gamma\in\mathscr{S}$, we introduce the continuation operator $\Pi_{\gamma}:L^\infty(\gamma)\rightarrow L^\infty(\mathcal{N}_\gamma)$ defined in \cite[Section 3]{MR1650051}. This operator maps polynomials onto piecewise polynomials of the same degree and it will be useful for controlling the involved \emph{jump} terms.

We now provide the following result \cite[Lemmas 3.1 and 3.2]{MR3237857}. 

\begin{lemma}[bubble function properties]\label{lemma:bubble_properties}
Let $K\in\mathscr{T}_h$, $\gamma\in\mathscr{S}$, and $r\in(1,\infty)$. If $S_{h}|^{}_K\in\mathbb{P}_1(K)$ and $R_{h}|^{}_\gamma\in \mathbb{P}_1(\gamma)$, then
\begin{align*}
\|S_{h}\|_{L^r(K)} &\lesssim \|S_{h} \phi_K^{\frac{1}{r}}\|_{L^r(K)} \lesssim \|S_{h}\|_{L^r(K)}, 
\\
\|R_{h}\|_{L^r(\gamma)} &\lesssim \|R_{h} \phi_\gamma^{\frac{1}{r}}\|_{L^r(\gamma)} \lesssim \|R_{h}\|_{L^r(\gamma)},
\\
\|\phi_\gamma \Pi_{\gamma} (R_{h})\|_{L^r(\mathcal{N}_{\gamma})} & \lesssim h_{\gamma}^{\frac{1}{r}}\|R_{h}\|_{L^r(\gamma)}.
\end{align*}
\end{lemma}

\subsubsection{Local estimates for $\mathpzc{E}_{p,K}$}

We now investigate local estimates for the local error indicators $\mathpzc{E}_{p,K}$ defined in \eqref{eq:local_indicator_temp_I} and \eqref{eq:local_indicator_temp_II}.

\begin{theorem}[local estimate for $\mathpzc{E}_{p,K}$]\label{thm:efficiency}
 Let $(\mathbf{u},\mathsf{p},T) \!\in\! \mathbf{H}_0(\textnormal{div},\Omega)\times L_0^2(\Omega)\times W_0^{1,p}(\Omega)$ be a solution to \eqref{eq:modelweak} with a forcing term $\mathbf{f}$, which is such that $\|\mathbf{f}\|_{\mathbf{L}^2(\Omega)}\leq \nu_{-}\min\{\mathfrak{C},\tilde{\mathfrak{C}}\}$. Let $(\mathbf{u}_h,\mathsf{p}_h,T_h)\in\mathbf{X}_{h}\times Q_{h}\times V_{h}$ be a solution to the discrete system \eqref{eq:model_discrete} for $0< h \leq h_{\star}$. Then, for $K\in \mathscr{T}_h$, the local indicator $\mathpzc{E}_{p,K}$ satisfies the bound
\begin{equation}\label{eq:local_efficiency_temp}
\mathpzc{E}_{p,K}
\lesssim
\|\nabla e_{T}\|_{\mathbf{L}^p(\mathcal{N}_K^{*})} + \|e_{T}\|_{L^{\frac{2p}{2-p}}(\mathcal{N}_K^{*})} + \|\mathbf{e}_{\mathbf{u}}\|_{\mathbf{L}^2(\mathcal{N}_K^{*})},
\end{equation}
where $\mathcal{N}_{K}^{*}$ is defined in \eqref{eq:patch}. The hidden constant is independent of continuous and discrete solutions $(\mathbf{u},\mathsf{p},T)$ and $(\mathbf{u}_h,\mathsf{p}_h,T_h)$, respectively, the size of the elements in the mesh $\mathscr{T}_h$, and $\#\mathscr{T}_h$.
\end{theorem}
\begin{proof}
We begin by noticing that similar arguments to the ones used to derive \eqref{eq:error_equation_temp} yield the identity
\begin{equation}\label{eq:error_equation_temp_2}
\int_\Omega (\kappa\nabla e_{T}\cdot\nabla S - T\mathbf{e}_{\mathbf{u}} \cdot \nabla S - e_{T}\mathbf{u}_h \cdot \nabla S)\mathsf{d}\mathbf{x} 
=\langle \delta_{z},S\rangle  +\sum_{K\in\mathscr{T}}\int_{K}\mathpzc{R}_K S\mathsf{d}{\mathbf{x}}  
+
\sum_{\gamma\in\mathscr{S}}\int_\gamma \mathpzc{J}_{\gamma} S \mathsf{d}s,
\end{equation}
for every $S\in W_0^{1,p\prime}(\Omega)$.

On the basis of identity \eqref{eq:error_equation_temp_2}, we proceed in three steps.

\underline{Step 1.} Let $K\in\mathscr{T}_h$. In what follows, we bound the term $h_K\|\mathpzc{R}_K\|_{L^p(K)}$ in \eqref{eq:local_indicator_temp_I}--\eqref{eq:local_indicator_temp_II}. To accomplish this task, we set $S=\phi_K \mathpzc{R}_K$ in \eqref{eq:error_equation_temp_2}, where $\phi_K$ is the bubble function defined in \eqref{def:bubble_element}. Since $\phi_K$ is such that $\langle \delta_{z}, \phi_K\mathpzc{R}_K \rangle = 0$, we thus obtain
\begin{multline*}
\|\mathpzc{R}_K\|_{L^2(K)}^2
\lesssim 
(\|\nabla e_{T}\|_{\mathbf{L}^p(K)} + \|T\|_{L^{\frac{2p}{2-p}}(K)} \|\mathbf{e}_{\mathbf{u}}\|_{\mathbf{L}^2(K)}  \\
+ \|e_{T}\|_{L^{\frac{2p}{2-p}}(K)}\|\mathbf{u}_h\|_{\mathbf{L}^2(K)})\|\nabla(\phi_K\mathpzc{R}_K)\|_{\mathbf{L}^{p\prime}(K)},
\end{multline*}
where we have used the first estimate in Lemma \ref{lemma:bubble_properties}. To control $\|\nabla(\phi_K\mathpzc{R}_K)\|_{\mathbf{L}^{p\prime}(K)}$, we utilize standard inverse estimates \cite[Lemma 4.5.3]{MR2373954} and properties of $\phi_K$ to obtain 
\[
 \|\nabla(\phi_K\mathpzc{R}_K)\|_{\mathbf{L}^{p\prime}(K)}\lesssim h_K^{-1}\|\phi_K \mathpzc{R}_K\|_{L^{p\prime}(K)}\lesssim h_K^{\frac{2}{p\prime}-2}\|\mathpzc{R}_K\|_{L^{2}(K)}.
\]
In view of this bound, the estimate $h_{K}^{1-\frac{2}{p}}\|\mathpzc{R}_K\|_{L^p(K)}\lesssim \|\mathpzc{R}_K\|_{L^2(K)}$ yields
\begin{equation}\label{eq:residual_estimate_temp}
h_K\|\mathpzc{R}_K\|_{L^p(K)}
\lesssim 
\|\nabla e_{T}\|_{\mathbf{L}^p(K)} + \|e_{T}\|_{L^{\frac{2p}{2-p}}(K)} +\|\mathbf{e}_{\mathbf{u}}\|_{\mathbf{L}^2(K)},
\end{equation}
upon utilizing the stability estimate \eqref{eq:estimatesheat} and the fact that $\mathbf{u}_h$ is uniformly bounded in $\mathbf{L}^2(\Omega)$. 

\underline{Step 2.} Let $K\in\mathscr{T}_h$ and $\gamma\in\mathscr{S}_K$. We now bound $h_K^{\frac{1}{p}}\|\mathpzc{J}_{\gamma}\|_{L^p(\gamma)}$ in \eqref{eq:local_indicator_temp_I}--\eqref{eq:local_indicator_temp_II}. To accomplish this task, we utilize the bubble function $\phi_\gamma$ defined in \eqref{def:bubble_side}. In fact, let us set $S=\phi_\gamma \Pi_{\gamma}(\mathpzc{J}_{\gamma})$ in \eqref{eq:error_equation_temp_2}. We recall that $\Pi_{\gamma}$ denotes the continuation operator of \cite[Section 3]{MR1650051}. Standard arguments thus yield  
\begin{equation*}
\|\mathpzc{J}_{\gamma}\|_{L^2(\gamma)}^2
\lesssim 
\sum_{K'\in\mathcal{N}_{\gamma}}h_{K'}^{-1}  (\|\nabla e_{T}\|_{\mathbf{L}^p(K')} +  \|e_{T}\|_{L^{\mathfrak{m}}(K')} +\|\mathbf{e}_{\mathbf{u}}\|_{\mathbf{L}^2(K')} )\|\phi_\gamma\Pi_{\gamma}(\mathpzc{J}_{\gamma})\|_{L^{p\prime}(K')},
\end{equation*}
where, to simplify notation, we have defined $\mathfrak{m}=2p/(2-p)$. We now notice that
\[
\|\phi_\gamma\Pi_{\gamma}(\mathpzc{J}_{\gamma})\|_{L^{p\prime}(K')}\lesssim h_{K}^{\frac{1}{p\prime}}\|\mathpzc{J}_{\gamma}\|_{L^{p\prime}(\gamma)},
\qquad
\|\mathpzc{J}_{\gamma}\|_{L^{p\prime}(\gamma)} \lesssim h_{K'}^{\frac{1}{p\prime}-\frac{1}{2}} \|\mathpzc{J}_{\gamma}\|_{L^{2}(\gamma)}.
\]
The first estimate follows immediately from Lemma \ref{lemma:bubble_properties} and the second one is a consequence of a scaled--trace inequality and an inverse estimate. With these estimates at hand, we can thus obtain
\begin{equation*}
\|\mathpzc{J}_{\gamma}\|_{L^2(\gamma)}^2
\lesssim 
\sum_{K'\in\mathcal{N}_\gamma} (\|\nabla e_{T}\|_{\mathbf{L}^p(K')} +  \|e_{T}\|_{L^{\mathfrak{m}}(K')}+\|\mathbf{e}_{\mathbf{u}}\|_{\mathbf{L}^2(K')}\big)h_{K'}^{-\frac{1}{p} + \frac{1}{p\prime}-\frac{1}{2}} \|\mathpzc{J}_{\gamma}\|_{L^{2}(\gamma)}.
\end{equation*}
This bound combined with the estimate $h_{K'}^{\frac{1}{2}-\frac{1}{p}}\|\mathpzc{J}_{\gamma}\|_{L^{p}(\gamma)} \lesssim  \|\mathpzc{J}_{\gamma}\|_{L^{2}(\gamma)}$ allow us to conclude the desired estimate
\begin{equation}\label{eq:jump_estimate_temp}
h_K^{\frac{1}{p}}\|\mathpzc{J}_{\gamma}\|_{L^p(\gamma)}
\lesssim 
\sum_{K'\in\mathcal{N}_\gamma}\left (\|\nabla e_{T}\|_{\mathbf{L}^p(K')}+ \|e_{T}\|_{L^{\mathfrak{m}}(K')} + \|\mathbf{e}_{\mathbf{u}}\|_{\mathbf{L}^2(K')}\right).
\end{equation}

\underline{Step 3.} Let $K\in\mathscr{T}_h$. We now bound the remaining term $h_K^{\frac{2-p}{p}}$ in \eqref{eq:local_indicator_temp_I}. We first notice that, if $K\cap \{z\}=\emptyset$, then the desired estimate \eqref{eq:local_efficiency_temp} follows directly from the previous two steps. If, on the other hand, $K\cap \{z\}\neq\emptyset$ and $z\in K$ is not a vertex of $K$, then we must obtain a bound for the aforementioned term. To accomplish this task, we invoke the smooth function $\mu$ introduced in \cite[Section 3]{MR2262756} which is such that
\begin{equation}\label{eq:properties_bubble_mu}
\mathfrak{M}_\mu:=\text{supp}(\mu)\subset \mathcal{N}_K^{*}, \quad
\mu(z)=1,\quad \|\mu\|_{L^\infty(\mathfrak{M}_\mu)}=1,\quad \|\nabla \mu \|_{\mathbf{L}^\infty(\mathfrak{M}_\mu)}\lesssim h_K^{-1}. \hspace{-0.3cm}
\end{equation}
In addition to \eqref{eq:properties_bubble_mu}, the function $\mu$ satisfies the estimates
\begin{equation}
\|\mu\|_{L^{p\prime}(\mathcal{N}_K^{*})}\lesssim h_K^{\frac{2}{p\prime}}, 
\qquad
\| \nabla \mu\|_{L^{p\prime}(\mathcal{N}_K^{*})}\lesssim h_K^{\frac{2}{p\prime}-1}, 
\qquad
\|\mu\|_{L^{p\prime}(\gamma)}\lesssim h_K^{\frac{1}{p\prime}}.
\label{eq:properties_bubble_mu_2}
\end{equation}
To bound the term $h_K^{\frac{2-p}{p}}$, we also need to introduce the set
\begin{equation*}
\mathscr{S}(\mathcal{N}_K^{*}):=\{\gamma\in\mathscr{S}: \gamma\in \partial K', \, K'\in \mathcal{N}_K^{*} \, ; \, \gamma\not\in \partial\mathcal{N}_K^{*} \}.
\end{equation*}

Set $S=\mu$ in \eqref{eq:error_equation_temp_2} and then invoke H\"older's inequality and the estimates for $\mu$ stated in \eqref{eq:properties_bubble_mu} and \eqref{eq:properties_bubble_mu_2} to obtain
\begin{multline*}
1 
\lesssim
\sum_{K'\in\mathcal{N}_K^{*}}
\left[\left(\|\nabla e_T\|_{\mathbf{L}^p(K')} + \|e_{T}\|_{L^{\mathfrak{m}}(K')} + \|\mathbf{e}_\mathbf{u}\|_{\mathbf{L}^2(K')}\right)\|\nabla \mu\|_{\mathbf{L}^{p\prime}(K')} 
\right.
\\
\left.
+ \| \mathpzc{R}_{K'}\|_{L^p(K')}\|\mu\|_{L^{p\prime}(K')}\right] 
+
\sum_{\gamma\in\mathscr{S}(\mathcal{N}_K^{*})}\|\mathpzc{J}_{\gamma}\|_{L^p(\gamma)}\|\mu\|_{L^{p\prime}(\gamma)}  \lesssim \sum_{\gamma\in\mathscr{S}(\mathcal{N}_K^{*})}h_{K'}^{\frac{1}{p\prime}}\|\mathpzc{J}_{\gamma}\|_{L^p(\gamma)}
\\
+ \sum_{K'\in\mathcal{N}_K^{*}}\left[h_{K'}^{\frac{2}{p\prime}-1}\big(\|\nabla e_T\|_{\mathbf{L}^p(K')} + \|e_{T}\|_{L^{\mathfrak{m}}(K')} + \|\mathbf{e}_\mathbf{u}\|_{\mathbf{L}^2(K')}\big) + h_{K'}^{\frac{2}{p\prime}}\|\mathpzc{R}_{K'}\|_{L^p(K')}\right].
\end{multline*}
Here, $\mathfrak{m}=2p/(2-p)$. In view of estimates \eqref{eq:residual_estimate_temp} and \eqref{eq:jump_estimate_temp}, we can thus arrive at
\begin{equation}\label{eq:delta_estimate_temp}
h_K^{\frac{2-p}{p}} \lesssim \sum_{K'\in\mathcal{N}_K^{*}}\left(\|\nabla e_T\|_{\mathbf{L}^p(K')} + \|e_{T}\|_{L^{\mathfrak{m}}(K')} + \|\mathbf{e}_\mathbf{u}\|_{\mathbf{L}^2(K')}\right).
\end{equation}

We conclude the proof by gathering the estimates \eqref{eq:residual_estimate_temp}, \eqref{eq:jump_estimate_temp}, and \eqref{eq:delta_estimate_temp}.
\end{proof}

\subsubsection{Local estimates for $\mathscr{E}_{K}$}

We now analyze local estimates for the indicator $\mathscr{E}_{K}$ defined in \eqref{eq:local_indicator_darcy_II}. As an instrumental ingredient, we introduce a suitable approximation of the term $\nu(T_h)$ involved in the definition of the element residual $\mathscr{R}_{K}$ given in \eqref{eq:residuals_Darcy_2}. For $K \in \mathscr{T}_h$, we define the linear approximation
$\nu_h:  W^{1,p}(K)\ni S\to\nu_{h}(S)\in\mathbb{P}_{1}(K)$ by
\begin{equation}\label{eq:nu_h_f}
\nu_h(S)(\mathbf{x})|_{K}:=\frac{1}{|K|}\int_{K}\nu(S(\mathbf{y}))\mathsf{d}\mathbf{y} +\left[\frac{1}{|K|}\int_{K}(\nabla\nu(S))(\mathbf{y})\mathsf{d}\mathbf{y}\right]\cdot (\mathbf{x}-\mathbf{c}).
\end{equation}
Here, $|K|$ and $\mathbf{c}$ denote the Lebesgue measure and the center of $K$, respectively. We notice that $|\mathbf{x}-\mathbf{c}|\leq h_K$ and observe that $\nu_h$ is invariant under affine transformations and that, if $\nu(S)|_K \in \mathbb{P}_1(K)$, then $\nu_h(S) = \nu(S)$ in $K$ \cite[Section 4.2]{MR4041519}.

In what follows, and in addition to the assumptions stated in Section \ref{sec:main_assump}, we will assume that $\nu \in W^{2,\infty}(\mathbb{R})$. This immediately implies the existence of a real number $\nu_{+}'$ such that $\nu'(s) \leq \nu_{+}'$ for all $s\in \mathbb{R}$. In addition, we have that $\nu'\in C^{0,1}(\mathbb{R})$ with a Lipschitz constant $C_{\mathcal{L}}^{\prime}$. With the assumption that $\nu \in W^{2,\infty}(\mathbb{R})$ at hand, basic computations, on the basis of definition \eqref{eq:nu_h_f} and  \eqref{eq:nu}, reveal the following bound:
\begin{equation}\label{eq:nu_h_f_bound}
\|\nu_{h}(S)\|_{L^{\infty}(K)} \leq \nu_{+}+\nu_{+}^{\prime}h_K|K|^{-\frac{1}{p}}\|\nabla S\|_{\mathbf{L}^{p}(K)} \quad  \forall S\in W^{1,p}(K). 
\end{equation}
Similar arguments also yield, for $S,R\in W^{1,p}(K)$ and $p > 4/3$, the following estimate:
\begin{multline}\label{eq:nu_S_nu_R}
\|\nu_{h}(S)-\nu_{h}(R)\|_{L^{\frac{2p}{2-p}}(K)} \leq h_K\left(\nu_{+}^{\prime}|K|^{-\frac{1}{2}}\|\nabla (S - R)\|_{\mathbf{L}^p(K)}
\right.
\\
\left.
+C_{\mathcal{L}}^{\prime}|K|^{-\frac{1}{p}}\|S-R\|_{L^{\frac{2p}{2-p}}(K)}\|\nabla R\|_{\mathbf{L}^p(K)}\right)
+
C_{\mathcal{L}}
\|S-R\|_{L^{\frac{2p}{2-p}}(K)}.
\end{multline}

The following projection estimate is instrumental.

\begin{lemma}[projection estimate]\label{thm:lemma1_efficiencydarcy}
Let $(\mathbf{u},\mathsf{p},T) \in \mathbf{H}_0(\textnormal{div},\Omega)\times L_0^2(\Omega)\times W_0^{1,p}(\Omega)$ be a solution to \eqref{eq:modelweak} with a forcing term $\mathbf{f}$, which is such that $\|\mathbf{f}\|_{\mathbf{L}^2(\Omega)}\leq \nu_{-}\min\{\mathfrak{C},\tilde{\mathfrak{C}}\}$.  Let $(\mathbf{u}_h,\mathsf{p}_h,T_h) \in \mathbf{X}_{h}\times Q_{h}\times V_{h}$ be a solution to the discrete system \eqref{eq:model_discrete} for $0<h\leq h_{\star}$. If $\nu\in W^{2,\infty}(\mathbb{R})$, and \eqref{eq:discrete_assump} and \eqref{eq:global_rel_assump} holds, then, for $K\in \mathscr{T}_h$, we have
\begin{multline}\label{eq:lemma_nu}
\|(\nu(T_h)-\nu_h(T_h))\mathbf{u}_h\|_{\mathbf{L}^2(K)}
\lesssim
 \|\nu_h(T)-\nu(T)\|_{L^{\frac{2p}{2-p}}(K)}
\\
+
h_K^{1-\frac{2}{p}}\left( 
\|\mathbf{e}_{\mathbf{u}}\|_{\mathbf{L}^2(K)} + \|\nabla e_{T}\|_{\mathbf{L}^p(K)}
+
 \|e_{T}\|_{L^{\frac{2p}{2-p}}(K)} \right).
\end{multline}
\end{lemma}
\begin{proof}
We begin with a simple application of the triangle inequality to obtain
\begin{multline}\label{eq:projection_triangle}
\|(\nu(T_h)-\nu_h(T_h))\mathbf{u}_h\|_{\mathbf{L}^2(K)}\\
\leq \|(\nu(T_h)-\nu_h(T_h))\mathbf{e}_{\mathbf{u}}\|_{\mathbf{L}^2(K)} + \|(\nu(T_h)-\nu_h(T_h))\mathbf{u}\|_{\mathbf{L}^2(K)}=: \mathrm{I}+\mathrm{II}.
\end{multline}

We first control the term $\mathrm{I}$. To accomplish this task, we invoke the fact that $\nu$ is uniformly bounded, i.e., $\nu$ satisfies \eqref{eq:nu}, in combination with estimate \eqref{eq:nu_h_f_bound} to arrive at
\begin{align*}
\mathrm{I} &\lesssim \|\mathbf{e}_{\mathbf{u}}\|_{\mathbf{L}^2(K)} + 
\left[ 1+ h_K|K|^{-\frac{1}{p}}\left(\|\nabla e_{T}\|_{\mathbf{L}^p(K)}+\|\nabla T\|_{\mathbf{L}^p(K)}\right) \right]
\|\mathbf{e}_{\mathbf{u}}\|_{\mathbf{L}^2(K)} \\
&\lesssim  \|\mathbf{e}_{\mathbf{u}}\|_{\mathbf{L}^2(K)} + h_{K}^{1-\frac{2}{p}}\left(\|\nabla e_{T}\|_{\mathbf{L}^p(K)} + \|\mathbf{e}_{\mathbf{u}}\|_{\mathbf{L}^2(K)}\right),
\end{align*}
where we have also utilized assumption \eqref{eq:discrete_assump}, the stability estimate \eqref{eq:estimatesheat}, and the fact that $|K|\approx h_K^2$. To estimate $\mathrm{II}$, we invoke, again, a triangle inequality to obtain
\begin{align*}
\mathrm{II} &\leq \|(\nu_h(T_h)-\nu_h(T))\mathbf{u}\|_{\mathbf{L}^2(K)} 
+ \|(\nu_h(T)-\nu(T))\mathbf{u}\|_{\mathbf{L}^2(K)}
+
\|(\nu(T)-\nu(T_h))\mathbf{u}\|_{\mathbf{L}^2(K)}\\
&=:\mathrm{II}_{1} + \mathrm{II}_{2} + \mathrm{II}_{3}.
\end{align*}
H\"older's inequality combined with the assumption \eqref{eq:global_rel_assump} on $\mathbf{u}\in \mathbf{L}^{p\prime}(\Omega)$ yields the estimate $\mathrm{II}_{1} \leq  \|\nu_h(T_h)-\nu_h(T)\|_{L^{\mathfrak{m}}(K)}\|\mathbf{u}\|_{\mathbf{L}^{p\prime}(K)} \lesssim \|\nu_h(T_h)-\nu_h(T)\|_{L^{\mathfrak{m}}(K)}$, where $\mathfrak{m} = 2p/(2-p)$. This bound, estimate 
\eqref{eq:nu_S_nu_R}, the fact that $|K|\approx h_K^2$, and the stability estimate \eqref{eq:estimatesheat} allow us to conclude that
\begin{equation*}
\mathrm{II}_{1}
\lesssim
\|\nabla e_{T}\|_{\mathbf{L}^p(K)}+\left( 1+h_K^{1-2/p} \right) \|e_{T}\|_{L^{\mathfrak{m}}(K)}.
\end{equation*}
To control the terms $\mathrm{II}_{2}$ and $\mathrm{II}_{3}$, we use H\"older's inequality and the Lipschitz property of $\nu$. These arguments yield 
\begin{equation*}
\mathrm{II}_{2} \leq \|\nu_h(T)-\nu(T)\|_{L^{\mathfrak{m}}(K)}\|\mathbf{u}\|_{\mathbf{L}^{p'}(K)}, \quad  \mathrm{II}_{3} \leq C_{\mathcal{L}} \|e_{T}\|_{L^{\mathfrak{m}}(K)}\|\mathbf{u}\|_{\mathbf{L}^{p'}(K)}.
\end{equation*}
Consequently, in view of the fact that $\mathbf{u}\in \mathbf{L}^{p\prime}(\Omega)$, we obtain
\begin{equation*}
\mathrm{II} 
\lesssim
\|\nabla e_{T}\|_{\mathbf{L}^p(K)}+\left(1+h_K^{1-2/p } \right) \|e_{T}\|_{L^{\mathfrak{m}}(K)} + \|\nu_h(T)-\nu(T)\|_{L^{\mathfrak{m}}(K)}.
\end{equation*}

The desired estimate \eqref{eq:lemma_nu} thus follows from replacing the obtained ones for $\mathrm{I}$ and $\mathrm{II}$ into \eqref{eq:projection_triangle}. This concludes the proof.
\end{proof}

We now proceed to investigate local estimates for the error indicator $\mathscr{E}_{K}$ defined in \eqref{eq:local_indicator_darcy_II}.

\begin{theorem}[local estimate for $\mathscr{E}_{K}$]\label{thm:efficiencydarcy}
Under the framework of Lemma \ref{thm:lemma1_efficiencydarcy}, we have, for $K\in\mathscr{T}_h$, the following local estimate for the error indicator $\mathscr{E}_{K}$:
\begin{multline}\label{eq:local_efficiency_darcya}
\mathscr{E}_{K}
\lesssim 
h_K^{\frac{2}{p\prime}}
\left(\|\mathbf{e}_{\mathbf{u}}\|_{\mathbf{L}^2(\mathcal{N}_K)}+\|e_{T}\|_{L^{\frac{2p}{2-p}}(\mathcal{N}_K)}+ \|\nabla e_{T}\|_{\mathbf{L}^p(\mathcal{N}_K)}\right) + \|e_\mathsf{p}\|_{L^2(\mathcal{N}_K)}\\
+\sum_{K'\in\mathcal{N}_K}h_{K'}\left(\|\nu_h(T)-\nu(T)\|_{L^{\frac{2p}{2-p}}(K')} + \|\mathbf{f}-\mathscr{P}_{K'}\mathbf{f}\|_{\mathbf{L}^2(K')}\right),
\end{multline}
where $\mathcal{N}_K$ is defined in \eqref{eq:patch} and $\mathscr{P}_{K}$ denotes the $\mathbf{L}^2(K)$--orthogonal projection operator onto $[\mathbb{P}_{0}(K)]^2$. The hidden constant is independent of continuous and discrete solutions $(\mathbf{u},\mathsf{p},T)$ and $(\mathbf{u}_h,\mathsf{p}_h,T_h)$, respectively, the size of the elements in the mesh $\mathscr{T}_h$, and $\#\mathscr{T}_h$.
\end{theorem}
\begin{proof}
We begin the proof by noticing that similar arguments to the ones used to derive  \eqref{eq:error_equation_darcy_2} yield, for an arbitrary function $\mathbf{v} \in \mathbf{H}_0^{1}(\Omega)$, the identity
\begin{multline}\label{eq:error_equation_darcy_3}
\int_\Omega \nu(T_h)\mathbf{e}_{\mathbf{u}}\cdot\mathbf{v}\mathsf{d}\mathbf{x} + \int_\Omega (\nu(T)-\nu(T_h))\mathbf{u}\cdot\mathbf{v}\mathsf{d}\mathbf{x} - \int_\Omega e_{\mathsf{p}}\text{ div }\mathbf{v}\mathsf{d}\mathbf{x} 
\\
= \sum_{K\in\mathscr{T}}\left(\int_K(\mathscr{P}_{K}\mathbf{f}-\nu_h(T_h)\mathbf{u}_h-\nabla \mathsf{p}_h)\cdot\mathbf{v}\mathsf{d}\mathbf{x} + \int_K(\mathbf{f}-\mathscr{P}_{K}\mathbf{f})\cdot\mathbf{v}\mathsf{d}\mathbf{x}\right.\\
\left.
- \int_K((\nu(T_h)-\nu_h(T_h))\mathbf{u}_h\cdot\mathbf{v}\mathsf{d}\mathbf{x}\right)
-
\sum_{\gamma\in\mathscr{S}}\int_\gamma \mathscr{J}_{\gamma}\cdot \mathbf{v}\mathsf{d}s.
\end{multline}

We now proceed in two steps. 

\underline{Step 1.} Let $K\in\mathscr{T}_h$. We bound the residual term $h_K\|\mathscr{R}_{K}\|_{\mathbf{L}^2(K)}$ in \eqref{eq:local_indicator_darcy_II}. We begin with an application of a triangle inequality to obtain
\begin{multline}\label{eq:residual_triangle_darcy}
h_K\|\mathscr{R}_{K}\|_{\mathbf{L}^2(K)} 
\leq
h_K\|\mathscr{P}_{K}\mathbf{f}-\nu_h(T_h)\mathbf{u}_h-\nabla \mathsf{p}_h\|_{\mathbf{L}^2(K)}\\
+ h_K\|(\nu(T_h)-\nu_h(T_h))\mathbf{u}_h\|_{\mathbf{L}^2(K)} + h_K\|\mathbf{f}-\mathscr{P}_{K}\mathbf{f}\|_{\mathbf{L}^2(K)}.
\end{multline}
With the bound \eqref{eq:lemma_nu} at hand, it thus suffices to control the first term on the right-hand side of \eqref{eq:residual_triangle_darcy}. To accomplish this task, we set $\mathbf{v} = \varphi_K\tilde{\mathscr{R}}_{K}$ in \eqref{eq:error_equation_darcy_3}, where $\tilde{\mathscr{R}}_{K}=(\mathscr{P}_{K}\mathbf{f}-\nu_h(T_h)\mathbf{u}_h-\nabla \mathsf{p}_h)|^{}_K$. Standard properties of the bubble function $\varphi_K$ combined with basic inequalities and standard inverse estimates \cite[Lemma 4.5.3]{MR2373954} yield
\begin{multline*}
\|\tilde{\mathscr{R}}_K\|_{\mathbf{L}^2(K)}
\lesssim
\|\mathbf{e}_{\mathbf{u}}\|_{\mathbf{L}^2(K)} + \|\nu(T)-\nu(T_h)\|_{L^{\frac{2p}{2-p}}(K)}\|\mathbf{u}\|_{\mathbf{L}^{p\prime}(K)} + h_K^{-1}\| e_{\mathsf{p}}\|_{L^2(K)}
\\
+  \|\mathbf{f}-\mathscr{P}_{K}\mathbf{f}\|_{\mathbf{L}^2(K)} 
+ \|(\nu(T_h)-\nu_h(T_h))\mathbf{u}_h\|_{\mathbf{L}^2(K)}.
\end{multline*}
We now invoke the Lipschitz property that $\nu$ satisfies, estimate \eqref{eq:lemma_nu}, and the fact that $\mathbf{u}\in \mathbf{L}^{p\prime}(\Omega)$, to conclude that
\begin{multline}\label{eq:estimate_residual_darcy}
h_K\|\tilde{\mathscr{R}}_K\|_{\mathbf{L}^2(K)}
\lesssim
h_K^{\frac{2}{p'}}\left(\|\mathbf{e}_{\mathbf{u}}\|_{\mathbf{L}^2(K)}+\|e_{T}\|_{L^{\frac{2p}{2-p}}(K)} + \|\nabla e_{T}\|_{\mathbf{L}^p(K)}\right) \\
+ \|e_\mathsf{p}\|_{L^2(K)} + h_K\|\nu_h(T)-\nu(T)\|_{L^{\frac{2p}{2-p}}(K)} + h_K\|\mathbf{f}-\mathscr{P}_{K}\mathbf{f}\|_{\mathbf{L}^2(K)}.
\end{multline}

The desired estimate for $h_K\| \mathscr{R}_K\|_{\mathbf{L}^2(K)}$ follows from \eqref{eq:residual_triangle_darcy} and \eqref{eq:estimate_residual_darcy}.

\underline{Step 2.} Let $K\in\mathscr{T}_h$ and $\gamma\in\mathscr{S}_K$. We now bound the term $h_K^{\frac{1}{2}}\|\mathscr{J}_{\gamma}\|_{\mathbf{L}^2(\gamma)}$ in \eqref{eq:local_indicator_darcy_II}. To accomplish this task, we set $\mathbf{v} = \varphi_\gamma\Pi_{\gamma}(\mathscr{J}_{\gamma})$ in \eqref{eq:error_equation_darcy_3}, where $\varphi_\gamma$ denotes the bubble function defined in \eqref{def:standard_bubbles}. Standard properties of the bubble function $\varphi_\gamma$ and inverse estimates allow us to thus obtain the estimate 
\begin{multline*}
\|\mathscr{J}_{\gamma}\|_{\mathbf{L}^2(\gamma)}^2
\lesssim 
\sum_{K'\in\mathcal{N}_\gamma}\bigg[
\|\mathbf{e}_{\mathbf{u}}\|_{\mathbf{L}^2(K')} + h_{K'}^{-1}\|e_{\mathsf{p}}\|_{L^2(K')} + \|\tilde{\mathscr{R}}_{K'} \|_{\mathbf{L}^2(K')} \\
 + \|\nu(T)-\nu(T_h)\|_{L^{\frac{2p}{2-p}}(K')}\|\mathbf{u}\|_{\mathbf{L}^{p\prime}(K')} 
 + \|\mathbf{f}-\mathscr{P}_{K'}\mathbf{f}\|_{\mathbf{L}^2(K')}\\
 +\|(\nu(T_h)-\nu_h(T_h))\mathbf{u}_h\|_{\mathbf{L}^2(K')}\bigg]\|\varphi_\gamma\Pi_{\gamma}(\mathscr{J}_{\gamma})\|_{\mathbf{L}^2(K')}.
\end{multline*}
We now invoke the bound $\|\varphi_\gamma\Pi_{\gamma}(\mathscr{J}_{\gamma})\|_{\mathbf{L}^2(K')}\lesssim h_K^{\frac{1}{2}}\|\mathscr{J}_{\gamma}\|_{\mathbf{L}^2(\gamma)}$, the Lipschitz property that $\nu$ satisfies, the fact that $\mathbf{u}\in \mathbf{L}^{p\prime}(\Omega)$, and  estimates \eqref{eq:lemma_nu} and \eqref{eq:estimate_residual_darcy} to arrive at
\begin{multline}\label{eq:jump_darcy}
h_{K}^{\frac{1}{2}}\|\mathscr{J}_{\gamma}\|_{\mathbf{L}^2(\gamma)}
\lesssim 
\sum_{K'\in\mathcal{N}_\gamma} \bigg[
h_{K'}^{\frac{2}{p\prime}}
\left(\|\mathbf{e}_{\mathbf{u}}\|_{\mathbf{L}^2(K')}+\|e_{T}\|_{L^{\frac{2p}{2-p}}(K')}  + \|\nabla e_{T}\|_{\mathbf{L}^p(K')}
\right) 
\\
+ \|e_\mathsf{p}\|_{L^2(K')}  + h_{K'}\|\nu_h(T)-\nu(T)\|_{L^{\frac{2p}{2-p}}(K')} +  h_{K'}\|\mathbf{f}-\mathscr{P}_{K'}\mathbf{f}\|_{\mathbf{L}^2(K')}\bigg].
\end{multline}

The desired local estimate \eqref{eq:local_efficiency_darcya} thus follows by collecting the bounds \eqref{eq:residual_triangle_darcy}, \eqref{eq:estimate_residual_darcy}, and \eqref{eq:jump_darcy}. This concludes the proof.
\end{proof}

\subsubsection{Local estimates for $\mathfrak{E}_{K}$}

We now present local estimates for the local error indicator $\mathfrak{E}_{K}$ defined in \eqref{eq:local_indicator_darcy_I}.

\begin{theorem}[local estimates for $\mathfrak{E}_{K}$]\label{thm:efficiencydarcy2}
% Let $\mathbf{f}$ be a forcing term such that it satisfies $\|\mathbf{f}\|_{\mathbf{L}^2(\Omega)}\leq \nu_{-}\min\{\mathfrak{C},\tilde{\mathfrak{C}}\}$. 
Let $(\mathbf{u},\mathsf{p},T) \in \mathbf{H}_0(\textnormal{div},\Omega)\times L_0^2(\Omega)\times W_0^{1,p}(\Omega)$ be a solution of \eqref{eq:modelweak} with a forcing term $\mathbf{f}$ which is such that $\|\mathbf{f}\|_{\mathbf{L}^2(\Omega)}\leq \nu_{-}\min\{\mathfrak{C},\tilde{\mathfrak{C}}\}$. Let $(\mathbf{u}_h,\mathsf{p}_h,T_h) \in \mathbf{X}_{h}\times Q_{h}\times V_{h}$ be a solution to the discrete system \eqref{eq:model_discrete} for $0< h \leq h_{\star}$. If assumptions \eqref{eq:discrete_assump} and \eqref{eq:global_rel_assump} hold and, for each $K\in\mathscr{T}_h$, $\nu|^{}_K$ is a polynomial, then the local error indicator $\mathfrak{E}_{K}$ satisfies the local estimate
\begin{multline}\label{eq:local_efficiency_darcya2}
\mathfrak{E}_{K}
\lesssim 
\|\mathbf{e}_{\mathbf{u}}\|_{\mathbf{L}^2(\mathcal{N}_K)} +  \sum_{{K'}\in\mathcal{N}_{K}}h_{K'}\|\mathbf{curl}(\mathbf{f}-\mathscr{P}_{K'}\mathbf{f})\|_{L^2(K')} \\
+ \|e_{T}\|_{L^{\frac{2p}{2-p}}(\mathcal{N}_K)} + \sum_{\gamma\in\mathscr{S}_{K}} h_K^{\frac{1}{2}}\|\llbracket (\mathbf{f}-\mathscr{P}_{K}\mathbf{f})\cdot\boldsymbol\tau\rrbracket\|_{L^2(\gamma)},
\end{multline}
where $\mathcal{N}_{K}$ is defined in \eqref{eq:patch} and $\mathscr{P}_{K}$ denotes the $\mathbf{L}^2(K)$--orthogonal projection operator onto $[\mathbb{P}_{0}(K)]^2$. The hidden constant is independent of continuous and discrete solutions, $(\mathbf{u},\mathsf{p},T)$ and $(\mathbf{u}_h,\mathsf{p}_h,T_h)$, respectively, the size of the elements in the mesh $\mathscr{T}_h$, and $\#\mathscr{T}_h$.
\end{theorem}
\begin{proof}
Let $\mathbf{v}\in \mathbf{V}(\Omega)$. In view of Lemma \ref{lemma:aux_curl_result}, we deduce the existence of a unique function $\vartheta\in H_0^1(\Omega)$ such that $\mathbf{v} = \mathbf{curl~}\vartheta$ together with the estimate $ \|\vartheta\|_{H^1(\Omega)}\lesssim \|\mathbf{v}\|_{\mathbf{L}^2(\Omega)}$. With this setting at hand, we invoke similar arguments to the ones utilized to obtain \eqref{eq:error_equation_darcy_1} to arrive at the identity
\begin{multline}\label{eq:error_equation_darcy_4}
\int_\Omega \nu(T_h)\mathbf{e}_{\mathbf{u}}\cdot\mathbf{v}\mathsf{d}\mathbf{x} + \int_\Omega (\nu(T)-\nu(T_h))\mathbf{u}\cdot\mathbf{v}\mathsf{d}\mathbf{x} \\
= \sum_{K\in\mathscr{T}}\left[ \int_K \mathbf{curl}(\mathscr{P}_{K}\mathbf{f}-\nu(T_h)\mathbf{u}_h)\vartheta\mathsf{d}\mathbf{x} + \int_K \mathbf{curl}(\mathbf{f}-\mathscr{P}_{K}\mathbf{f})\vartheta\mathsf{d}\mathbf{x} \right] \\
+ \sum_{\gamma\in\mathscr{S}}\left[\int_\gamma \llbracket (\mathbf{f}-\mathscr{P}_{K}\mathbf{f})\cdot\boldsymbol\tau\rrbracket\vartheta\mathsf{d}s + \int_\gamma \llbracket (\mathscr{P}_{K}\mathbf{f}\!-\!\nu(T_h)\mathbf{u}_h)\cdot\boldsymbol\tau\rrbracket\vartheta\mathsf{d}s\right].
\end{multline}

We now proceed in two steps.

\underline{Step 1.} Let $K\in\mathscr{T}_h$. The goal of this step is to control the residual term  $h_K\|\mathfrak{R}_{K}\|_{L^2(K)}$ in \eqref{eq:local_indicator_darcy_I}. To accomplish this task, we begin with a simple application of a triangle inequality to obtain
\begin{equation}\label{eq:residual_triangle_curl}
h_K\|\mathfrak{R}_{K}\|_{L^2(K)} 
\leq
h_K\|\mathbf{curl}(\mathscr{P}_{K}\mathbf{f}-\nu(T_h)\mathbf{u}_h)\|_{L^2(K)} + h_K\|\mathbf{curl}(\mathbf{f}-\mathscr{P}_{K}\mathbf{f})\|_{L^2(K)}.
\end{equation}
It thus suffices to control the first term on the right-hand side of the previous expression. To accomplish this task, we define $\tilde{\mathfrak{R}}_K:=\mathbf{curl}(\mathscr{P}_{K}\mathbf{f}-\nu(T_h)\mathbf{u}_h)|_{K}$. We thus set $\mathbf{v} = \mathbf{curl}(\tilde{\mathfrak{R}}_{K}\varphi_K)$ and $\vartheta=\tilde{\mathfrak{R}}_{K}\varphi_K$ in \eqref{eq:error_equation_darcy_4}, invoke standard inverse estimates, and basic properties of the bubble function $\varphi_K$ to obtain
\begin{multline*}
\|\tilde{\mathfrak{R}}_{K}\|_{L^2(K)}
\lesssim 
 h_K^{-1}\|\nu(T)-\nu(T_h)\|_{L^{\frac{2p}{2-p}}(K)}\|\mathbf{u}\|_{\mathbf{L}^{p\prime}(K)} \\+h_K^{-1}\|\mathbf{e}_{\mathbf{u}}\|_{\mathbf{L}^2(K)} + \|\mathbf{curl}(\mathbf{f}-\mathscr{P}_{K}\mathbf{f})\|_{L^2(K)}.
\end{multline*}
The Lipschitz property that $\nu$ satisfies combined with assumption \eqref{eq:global_rel_assump} thus yield
\begin{equation}\label{eq:estimate_residual_curl}
h_K\|\tilde{\mathfrak{R}}_K\|_{L^2(K)} 
\lesssim
\|\mathbf{e}_{\mathbf{u}}\|_{\mathbf{L}^2(K)} + \|e_{T}\|_{L^{\frac{2p}{2-p}}(K)} 
+  h_K\|\mathbf{curl}(\mathbf{f}-\mathscr{P}_{K}\mathbf{f})\|_{L^2(K)}.
\end{equation}
The desired bound for the residual term thus follows directly from \eqref{eq:residual_triangle_curl} and \eqref{eq:estimate_residual_curl}.

\underline{Step 2.} Let $K\in\mathscr{T}_h$ and $\gamma\in\mathscr{S}_K$. We define $\tilde{\mathfrak{J}}_{\gamma}=\llbracket (\mathscr{P}_{K}\mathbf{f}-\nu(T_h)\mathbf{u}_h)\cdot\boldsymbol\tau\rrbracket$ and bound $h_K^{\frac{1}{2}}\|\mathfrak{J}_{\gamma}\|_{L^2(\gamma)}$ in \eqref{eq:local_indicator_darcy_I}. 
% Define $\tilde{\mathfrak{J}}_{\gamma}=\llbracket (\mathscr{P}_{K}\mathbf{f}-\nu(T_h)\mathbf{u}_h)\cdot\boldsymbol\tau\rrbracket$ and 
Invoke a triangle inequality to arrive at
\begin{equation}\label{eq:residual_triangle_curl_jump}
h_K^{\frac{1}{2}}\|\mathfrak{J}_{\gamma}\|_{L^2(\gamma)}
\leq
h_K^{\frac{1}{2}}\|\tilde{\mathfrak{J}}_{\gamma}\|_{L^2(\gamma)} + h_K^{\frac{1}{2}}\|\llbracket (\mathbf{f}-\mathscr{P}_{K}\mathbf{f})\cdot\boldsymbol\tau\rrbracket\|_{L^2(\gamma)}.
\end{equation}
In view of \eqref{eq:residual_triangle_curl_jump} it thus suffices to bound the term $h_K^{\frac{1}{2}}\|\tilde{\mathfrak{J}}_{\gamma}\|_{L^2(\gamma)}$. To accomplish this task, we set $\vartheta=\Pi_{\gamma}(\tilde{\mathfrak{J}}_{\gamma})\varphi_\gamma$, where  $\varphi_\gamma$ corresponds to the bubble function defined in \eqref{def:standard_bubbles}, and $\mathbf{v} = \mathbf{curl}( \Pi_{\gamma}(\tilde{\mathfrak{J}}_{\gamma}) \varphi_\gamma)$ in \eqref{eq:error_equation_darcy_4}. Invoke basic properties of $\varphi_\gamma$ to arrive at
\begin{multline*}
\|\tilde{\mathfrak{J}}_{\gamma}\|_{L^2(\gamma)}
\lesssim 
\sum_{K'\in\mathcal{N}_\gamma}\bigg(
h_{K'}^{-\frac{1}{2}} \|\mathbf{e}_{\mathbf{u}}\|_{\mathbf{L}^2({K'})} + h_{K'}^{-\frac{1}{2}} \|\nu(T)-\nu(T_h)\|_{L^{\frac{2p}{2-p}}({K'})}\|\mathbf{u}\|_{\mathbf{L}^{p\prime}({K'})} \\
+ h_{K'}^{\frac{1}{2}} \|\mathbf{curl}(\mathbf{f}-\mathscr{P}_{K'}\mathbf{f})\|_{L^2(K')}  + h_{K'}^{\frac{1}{2}}\|\tilde{\mathfrak{R}}_{K'}\|_{L^{2}(K')}\bigg) + \|\llbracket (\mathbf{f}-\mathscr{P}_{K'}\mathbf{f})\cdot\boldsymbol\tau\rrbracket\|_{L^2(\gamma)},
\end{multline*}
where we have also used the estimate $\|\varphi_\gamma\Pi_{\gamma}(\tilde{\mathfrak{J}}_{\gamma})\|_{L^2(K')}\lesssim h_K^{\frac{1}{2}}\|\tilde{\mathfrak{J}}_{\gamma}\|_{L^2(\gamma)}$. Consequently, the Lipschitz property that $\nu$ satisfies, estimate \eqref{eq:estimate_residual_curl}, and the regularity assumption $\mathbf{u}\in \mathbf{L}^{p\prime}(\Omega)$ yield
\begin{multline*}
h_{K}^{\frac{1}{2}}\|\tilde{\mathfrak{J}}_{\gamma}\|_{L^2(\gamma)}
\lesssim \! 
\sum_{K'\in\mathcal{N}_\gamma} \!  \! \bigg(\|\mathbf{e}_{\mathbf{u}}\|_{\mathbf{L}^2({K'})} + \|e_{T}\|_{L^{\frac{2p}{2-p}}({K'})} \\
 +  h_{K'}\|\mathbf{curl}(\mathbf{f}-\mathscr{P}_{K'}\mathbf{f})\|_{L^2(K')} \bigg) + h_K^{\frac{1}{2}}\|\llbracket (\mathbf{f}-\mathscr{P}_{K}\mathbf{f})\cdot\boldsymbol\tau\rrbracket\|_{L^2(\gamma)}.
\end{multline*}

The combination of the estimates obtained in Steps 1 and 2 concludes the proof.
\end{proof}

%%%%%%%%%%%%%%%%%%%%%%%%%%%%%%%%%%%%%%%%%%%%%%%%%%%%%%%
%%%%%%%%%%%%%%%%%%%%%%%%%%%%%%%%%%%%%%%%%%%%%%%%%%%%%%%
%%%%%%%%%%%%%%%%%%%%%%%%%%%%%%%%%%%%%%%%%%%%%%%%%%%%%%%
%%%%%%%%%%%%%%%%%%%%%%%%%%%%%%%%%%%%%%%%%%%%%%%%%%%%%%%
%%%%%%%%%%%%%%%%%%%%%%%%%%%%%%%%%%%%%%%%%%%%%%%%%%%%%%%
%%%%%%%%%%%%%%%%%%%%%%%%%%%%%%%%%%%%%%%%%%%%%%%%%%%%%%%
%%%%%%%%%%%%%%%%%%%%%%%%%%%%%%%%%%%%%%%%%%%%%%%%%%%%%%%
%%%%%%%%%%%%%%%%%%%%%%%%%%%%%%%%%%%%%%%%%%%%%%%%%%%%%%%

\section{Numerical experiments}
\label{sec:numericalexperiments}

In this section, we present a series of numerical examples that illustrate the performance of the devised error estimator $\mathsf{E}_{\T}$ defined in \eqref{def:total_error_estimator}. The examples  have been carried out with the help of a  code that we implemented using \texttt{C++}.
All matrices have been assembled exactly and global linear systems were solved using the multifrontal massively parallel sparse direct solver (MUMPS) \cite{MUMPS1,MUMPS2}. The right-hand sides, local indicators, and the error estimator were computed by a quadrature formula which is exact for polynomials of degree 19. To visualize finite element approximations we have used the open source application ParaView \cite{Ahrens2005ParaViewAE,Ayachit2015ThePG}.

For a given partition $\mathscr{T}_h$, we solve the discrete system \eqref{eq:model_discrete}, within the discrete setting $\mathbf{X}_{h}\times Q_h\times V_h$, by using the iterative strategy described in \textbf{Algorithm 1}. Once a discrete solution is obtained, we compute, for each $K \in \mathscr{T}_h$, the local error indicator $\mathsf{E}_{K}$, defined by
\begin{equation}\label{def:total_indicator}
\mathsf{E}_K:=\mathpzc{E}_{p,K}+\mathfrak{E}_{K}+\mathscr{E}_K,
\end{equation}
to drive the adaptive procedure described in \textbf{Algorithm 2}. A sequence of adaptively refined meshes is thus generated from the initial meshes shown in Figure \ref{fig:mesh}.

\begin{figure}[!ht]
\centering
\hspace{-1.0cm}
\begin{minipage}[b]{0.35\textwidth}\centering
\includegraphics[width=1.8cm,height=1.8cm,scale=0.66]{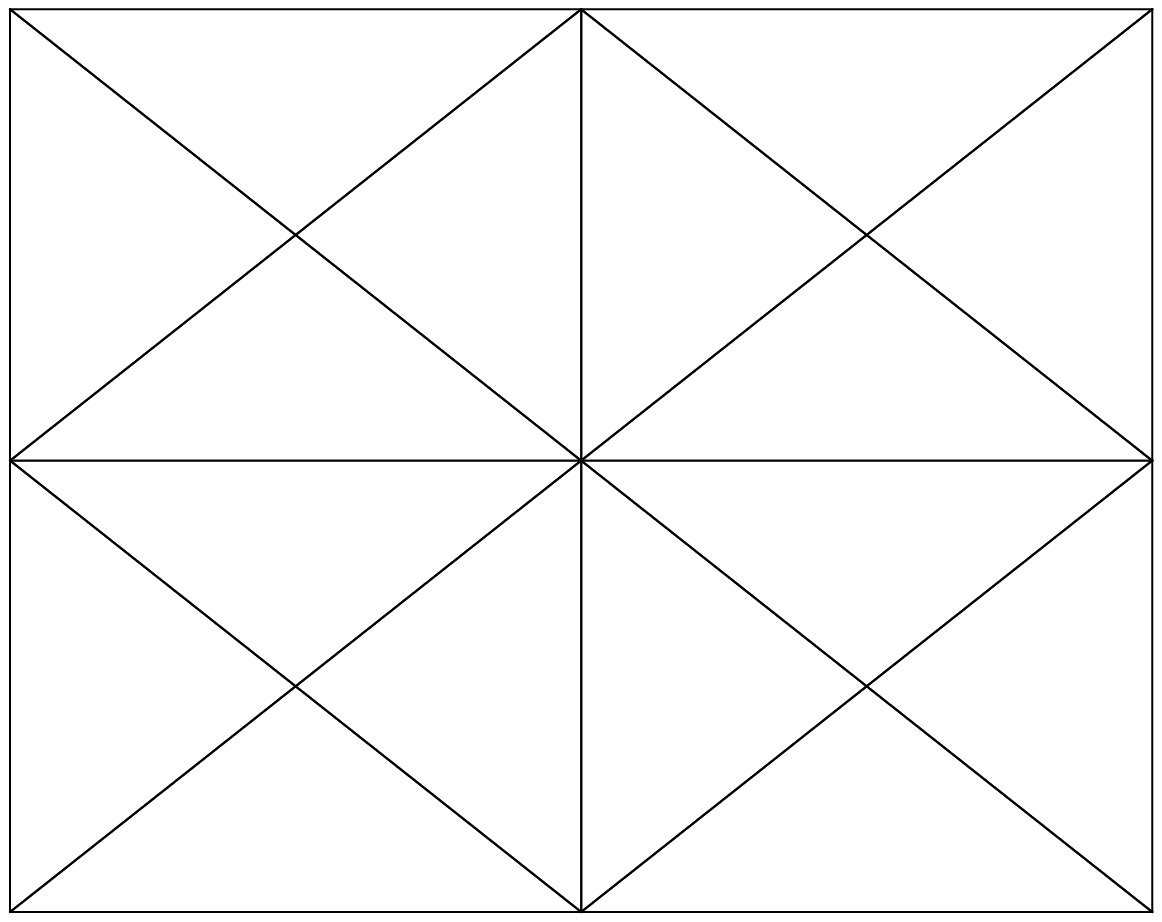} \\
\tiny{(A.1)}
\end{minipage}
\qquad
\begin{minipage}[b]{0.35\textwidth}\centering
\includegraphics[width=1.8cm,height=1.8cm,scale=0.66]{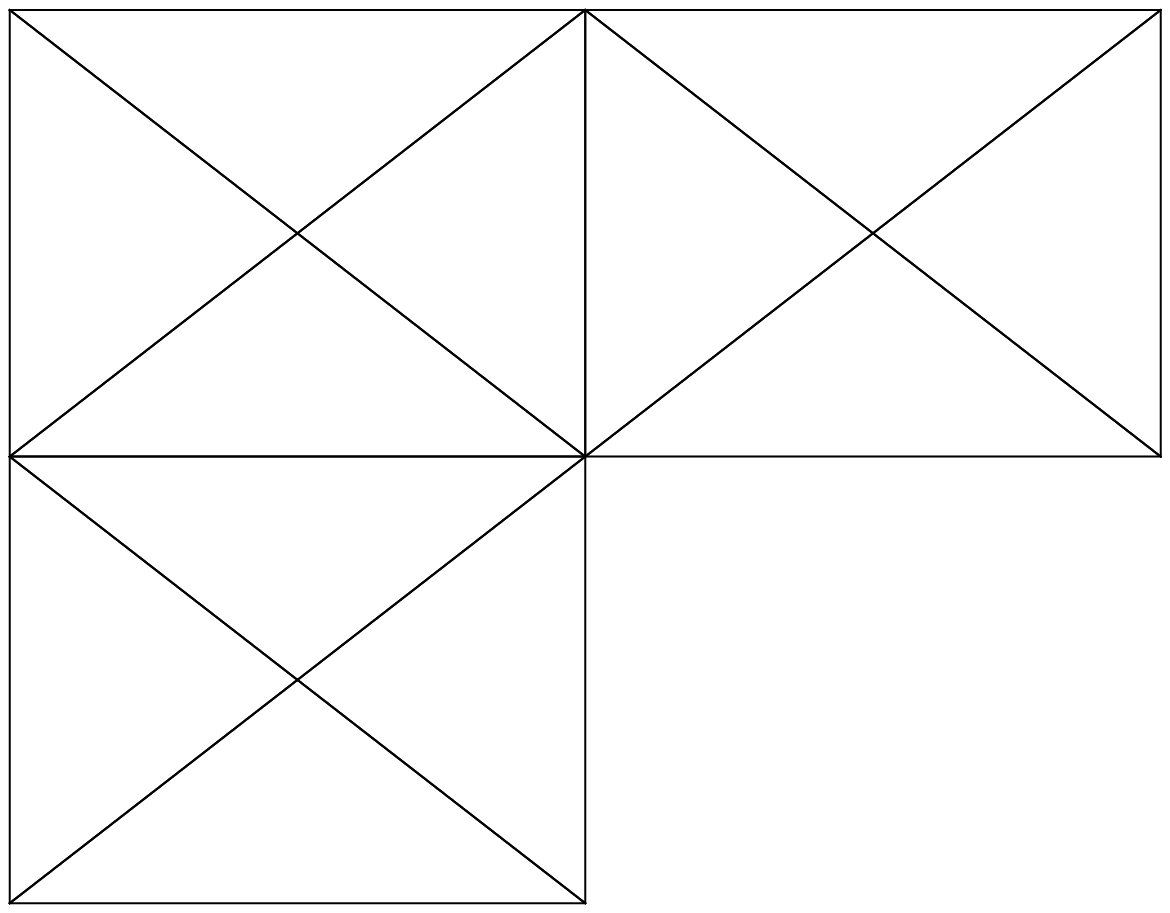} \\
\qquad \tiny{(A.2)}
\end{minipage}
\caption{The initial meshes used in the adaptive algorithm when the domain $\Omega$ is a square (Example 1) or a two-dimensional $L$-shaped (Example 2).}
\label{fig:mesh}
\end{figure}

In the numerical experiments that we perform we go beyond the presented theory and consider a series of Dirac delta sources on the right--hand side of the temperature equation. To be precise, we consider $\textnormal{g} = \sum_{z\in\mathcal{D}}\delta_z$. Here, $\mathcal{D}$ corresponds to a finite ordered subset of $\Omega$ with cardinality $\#\mathcal{D}$. Within this setting, we modify the error estimator $\mathpzc{E}_{p,\T}$, associated to the discretization of the heat equation, as follows:
\begin{equation}\label{eq:new_total_estimator_temp}
\mathpzc{E}_{p,\T}:=\left( \sum_{K\in\mathscr{T}} \mathpzc{E}_{p,K}^p\right)^{\frac{1}{p}},\qquad p<2,
\end{equation}
where, for each $K\in\mathscr{T}$, the local error indicators $\mathpzc{E}_{p,K}$ are given now as follows: if $z\in\mathcal{D}\cap K$ and $z$ is not a vertex of $K$, then
\begin{equation}
\label{eq:local_indicator_tempI_seriesdirac}
\mathpzc{E}_{p,K}:=
\left( 
\sum_{z\in\mathcal{D}\cap K}h_K^{2-p} 
+
h_K^p \| \mathpzc{R}_K
\|_{L^p(K)}^p 
+
h_K \| \mathpzc{J}_{\gamma}
\|_{L^p(\partial K \setminus \partial \Omega)}^p
\right)^{\frac{1}{p}}.
\end{equation}
If $z\in\mathcal{D}\cap K$ and $z$ is a vertex of $K$, then
\begin{equation}\label{eq:local_indicatorII_seriesdirac}
\mathpzc{E}_{p,K}:=
\left( 
h_K^p \| \mathpzc{R}_K  \|_{L^p(K)}^p 
+
h_K \|
\mathpzc{J}_{\gamma}
\|_{L^p(\partial K \setminus \partial \Omega)}^p
\right)^{\frac{1}{p}}.
\end{equation}
If $\mathcal{D}\cap K=\emptyset$, then the indicator $\mathpzc{E}_{p,K}$ is defined as in \eqref{eq:local_indicatorII_seriesdirac}. We notice that the previous modification is not needed if $\#\mathcal{D}=1$; \eqref{eq:new_total_estimator_temp} and \eqref{eq:error_estimator_heat} coincide.

\begin{algorithm}[ht]
\caption{\textbf{Iterative Scheme}}
\label{Algorithm1}
\textbf{Input:} Initial guess $(\mathbf{u}_{h}^{0},\mathsf{p}_{h}^{0},T_{h}^{0}) \in \mathbf{X}_h \times Q_h \times V_h$ and tol=$10^{-8}$;
\\
$\boldsymbol{1}$: For $i\geq 0$, find $(\mathbf{u}_{h}^{i+1}, \mathsf{p}_{h}^{i+1}) \in \mathbf{X}_{h}\times Q_h$ such that
\begin{equation*}
\begin{array}{rcll}
\displaystyle\int_\Omega (\nu(T_{h}^{i}) \mathbf{u}_{h}^{i+1}\cdot \mathbf{v}_h-  \mathsf{p}_{h}^{i+1} \textnormal{ div }\mathbf{v}_h)\mathsf{d}\mathbf{x}& = & 	\displaystyle\int_\Omega\mathbf{f}\cdot \mathbf{v}_h\mathsf{d}\mathbf{x} \quad &\forall {\mathbf{v}_h}\in \mathbf{X}_h,\\
\displaystyle\int_\Omega \mathsf{q}_h\textnormal{ div }\mathbf{u}_{h}^{i+1}\mathsf{d}\mathbf{x} & = & 0 \quad &\forall \mathsf{q}_h\in Q_h.
\end{array}
\end{equation*}
\\
Then, $T_{h}^{i+1} \in V_h$ is found as the solution to
\[
\displaystyle\int_\Omega(\kappa\nabla T_{h}^{i+1}\cdot \nabla S_h - T_{h}^{i+1} \mathbf{u}_{h}^{i+1}\cdot\nabla S_h)\mathsf{d}\mathbf{x}  =  \sum_{z\in\mathcal{D}}\langle \delta_{z}, S_h\rangle \quad \forall S_h\in V_h.
\]
$\boldsymbol{2}$: If $|(\mathbf{u}_{h}^{i+1}, \mathsf{p}_{h}^{i+1}, T_{h}^{i+1})-(\mathbf{u}_{h}^{i}, \mathsf{p}_{h}^{i}, T_{h}^{i} )|>$tol, set $i \leftarrow i + 1$ and go to step $\boldsymbol{1}$. Here, $|\cdot|$ denotes the Euclidean norm.
\end{algorithm}

\begin{algorithm}[ht]
\caption{\textbf{Adaptive Algorithm.}}
\label{Algorithm2}
\textbf{Input:} Initial mesh $\mathscr{T}_{0}$, finite subset $\mathcal{D} \subset \Omega$, viscosity coefficient $\nu$, thermal diffusivity $\kappa$, and external source $\mathbf{f}$;
\\
$\boldsymbol{1}$: Solve the discrete problem \eqref{eq:model_discrete} by using \textbf{Algorithm} \ref{Algorithm1};
\\
$\boldsymbol{2}$: For each $K\in\mathscr{T}_{i}$ compute the local error indicator $\mathsf{E}_K$ defined in \eqref{def:total_indicator};  
\\
$\boldsymbol{3}$: Mark an element $K\in\mathscr{T}_{i}$ for refinement if;
\begin{equation*}
\mathsf{E}_K>\dfrac{1}{2}\max_{K'\in \mathscr{T}_{i}} \mathsf{E}_{K'};
\end{equation*}
$\boldsymbol{4}$: From step $\boldsymbol{3}$, construct a new mesh $\mathscr{T}_{i+1}$, using a longest edge bisection algorithm. Set $i \leftarrow i + 1$ and go to step $\boldsymbol{1}$;
\end{algorithm}

We consider two problems with homogeneous Dirichlet boundary conditions whose exact solutions are not known. We finally mention that, in the numerical experiments that we perform, we violate the assumption that $\nu$ is piecewise polynomial; cf. Theorem \ref{thm:efficiencydarcy2}.

~\\
\textbf{Example 1:} We let $\Omega=(0,1)^2$, the thermal coefficient $\kappa=1$, the viscosity function $\nu(s):=\sin(s)+2$, the external density force $\mathbf{f}(x_{1},x_{2}):=\left(x_{1}x_{2}(1-x_{1})(1-x_{2}),0\right)$, and 
\begin{equation*}
\mathcal{D} = \{ (0.25,0.25), (0.25,0.75), (0.75,0.25), (0.75,0.75\}.
\end{equation*}
In Figure \ref{fig:test_01} we report the results obtained for Example 1. We present, for different values of the integrability index $p\in\{1.2, 1.4, 1.6, 1.8\}$, experimental rates of convergence for each contribution of the total error estimator, a finite element approximation of the temperature $T_h$, and an adaptively refined mesh. We observe, in subfigures (B.1)--(B.4), that our devised AFEM delivers optimal experimental rates of convergence for all the contributions of the total error estimator $\mathsf{E}_{\T}$ and for all the values considered for the integrability index $p$.  We also observe, in the adaptively refined mesh (B.6), that the adaptive refinement is mostly concentrated on the points where the Dirac measures are supported ($p=1.6$).

%%%%EXAMPLE 01
\begin{figure}[!ht]
\centering
\psfrag{graficap}{\huge $\mathpzc{E}_{p,\mathscr{T}}$}
\psfrag{graficaT}{\huge $\mathfrak{E}_{\mathscr{T}}$}
\psfrag{graficau}{\huge $\mathscr{E,\mathscr{T}}$}
\psfrag{Ndof(-1/2)}{\Large $\text{Ndof}^{-1/2}$}
\psfrag{pp1}{\Large $p=1.2$}
\psfrag{pp2}{\Large $p=1.4$}
\psfrag{pp3}{\Large $p=1.6$}
\psfrag{pp4}{\Large $p=1.8$}
\psfrag{Tp1}{\Large $p=1.2$}
\psfrag{Tp2}{\Large $p=1.4$}
\psfrag{Tp3}{\Large $p=1.6$}
\psfrag{Tp4}{\Large $p=1.8$}
\psfrag{up1}{\Large $p=1.2$}
\psfrag{up2}{\Large $p=1.4$}
\psfrag{up3}{\Large $p=1.6$}
\psfrag{up4}{\Large $p=1.8$}
\psfrag{etp1}{\Large $p=1.2$}
\psfrag{etp2}{\Large $p=1.4$}
\psfrag{etp3}{\Large $p=1.6$}
\psfrag{etp4}{\Large $p=1.8$}
\begin{minipage}[b]{0.327\textwidth}\centering
\scriptsize{\qquad Estimator $\mathpzc{E}_{p,\mathscr{T}}$}
\includegraphics[trim={0 0 0 0},clip,width=4.15cm,height=3.9cm,scale=0.66]{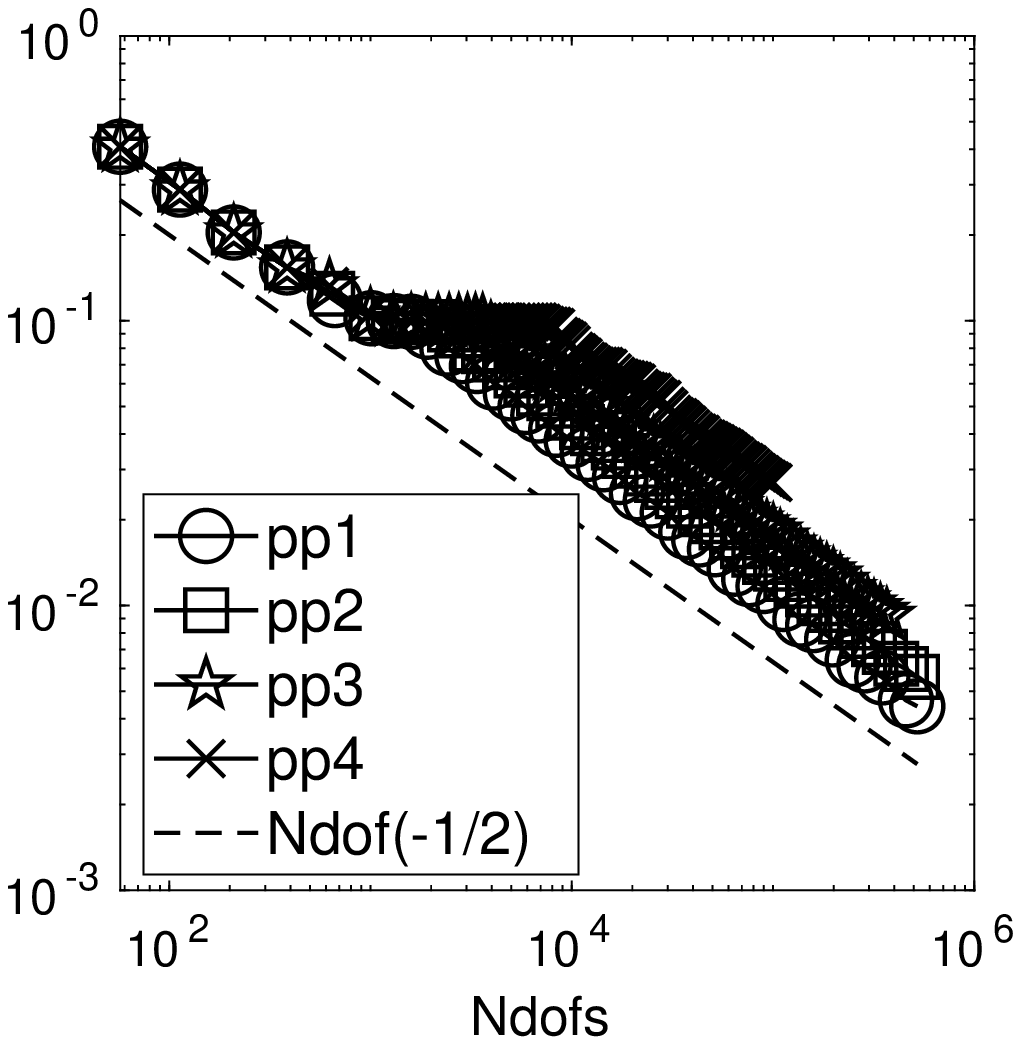} \\
\qquad \tiny{(B.1)}
\end{minipage}
\begin{minipage}[b]{0.327\textwidth}\centering
\scriptsize{\qquad Estimator $\mathfrak{E}_{\mathscr{T}}$}
\includegraphics[trim={0 0 0 0},clip,width=4.15cm,height=3.9cm,scale=0.66]{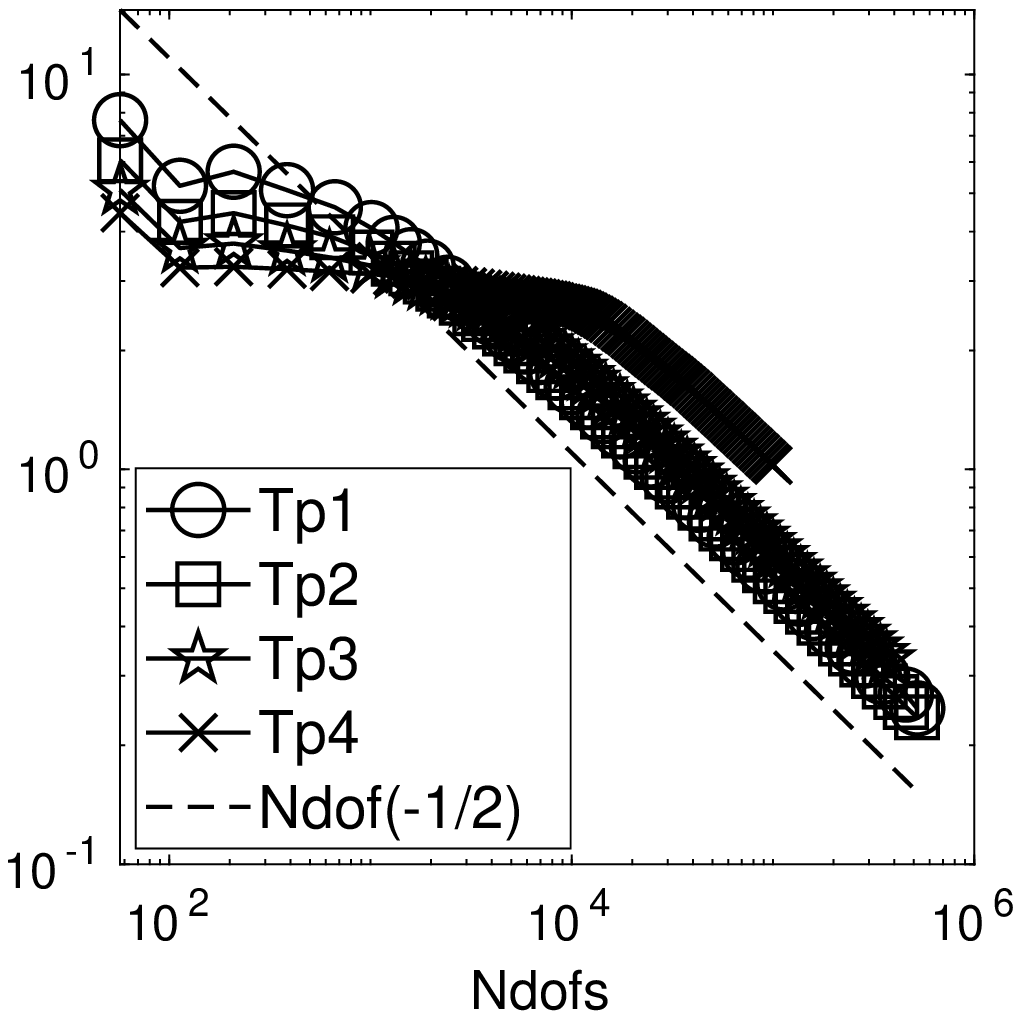} \\
\qquad \tiny{(B.2)}
\end{minipage}
\begin{minipage}[b]{0.327\textwidth}\centering
\scriptsize{\qquad Estimator $\mathscr{E}_{\mathscr{T}}$}
\includegraphics[trim={0 0 0 0},clip,width=4.15cm,height=3.9cm,scale=0.66]{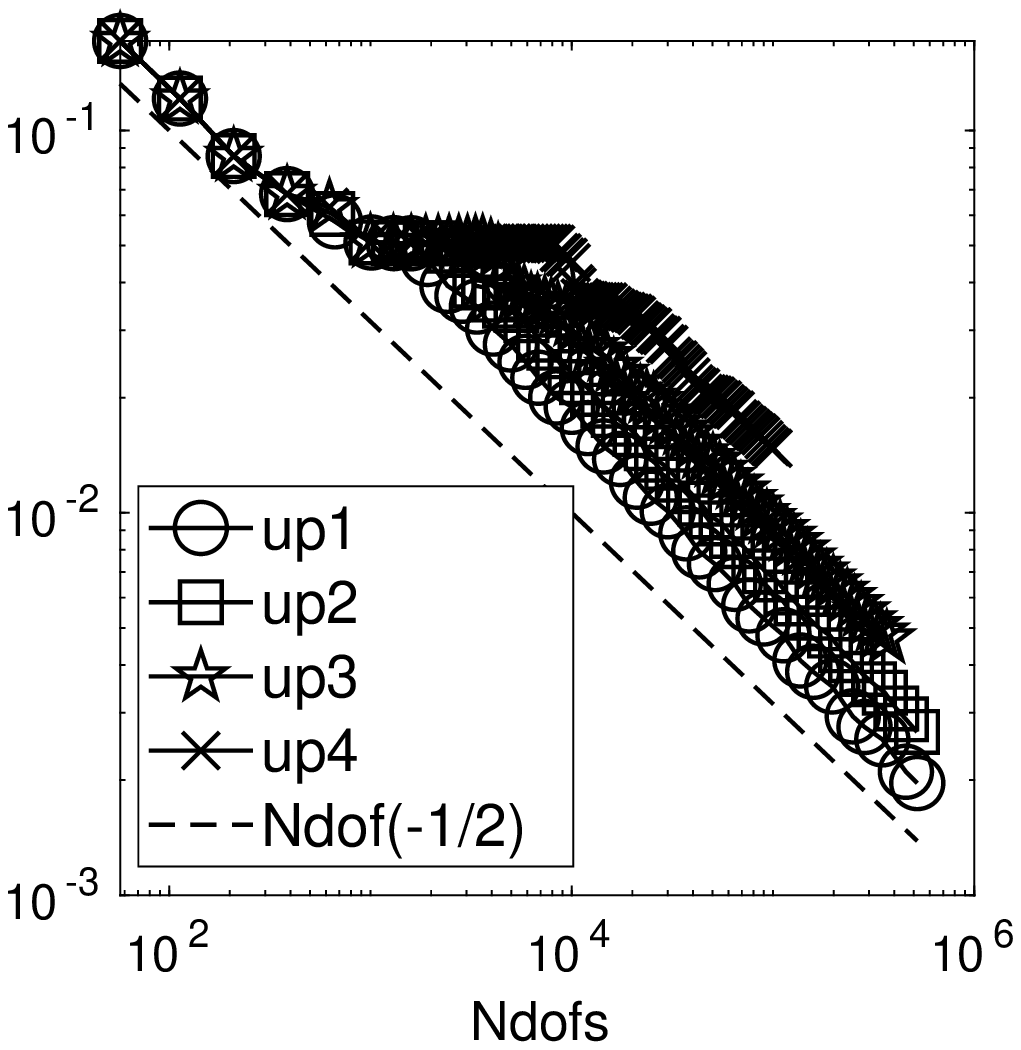} \\
\qquad \tiny{(B.3)}
\end{minipage}
\\~\\
\begin{minipage}[b]{0.327\textwidth}\centering
\scriptsize{\qquad Estimator $\mathsf{E}_{\T}$}
\includegraphics[trim={0 0 0 0},clip,width=4.15cm,height=3.9cm,scale=0.66]{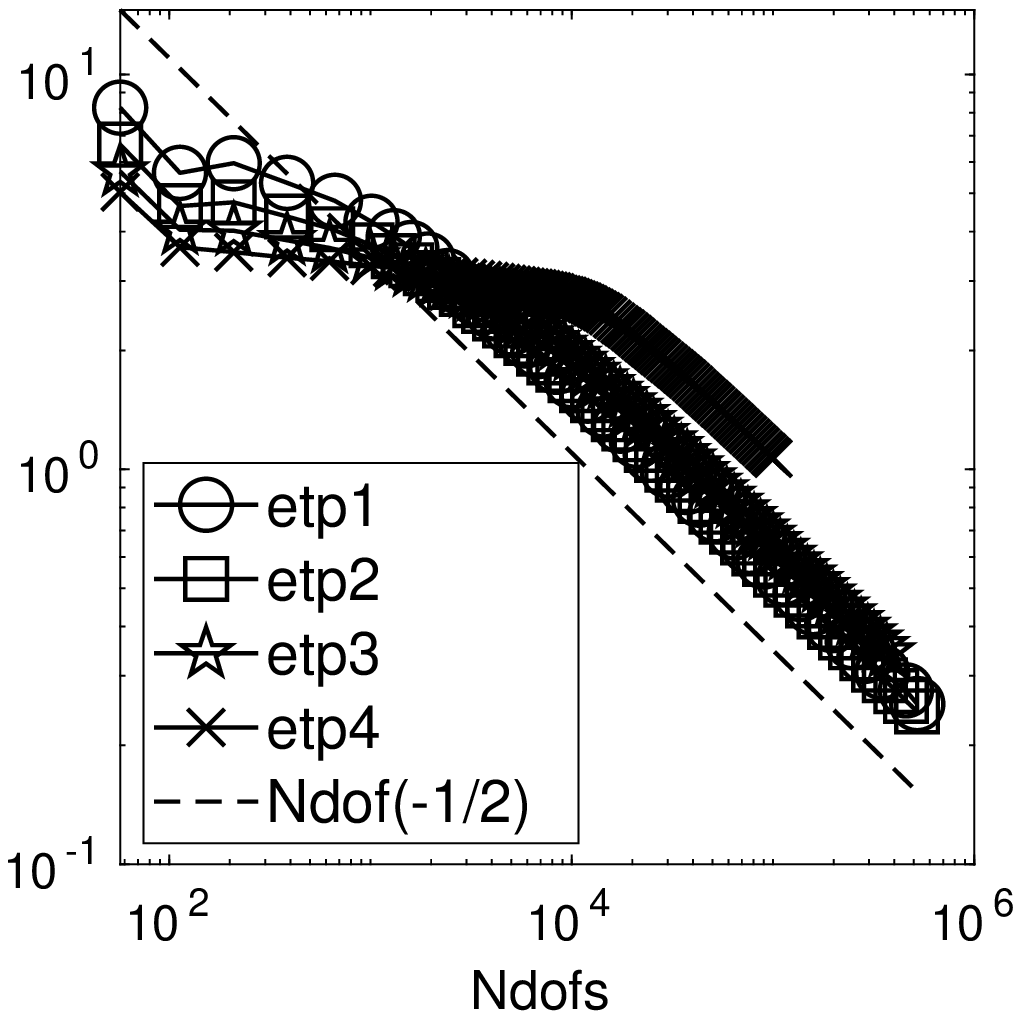} \\
\qquad \tiny{(B.4)}
\end{minipage}
\begin{minipage}[b]{0.327\textwidth}\centering 
\includegraphics[width=4.0cm,height=4.0cm,scale=0.66]{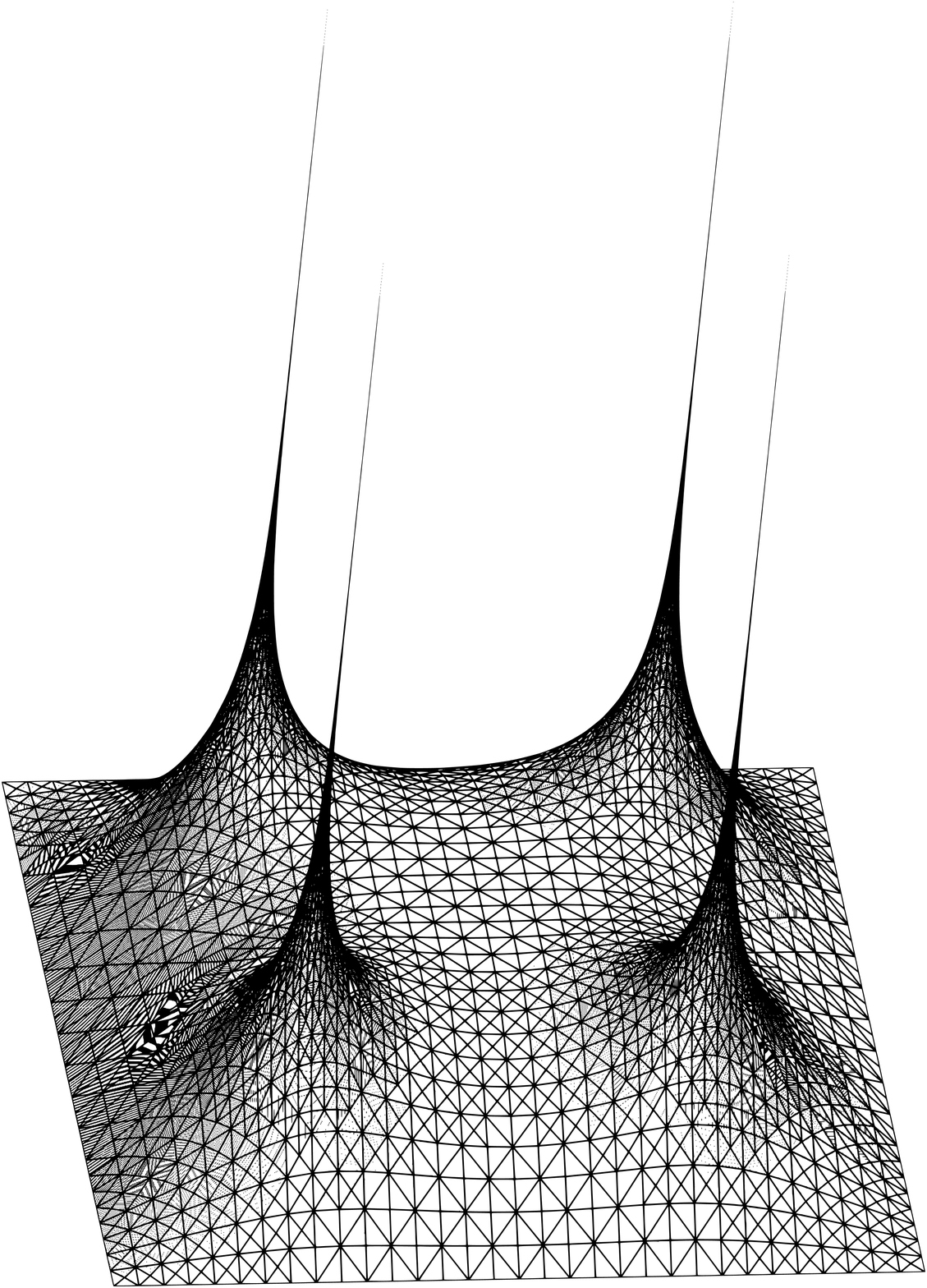} \\
\qquad \tiny{(B.5)}
\end{minipage}
\begin{minipage}[b]{0.327\textwidth}\centering  
\includegraphics[width=4.0cm,height=4.0cm,scale=0.66]{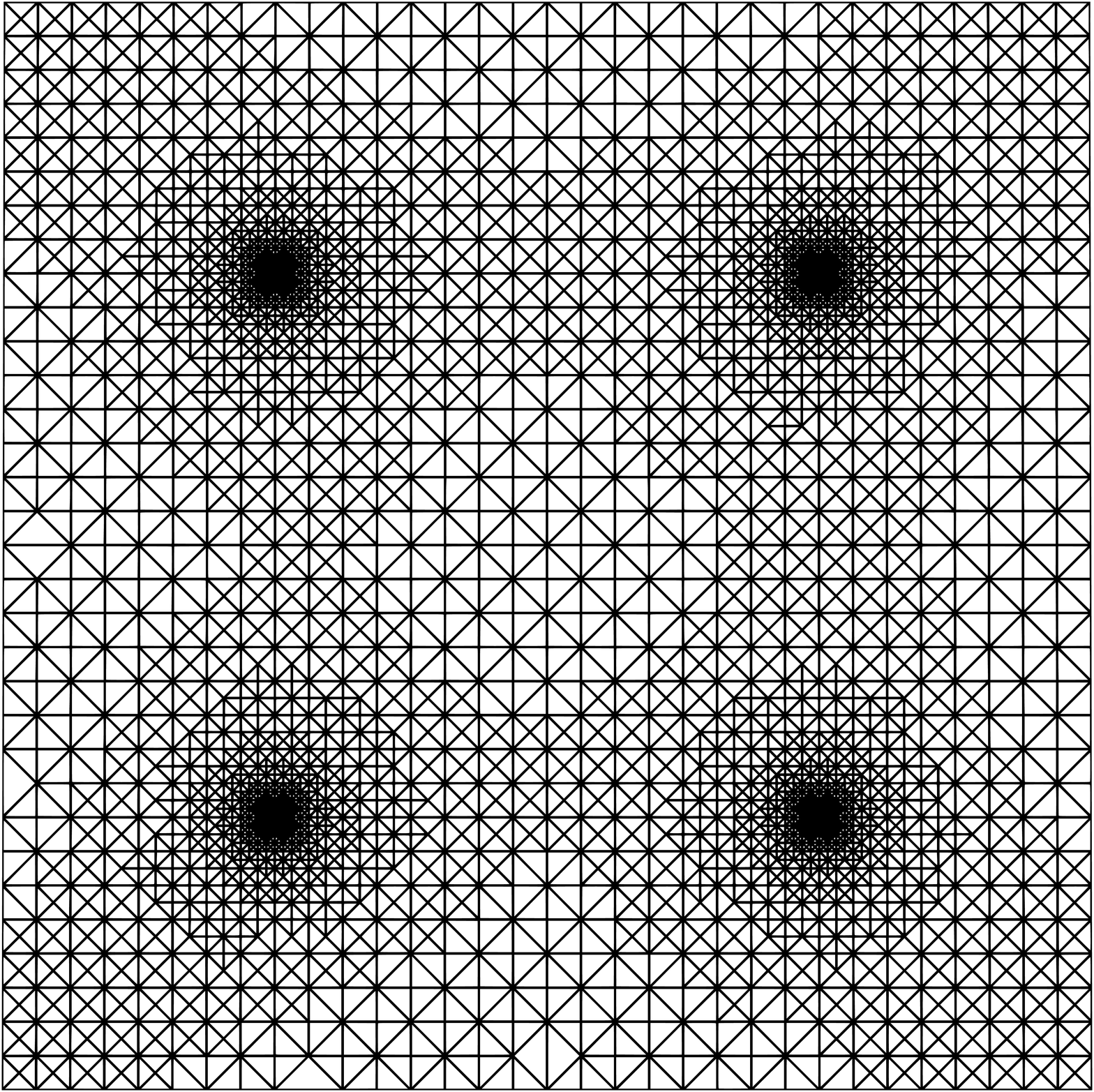} \\
\qquad \tiny{(B.6)}
\end{minipage}
\caption{Ex. 1: Experimental rates of convergence for the error estimators $\mathpzc{E}_{p,\mathscr{T}}$ (B.1), $\mathfrak{E}_{\mathscr{T}}$ (B.2), $\mathscr{E}_{\mathscr{T}}$ (B.3), and $\mathsf{E}_{\T}$ (B.4), for $p \in \{1.2, 1.4, 1.6, 1.8\}$, a finite element approximation of the temperature $T_h$ (B.5), and the mesh obtained after 29 iterations of the adaptive loop for $p = 1.6$ (B.6).}
\label{fig:test_01}
\end{figure}

~\\
\textbf{Example 2:} We let $\Omega=(-1,1)^2 \setminus[0,1)\times [-1,0)$, the thermal coefficient $\kappa=1$, the viscosity function $\nu(s):=e^{-s^2}+1$, the external force $\mathbf{f}(x_{1},x_{2}):=\big(10x_{2}(1-x_{1})(1+x_{1}),5x_{1}(1-x_{2})(1+x_{1})\big)$, and $\mathcal{D} = \{(-0.5,-0.5), (-0.5,0.5), (0.5,0.5)\}$.

In Figure \ref{fig:test_02} we report the results obtained for Example 2. Similar conclusions to the ones presented for Example 1 can be derived. In particular, we observe optimal experimental rates of convergence for all the individual contributions of the total error estimator $\mathsf{E}_{\T}$ and for all  the values considered for the integrability index $p$ (C.1)--(C.4). We also observe, in subfigure (C.6), that the adaptive refinement is mostly concentrated on the points where the Dirac measures are supported and near to the region of the domain that involves a geometric singularity ($p=1.6$).

%%%%EXAMPLE 02

\begin{figure}[!ht]
\centering
\psfrag{graficap}{\huge $\mathpzc{E}_{p}$}
\psfrag{graficaT}{\huge $\mathfrak{E}_{\mathscr{T}}$}
\psfrag{graficau}{\huge $\mathscr{E}$}
\psfrag{Ndof(-1/2)}{\Large $\text{Ndof}^{-1/2}$}
\psfrag{pp1}{\Large $p=1.2$}
\psfrag{pp2}{\Large $p=1.4$}
\psfrag{pp3}{\Large $p=1.6$}
\psfrag{pp4}{\Large $p=1.8$}
\psfrag{Tp1}{\Large $p=1.2$}
\psfrag{Tp2}{\Large $p=1.4$}
\psfrag{Tp3}{\Large $p=1.6$}
\psfrag{Tp4}{\Large $p=1.8$}
\psfrag{up1}{\Large $p=1.2$}
\psfrag{up2}{\Large $p=1.4$}
\psfrag{up3}{\Large $p=1.6$}
\psfrag{up4}{\Large $p=1.8$}
\psfrag{etp1}{\Large $p=1.2$}
\psfrag{etp2}{\Large $p=1.4$}
\psfrag{etp3}{\Large $p=1.6$}
\psfrag{etp4}{\Large $p=1.8$}
\begin{minipage}[b]{0.327\textwidth}\centering
\scriptsize{\qquad Estimator $\mathpzc{E}_{p,\mathscr{T}}$}
\includegraphics[trim={0 0 0 0},clip,width=4.15cm,height=3.9cm,scale=0.66]{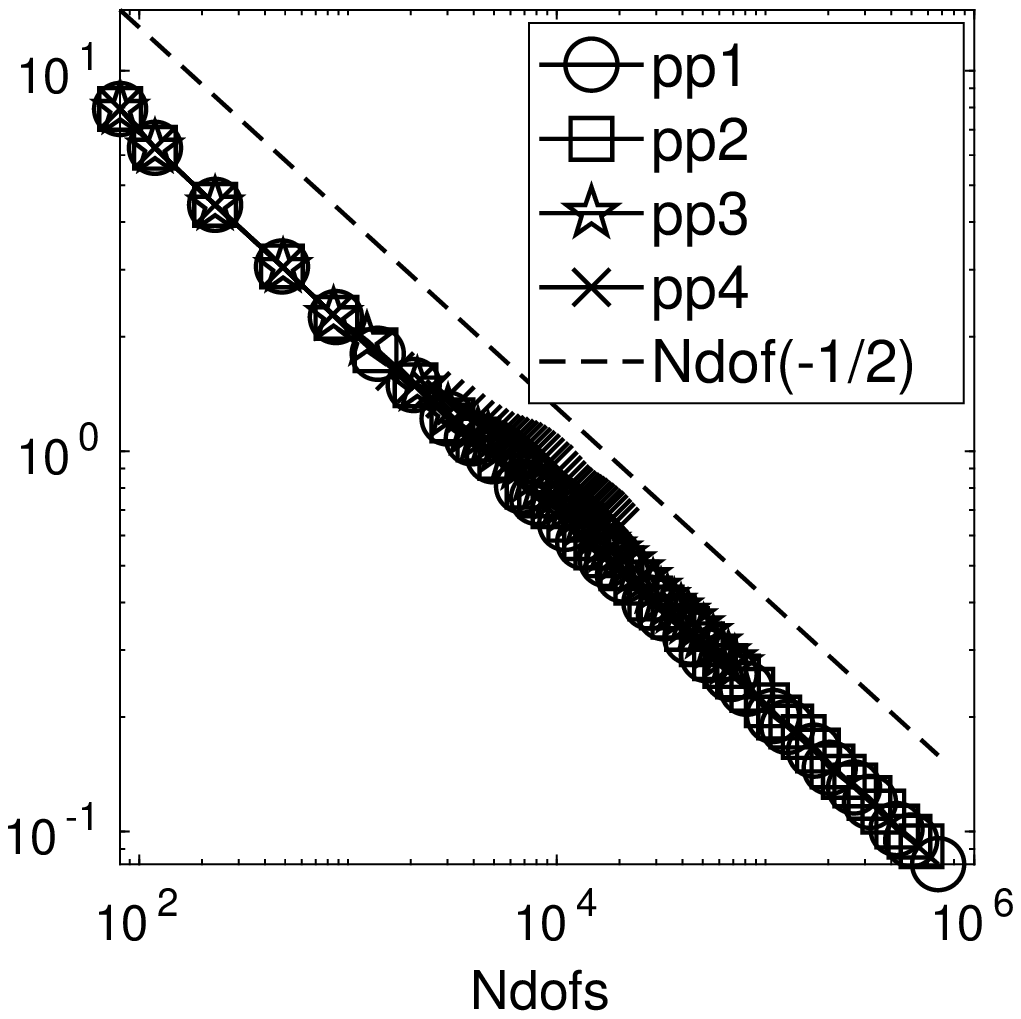} \\
\qquad \tiny{(C.1)}
\end{minipage}
\begin{minipage}[b]{0.327\textwidth}\centering
\scriptsize{\qquad Estimator $\mathfrak{E}_{\mathscr{T}}$}
\includegraphics[trim={0 0 0 0},clip,width=4.15cm,height=3.9cm,scale=0.66]{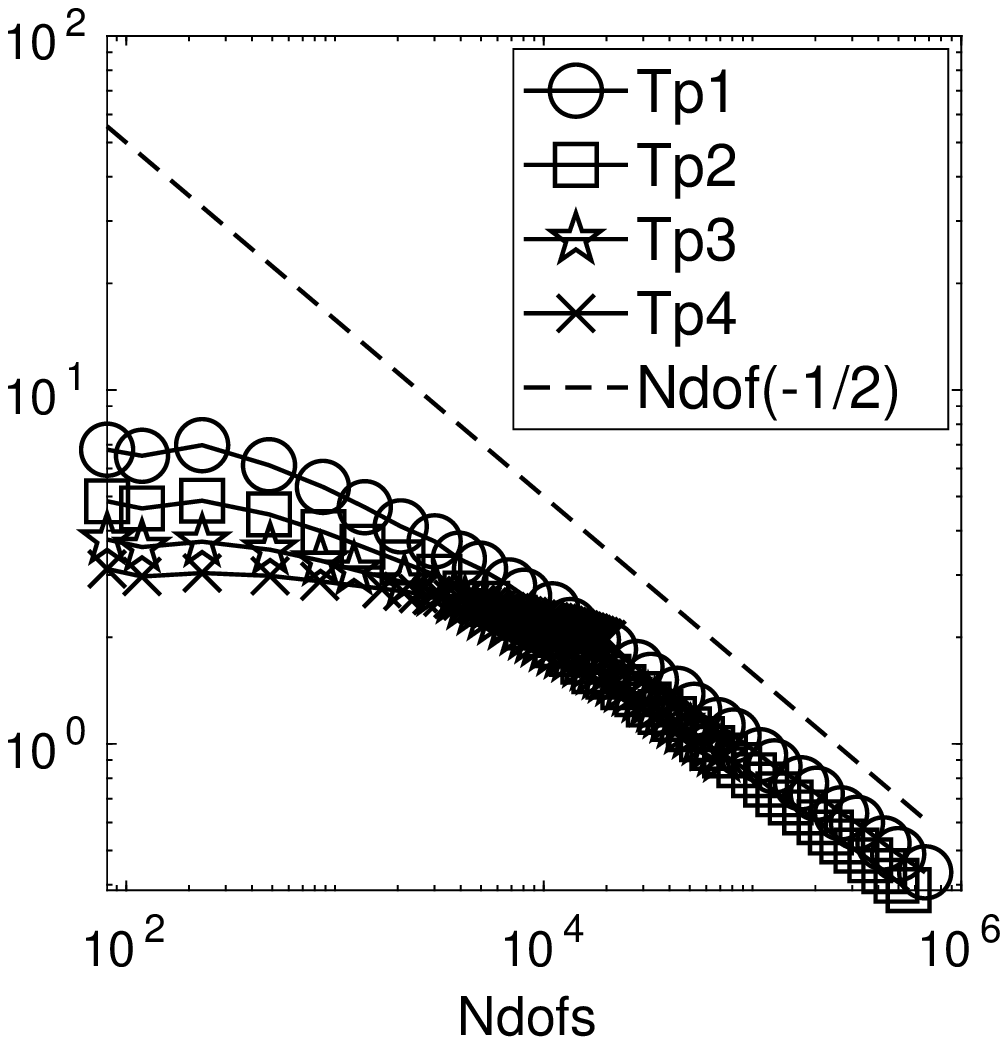} \\
\qquad \tiny{(C.2)}
\end{minipage}
\begin{minipage}[b]{0.327\textwidth}\centering
\scriptsize{\qquad Estimator $\mathscr{E}_{\mathscr{T}}$}
\includegraphics[trim={0 0 0 0},clip,width=4.15cm,height=3.9cm,scale=0.66]{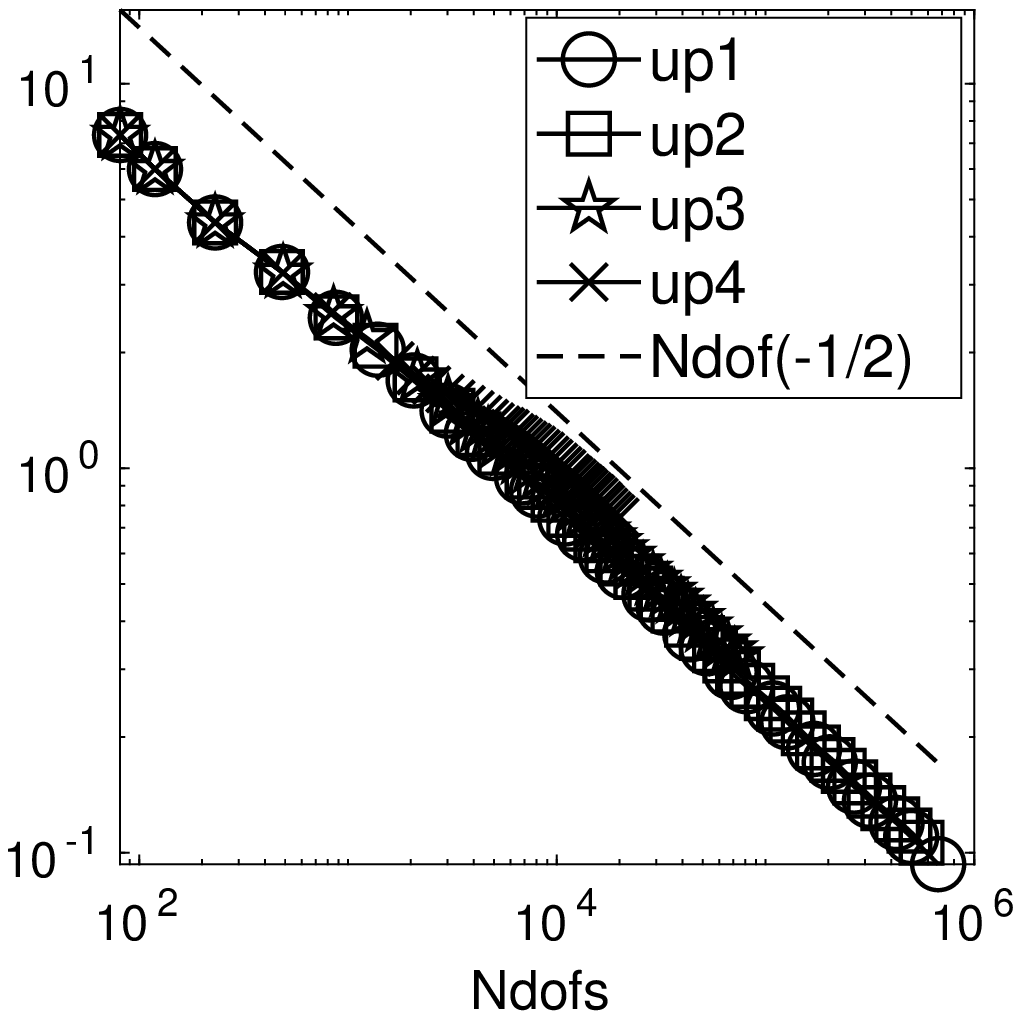} \\
\qquad \tiny{(C.3)}
\end{minipage}
\\~\\
\begin{minipage}[b]{0.327\textwidth}\centering
\scriptsize{\qquad Estimator $\mathsf{E}_{\T}$}
\includegraphics[trim={0 0 0 0},clip,width=4.15cm,height=3.9cm,scale=0.66]{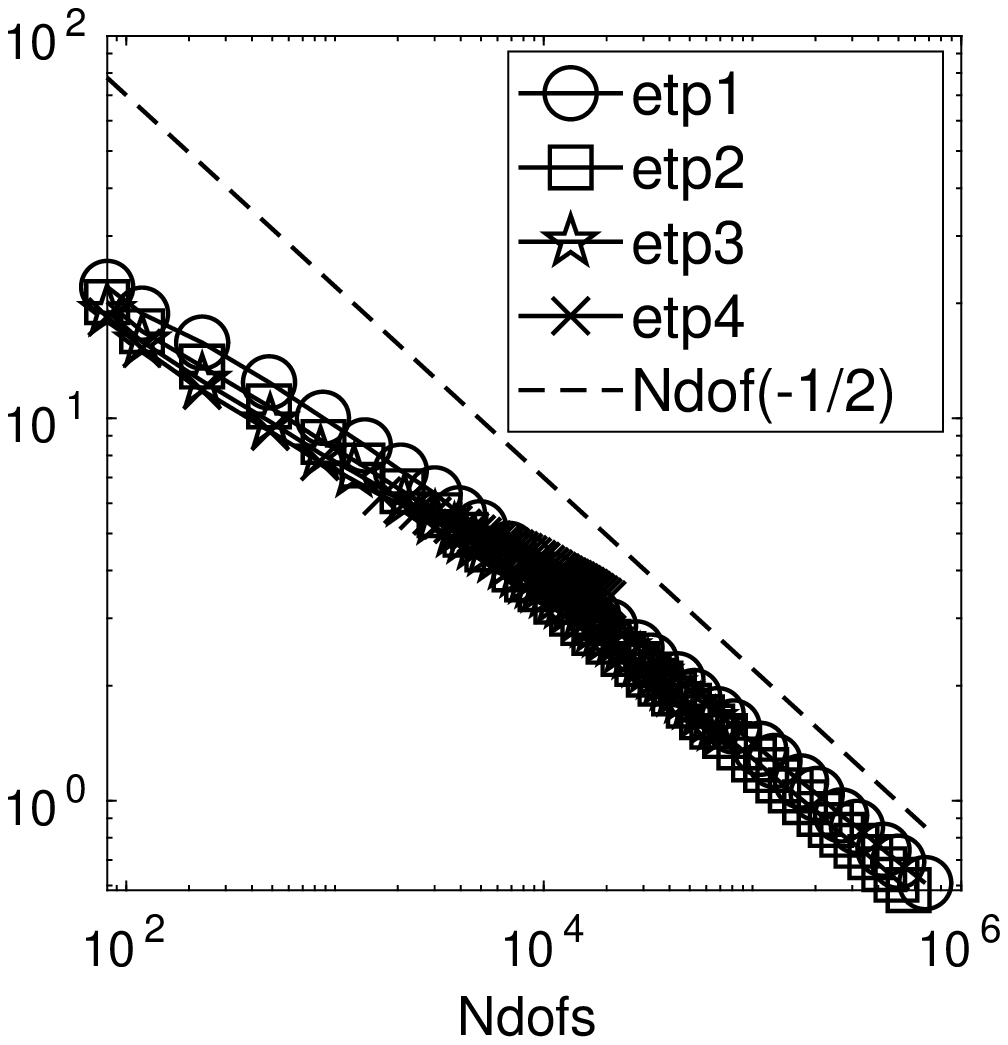} \\
\qquad \tiny{(C.4)}
\end{minipage}
\begin{minipage}[b]{0.327\textwidth}\centering 
\includegraphics[width=4.0cm,height=4.0cm,scale=0.66]{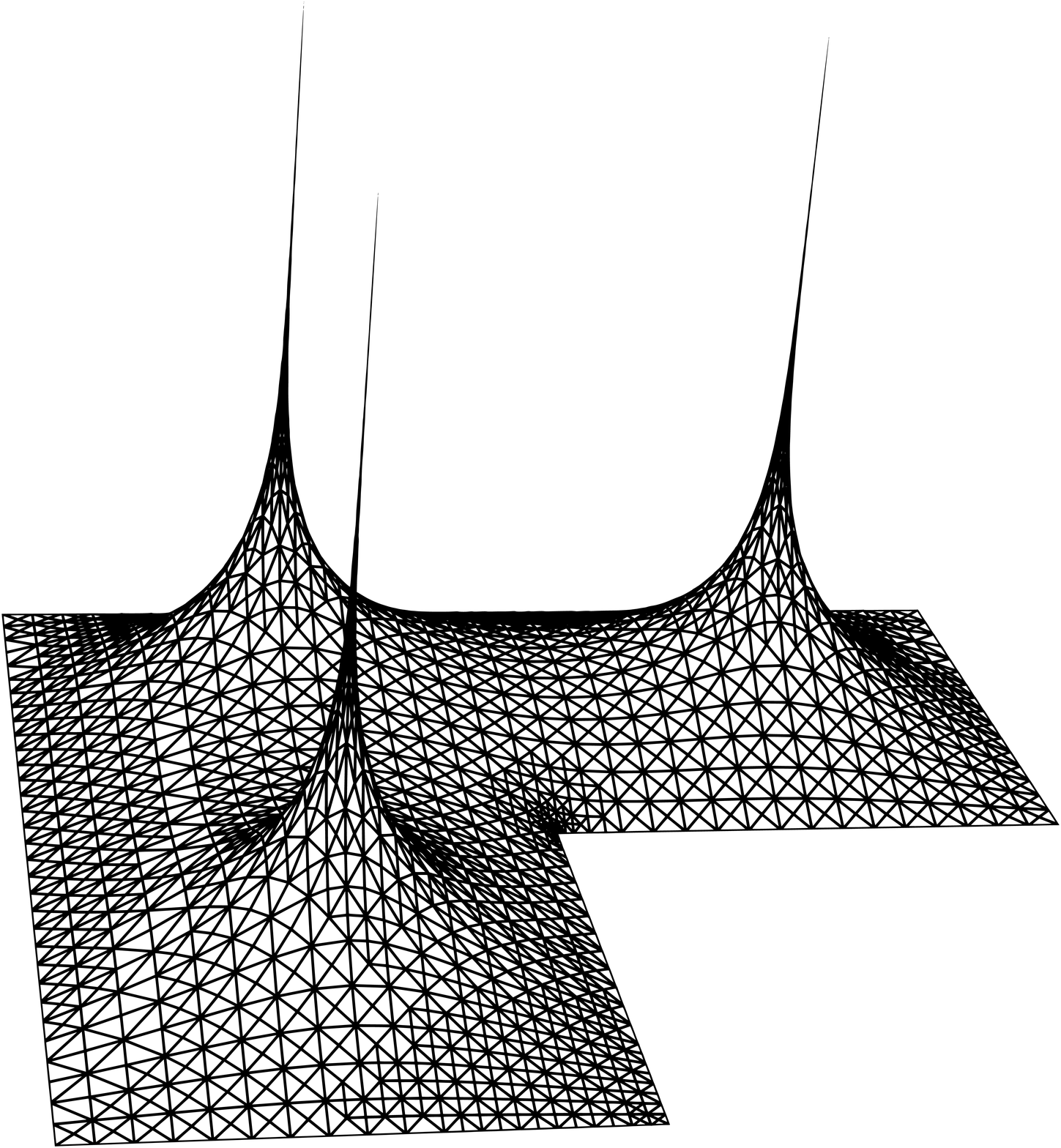} \\
\qquad \tiny{(C.5)}
\end{minipage}
\begin{minipage}[b]{0.327\textwidth}\centering  
\includegraphics[width=4.0cm,height=4.0cm,scale=0.66]{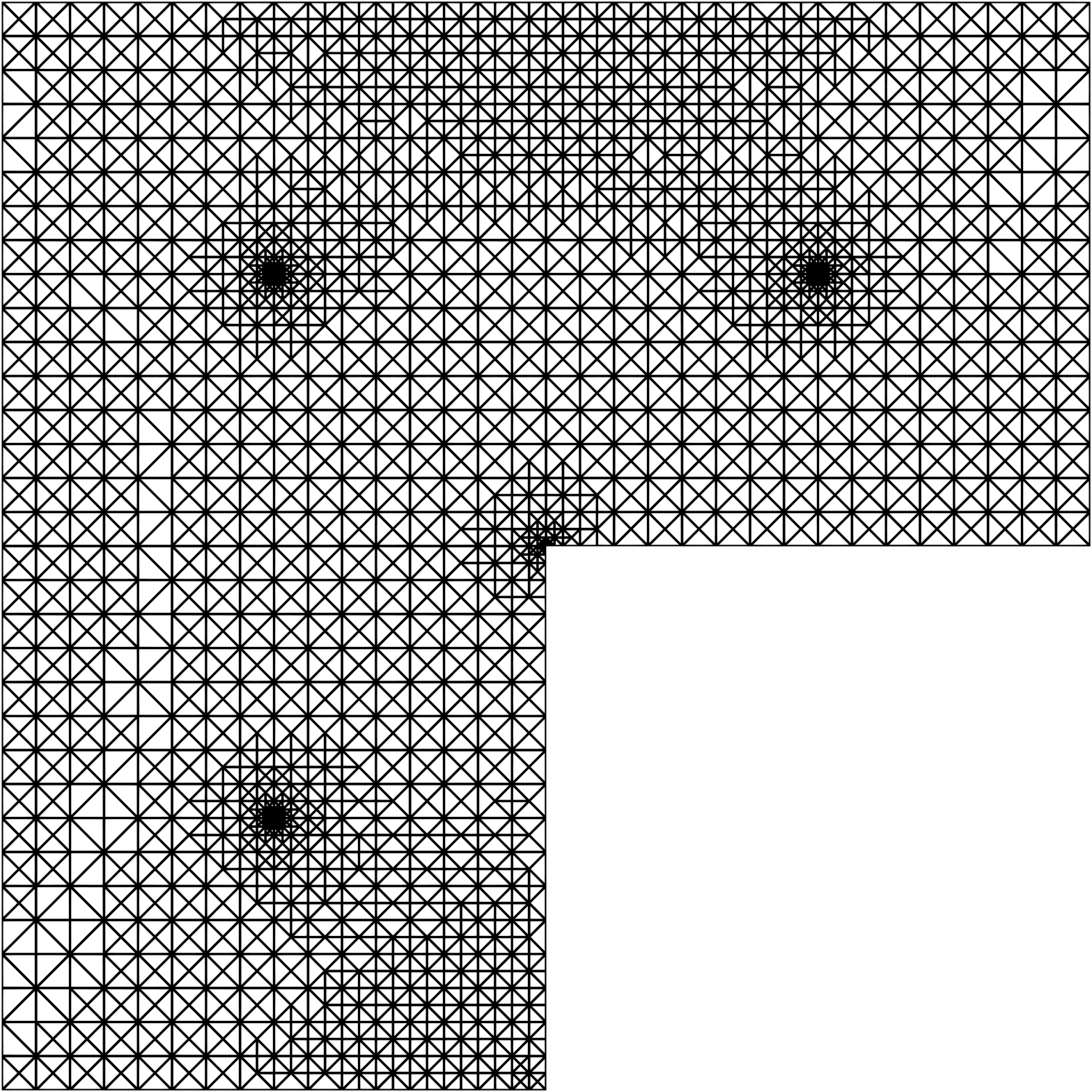} \\
\qquad \tiny{(C.6)}
\end{minipage}
\caption{Ex. 2: Experimental rates of convergence for the error estimators $\mathpzc{E}_{p,\mathscr{T}}$ (C.1), $\mathfrak{E}_{\mathscr{T}}$ (C.2), $\mathscr{E}_{\mathscr{T}}$ (C.3), and $\mathsf{E}_{\T}$ (C.4), for $p \in \{1.2, 1.4, 1.6, 1.8\}$, a finite element approximation of the temperature $T_h$ (C.5), and the mesh obtained after 23 iterations of the adaptive loop for $p = 1.6$ (C.6).}
\label{fig:test_02}
\end{figure}

\bibliographystyle{siamplain}
\bibliography{biblio}
\end{document}